\documentclass[11pt]{amsart}
\usepackage{fullpage,xcolor,graphicx}
\usepackage[OT2,T1]{fontenc}
\usepackage{hyperref}
\usepackage{pifont}
\usepackage{amsthm}
\usepackage{amsfonts}
\usepackage{amssymb}
\usepackage[mathscr]{euscript}
\usepackage[all]{xy}
\usepackage{amsmath}
\usepackage{epsfig}
\usepackage{latexsym}
\usepackage{stackengine}
\usepackage{multirow}
\usepackage{hhline}
\usepackage{MnSymbol}
\usepackage{bm}

\usepackage[OT2,OT1]{fontenc} \newcommand\cyr
{
\renewcommand\rmdefault{wncyr} \renewcommand\sfdefault{wncyss} \renewcommand\encodingdefault{OT2} \normalfont
\selectfont
}
\DeclareTextFontCommand{\textcyr}{\cyr} 

\newcommand*\wbar[1]{
  \hbox{ \kern-0.2em%
    \vbox{%
      \hrule height 0.5pt  
      \kern0.25ex
      \hbox{%
        \kern-0.15em
        \ensuremath{#1}%
        \kern-0.05em
      }%
    }%
  \kern0.05em}%
} 

\newcommand*\wbarnew[1]{
  \hbox{ \kern-0.2em%
    \vbox{%
      \hrule height 0.5pt  
      \kern0.25ex
      \hbox{%
        \kern-0.35em
        \ensuremath{#1}%
        \kern-0.05em
      }%
    }%
  \kern0.05em}%
}

\newcommand{\Zh}{\hbox{\hspace{-5.8mm} \textcyr{Zh}}} 
\newcommand{\zh}{\hbox{\hspace{-5.8mm} \textcyr{zh}}} 

\makeatletter
\DeclareRobustCommand\widecheck[1]{{\mathpalette\@widecheck{#1}}}
\def\@widecheck#1#2{%
    \setbox\z@\hbox{\m@th$#1#2$}%
    \setbox\tw@\hbox{\m@th$#1%
       \widehat{%
          \vrule\@width\z@\@height\ht\z@
          \vrule\@height\z@\@width\wd\z@}$}%
    \dp\tw@-\ht\z@
    \@tempdima\ht\z@ \advance\@tempdima2\ht\tw@ \divide\@tempdima\thr@@
    \setbox\tw@\hbox{%
       \raise\@tempdima\hbox{\scalebox{1}[-1]{\lower\@tempdima\box
\tw@}}}%
    {\ooalign{\box\tw@ \cr \box\z@}}}
\makeatother

\newtheorem{theorem}{Theorem}[subsection]
\newtheorem{lemma}[theorem]{Lemma}
\newtheorem{proposition}[theorem]{Proposition}

\theoremstyle{definition}
\newtheorem{definition}[theorem]{Definition}
\newtheorem{example}[theorem]{Example}

\newtheorem*{theorem*}{Theorem}
\theoremstyle{remark}
\newtheorem{remark}[theorem]{Remark}

\newcommand{\vlk}{\operatorname{{\it v}\ell{\it k}}}

\begin{document}


\title[Lie superalgebras and the minimal genus]{$U_q(\mathfrak{gl}(m|n))$ bounds on the minimal genus of virtual links}

\author{Micah Chrisman}
\address{Department of Mathematics, The Ohio State University at Marion, Marion, Ohio, USA}
\email{chrisman.76@osu.edu}
\thanks{}

\author{Killian Davis}
\address{Department of Mathematics, The Ohio State University, Columbus, Ohio, USA}
\email{davis.7030@buckeyemail.osu.edu}

\author{Anup Poudel}
\address{Department of Mathematics, The Ohio State University, Columbus, Ohio, USA}
\email{poudel.33@osu.edu}
\thanks{}


\maketitle

\begin{abstract} For links $L \subset \Sigma \times [0,1]$, where $\Sigma$ is a closed orientable surface, we define a $U_q(\mathfrak{gl}(1|1))$ Reshetikhin-Turaev invariant with coefficients in $\mathbb{Z}[H_1(\Sigma)]$. This invariant turns out to be equivalent to an infinite cyclic version of the Carter-Silver-Williams (CSW) polynomial. The importance of the CSW polynomial is that half its symplectic rank gives strong lower bounds on the virtual genus. Recall that the virtual genus of a virtual link $J$ is the smallest genus of all closed orientable surfaces $\Sigma$ on which $J$ can be represented by a link diagram  on $\Sigma$. Here we generalize the CSW lower bound to all quantum supergroups $U_q(\mathfrak{gl}(m|n))$ with $m,n>0$. For $(m,n)=(1,1)$, the $U_q(\mathfrak{gl}(m|n))$ bound is the same as the CSW bound. However, changing the value of the pair $(m,n)$ can give lower bounds better than those available from other known methods. We compare the $U_q(\mathfrak{gl}(m|n))$ lower bounds to those coming from the CSW polynomial, the surface bracket, the arrow polynomial, hyperbolicity, and the Gordon-Litherland determinant test. As a first application, we show that the Seifert genus of homologically trivial knots in thickened surfaces is not additive under the connected sum operation of virtual knots. As a second application, we prove that the Jaeger-Kauffman-Saleur invariant of a virtual knot is always realizable as the Alexander polynomial of an infinite cyclic cover of a knot complement in some $\Sigma \times [0,1]$, but is not always so on a surface of minimal genus. This is accomplished with a generalization of the $\Zh$-construction, called the homotopy $\Zh$-construction. 

\end{abstract}

\section{Introduction}
\subsection{Motivation} Lie superalgebra quantum invariants of links $L$ in $S^3$ have recently been shown to have useful topological applications. Consider, for example, the generalized Links-Gould invariant $LG_L^{m,1}(t_0,t_1)$, which is constructed from representations of $U_q(\mathfrak{gl}(m|1))$. This invariant is closely related to the minimal genus problem of Seifert surfaces for links in $S^3$. A lower bound on the 3-genus of $L$ can always be obtained from $LG_L^{m,1}(t_0,t_1)$, since various specializations of $LG_L^{m,1}(t_0,t_1)$ recover the Alexander polynomial $\Delta_L(t)$ (De Wit--Ishii--Links \cite{de_wit_ishii_links}, Kohli--Patureau-Mirand \cite{kohli_patureau_mirand}).  Remarkably, however, the degree of $LG^{2,1}_L(t_0,t_1)$ yields an even better lower bound on the $3$-genus of knots than the Alexander polynomial alone (Kohli--Tahar \cite{kohli_tahar}). Further applications of $LG_L^{2,1}(t_0,t_1)$ to fibers of fibered links were found by Lopez-Neumann and van der Veen \cite{lopez_neumann_van_der_veen,lopez_neumann_van_der_veen_fibred}. 

Here we will apply Lie superalgebras to a minimal genus problem that arises in virtual knot theory. Virtual knot theory, instituted by Kauffman \cite{kauffman_vkt}, gives a diagrammatic way of   studying links in thickened surfaces $\Sigma \times [0,1]$, where $\Sigma$ is closed and oriented. Two such links are \emph{virtually isotopic} if they are related by a finite sequence of ambient isotopies, diffeomorphisms of $\Sigma$, and stabilizations/destabilizations. These last two terms mean the addition/removal of $1$-handles to/from $\Sigma$ that are disjoint from a diagram of the link. An equivalence class is called a \emph{virtual link type}. Since different representatives of a virtual link type $L$ can lie on surfaces $\Sigma$ of different genus, the fundamental geometric problem is to determine the smallest allowable genus (Kuperberg \cite{kuperberg}). By analogy with ``$n$-genus'' nomenclature for classical links, we call this the \emph{virtual 2-genus}, $\widecheck{g}_2(L)$.  

In practice, one attempts to determine $\widecheck{g}_2(L)$ as follows. If the minimal classical crossing number of $L$ is known, it is easy to calculate the $\widecheck{g}_2(L)$ from a Gauss diagram of a minimal crossing representative (Manturov \cite{manturov_projection}, Boden--Rushworth \cite{boden_rushworth}). Otherwise, there is a complete invariant for computing $\widecheck{g}_2(L)$.  Carter, Silver, and Williams (CSW) showed that the virtual 2-genus is exactly half the symplectic rank of the operator group of a link in $\Sigma \times [0,1]$ \cite{CSW}. Unfortunately, this is not computable in general. Useful lower bounds on $\widecheck{g}_2(L)$, however, can be obtained from an Alexander-like polynomial which can be calculated from any operator group presentation. If the classical crossing number is unknown and the CSW bound is too low, one can then resort to a grab bag of possible techniques. There are methods available from skein theory, such as the surface bracket \cite{dye_kauffman_surf} and the Dye--Kauffman--Miyazawa (i.e.\! arrow) polynomial \cite{dye_kauffman_arrow,miyazawa}. Specialized tools can be effective for special classes of knots. The Gordon--Litherland determinant test can be applied to checkerboard-colorable knots (Boden--Chrisman--Karimi \cite{BCK}).  If a link in $\Sigma \times [0,1]$ is tg-hyperbolic, then it necessarily minimizes the virtual $2$-genus (Adams et al. \cite{adams_tg}). For each of the known computable methods, examples exist where the virtual $2$-genus is not detected.     

In this paper, we derive lower bounds on $\widecheck{g}_2(L)$ from $U_q(\mathfrak{gl}(m|n))$ Reshetikin-Turaev invariants, where $m,n>0$. The motivation behind this arises from the following observations. The computable methods listed above can be loosely categorized as ``Jones-type'' (surface bracket, arrow polynomial) or ``Alexander-type'' (CSW polynomial, Gordon-Litherland determinants). On the other hand, for classical links and the vector representation of $U_q(\mathfrak{gl}(m|n))$, the $U_q(\mathfrak{gl}(n+2|n))$ invariant is the Jones polynomial and the $U_q(\mathfrak{gl}(n|n))$ invariant is the Alexander polynomial (see e.g. Queffelec \cite{queffelec_19}). Then, in some loose sense, there are $U_q(\mathfrak{gl}(m|n))$ bounds on $\widecheck{g}_2(L)$ if $m-n=0$ or $|m-n|=2$. Interpolating and extrapolating from this, it is reasonable to expect that lower bounds on $\widecheck{g}_2(L)$ can be also be obtained from $U_q(\mathfrak{gl}(m|n))$ invariants when $|m-n|=1$ or $|m-n|>2$. The strategy of the paper, then, is to find improved bounds on $\widecheck{g}_2(L)$ by filling in all range values for pairs $(m,n)$.    

When $m=n=1$, we show our $U_q(\mathfrak{gl}(m|n))$ bound coincides with that of the CSW polynomial. When $n \ge 1$ and $m \ne 1$, the $U_q(\mathfrak{gl}(m|n))$ lower bound can differ from the CSW bound. Examples will be given where $\widecheck{g}_2(L)$ is not realized by the $U_q(\mathfrak{gl}(1|1))$ lower bound but is realized for other pairs $(m,n)$. In particular, we will show that there are knots whose virtual $2$-genus is not determined by the arrow polynomial, the Gordon--Litherland determinants, nor the CSW polynomial, but whose virtual 2-genus is determined by the $U_q(\mathfrak{gl}(2|1))$ lower bound. Relations between the $U_q(\mathfrak{gl}(m|n))$ polynomials and the surface bracket are also discussed. Furthermore, we use the $U_q(\mathfrak{gl}(1|1))$ bound to show $\widecheck{g}_2(K_r)=r+1$ for the infinite family $K_r$ of generalized $r$-Kishino knots, recovering a result previously proved by Adams et al. \cite{adams_tg} using tg-hyperbolicity. As a first topological application, we use the $U_q(\mathfrak{gl}(2|1))$ lower bound to demonstrate that the Seifert genus of homologically trivial knots in a thickened surface is not additive under the connected sum operation of virtual knots. As a second topological application, we prove that the Jaeger--Kauffman--Saleur (JKS) invariant of a virtual knot is always realizable as the Alexander polynomial of an infinite cyclic cover of a knot complement in some $\Sigma \times [0,1]$, but that this cannot always be done on a surface of minimal genus.

\subsection{Main idea} \label{sec_main_idea} As the CSW polynomial is already an Alexander polynomial, this is the natural starting point for the development of quantum superalgebra invariants. We first recall some basic facts about its construction. A detailed review is given in Section \ref{sec_csw}. Let $D$ be diagram of a link $L \subset \Sigma \times [0,1]$, where $\Sigma$ is a closed orientable surface $\Sigma$ of genus $g>0$. Choose a symplectic basis $\Omega=\{x_1,y_1,\ldots,x_g,y_g\}$ of $H_1(\Sigma;\mathbb{Z})$. From the universal cover of $\Sigma$, Carter, Silver, and Williams define a polynomial $\Delta_D^0(t,x_1,y_1,\ldots,x_g,y_g)$. Relative to a fixed symplectic basis, this is an invariant of links in the thickened surface $\Sigma \times [0,1]$. The polynomial itself is dependent on the choice of symplectic basis. A lower bound on $\widecheck{g}_2(L)$ can be extracted from the symplectic rank of this polynomial, denoted $\text{rk}_s(\Delta_D^0)$. This is the rank of the submodule of $H_1(\Sigma;\mathbb{R})$ generated by the nonzero coefficients of powers of $t$ in $\Delta_D^0(t,x_1,y_1,\ldots,x_g,y_g)$, modulo its submodule of self-orthogonal elements with respect to the intersection pairing on $\Sigma$. The main result of \cite{CSW} is that:
\begin{equation}\label{eqn_csw_bound}
\text{rk}_s(\Delta_D^0) \le 2 \cdot \widecheck{g}_2(D).
\end{equation}

The difficulty in generalizing (\ref{eqn_csw_bound}) to quantum invariants is the typical local-global problem. Quantum knot invariants are locally defined whereas homology is a global property. Homology, however, can be localized using the intersection form on $\Sigma$ and the intersection form can be calculated with respect to the basis $\Omega$. Thus, one only needs to account for the interaction between $D$ and $\Omega$. Now, the action of $\Gamma=\pi_1(\Sigma)$ on the universal cover of $\Sigma$ is by conjugation.  But the Wirtinger relations which appear in $\pi_1(\Sigma \times [0,1] \smallsetminus L)$ are also defined by conjugation. The main idea of the paper is to model the action of $\Gamma$ on $\Sigma$ by adding some components to the virtual link type of $L$, with a supplemental ``colored'' component for each element of $\Omega$. The advantage of these bulked-up virtual link diagrams is that they can be decomposed into elementary tangles, so that the $U_q(\mathfrak{gl}(m|n))$ invariants can be defined combinatorially with the vector representation. In essence, the local-global problem is resolved with careful bookkeeping. As will be seen, the bookkeeping method is a straightforward generalization of the well-known Bar-Natan $\Zh$-construction \cite{bar_natan_talk}. Lower bounds on the virtual $2$-genus are then be obtained by following the method of Carter-Silver-Williams \cite{CSW}.

\subsection{Main result} Suppose that $D$ is a diagram of a link on a closed orientable surface $\Sigma$. Let $\Omega=\{x_1,y_1,\ldots,x_g,y_g\}$ be a choice of symplectic basis for $\Sigma$. To $D$ we associate an $l+2g$ component virtual link diagram $\Zh_{\Omega}(D)$, where $l$ is the number of components of $D$. Such diagrams are called \emph{prismatic links} and are considered equivalent up to a set of moves that includes virtual link equivalence and some additional moves that model $\pi_1(\Sigma)$. The map $D \to \Zh_{\Omega}(D)$ is called the \emph{homotopy\! $\Zh$-construction}. This is the first part of the bookkeeping method described in the preceding paragraph. The second part of the bookkeeping method consists of a $U_q(\mathfrak{gl}(m|n))$ Reshetikhin-Turaev functor. Its domain of definition is a certain category of prismatic tangles, so that a prismatic link is a morphism $\varnothing \to \varnothing$. Composing the \emph{prismatic $U_q(\mathfrak{gl}(m|n))$ Reshetikhin-Turaev functor} with the homotopy $\Zh$-construction for $D$ then gives a polynomial $\widetilde{f}^{\,m|n}_{(\Sigma,\Omega,D)}(q,x_1,y_1,\ldots,x_g,y_g)$. As in the case of the CSW polynomial, this depends on the choice of symplectic basis. However, as in \cite{CSW}, the symplectic rank $\text{rk}_s(\widetilde{f}^{\,\,m|n}_{(\Sigma,\Omega,D)})$ gives a lower bound on the virtual 2-genus. More precisely:

\begin{theorem*} If $D$ is a diagram on $\Sigma$ of a non-split link in $\Sigma \times [0,1]$ and $L$ is its virtual link type,
\begin{equation} \label{eqn_main_theorem}
\text{rk}_s\left(\widetilde{f}^{\,\,m|n}_{(\Sigma,\Omega,D)}\right) \le 2 \cdot \widecheck{g}_2(L).
\end{equation}
\end{theorem*}

This theorem, which is the main result of the paper, is proved below as Theorem \ref{thm_main}.

\subsection{Application: the CSW and JKS polynomials} When $m=n=1$, we show that the CSW polynomial and the prismatic $U_q(\mathfrak{gl}(1|1))$ Reshetikhin-Turaev polynomial $\widetilde{f}^{\,\,1|1}_{(\Sigma,\Omega,D)}$ are equivalent after a change of variables, a change of basis, and up to multiplication by units. One advantage of the quantum model is computational. The rank of the vector representation of  $U_q(\mathfrak{gl}(1|1))$ is only two, so that computer calculations are typically fast. Our \emph{Mathematica} code can be found \href{https://github.com/micah-chrisman/virtual-knots/blob/main/prismatic_zh_final_switch.nb}{here}. 

A second advantage of the quantum model is that it gives a relation with the first and most influential of all virtual knot invariants: the two-variable JKS invariant $G_L(s,t)$. This was introduced by Kauffman and Saleur in 1992 \cite{kauffman_saleur_92}, and further studied in a paper with Jaeger \cite{jaeger_kauffman_saleur_94} in 1994. Although its original motivation came from statistical mechanics, alternative definitions due to Sawollek \cite{sawollek_01} and Silver--Williams \cite{silver_williams_03} have unlocked an unexpected topological significance. For example, $G_L(s,t)$ obstructs the existence of a homologically trivial representative of $D$ of $L$ in a thickened surface (Silver-Williams \cite{silver_williams_06_II}), it is a virtual slice obstruction (Boden--Chrisman \cite{boden_chrisman_21}), and it unifies several combinatorial invariants (Mellor \cite{mellor_16}). Many have contributed to the further development of the JKS invariant (see e.g. Boden et al. \cite{BDGGHN_15}, Crans--Henrich--Nelson \cite{crans_henrich_nelson}, Manturov \cite{manturov_02}).  In the literature, it is often called the \emph{generalized Alexander polynomial} (GAP). 

In Theorem \ref{thm_gap_from_prism}, we prove that for any virtual link type $L$, there is a diagram $D$ on a surface $\Sigma$ and a choice of symplectic basis $\Omega$, such that $G_L(s,t)$ is a specialization of the CSW polynomial $\Delta_D^0(t,x_1,y_1,\ldots,x_g,y_g)$, after a change of variables and up to multiplication by units. The GAP is thus realizable as a honest-to-goodness Alexander polynomial, that is, as the Alexander polynomial associated to an infinite cyclic cover of the link complement in $\Sigma \times [0,1]$. Furthermore, we show that the realizing surface $\Sigma$ cannot always be chosen so that its genus equals $\widecheck{g}_2(L)$ (Section \ref{sec_v2g_GAP}). 



\subsection{Application: the Dye--Kauffman surface bracket} In Section \ref{sec_homotopy_and_surface}, we  briefly consider skein-theoretic methods for determining the virtual $2$-genus. The prototypical such method is the surface bracket, due to Dye and Kauffman \cite{dye_kauffman_surf}. There are many skein-theoretic polynomials for virtual knots and link diagrams on surfaces (see e.g. Boden--Karimi--Sikora \cite{boden_karimi_sikora}, Boninger \cite{boninger}, Dye--Kauffman \cite{dye_kauffman_arrow}, Manturov \cite{manturov_surf}, Miller \cite{miller}). As our intention is only to establish a natural connection between two seemingly disparate methods, here we will focus solely on the surface bracket itself. In Section \ref{sec_surf}, we define a notion of symplectic rank for the surface bracket and prove that the Dye--Kauffman condition for minimality is equivalent to the maximality of this symplectic rank. It is then proved in Section \ref{sec_surf_factors} that the surface bracket with $\mathbb{Z}_2$ homological coefficients factors through the homotopy $\Zh$-construction. Thus, the homotopy $\Zh$-construction and the symplectic rank unify the ``Alexander-type'' and ``Jones-type'' methods for determining $\widecheck{g}_2(L)$.

\subsection{Application: non-additivity of the virtual 3-genus} Recall that the \emph{3-genus} $g_3(K)$ of a classical knot $K$ in $S^3$ is the smallest genus among all Seifert surfaces which are bounded by the knot. As is well-known, the 3-genus is additive under connected sum: $g_3(K_1\# K_2)=g_3(K_1)+g_3(K_2)$. A virtual knot is said to be \emph{almost classical} if it has a homologically trivial (with $\mathbb{Z}$ coefficients) representative in some thickened surface. The \emph{virtual 3-genus} $\widecheck{g}_3(K)$ of an almost classical knot $K$ is the smallest genus among all Seifert surfaces bounded by homologically trivial representatives of $K$ in thickened surfaces. Here we will show that the virtual $3$-genus is not additive under connected sum of virtual knots. In Section \ref{sec_v3g}, we give an example of a connect sum of unknots $K_1$, $K_2$ such that $K_1 \# K_2$ is almost classical and $\widecheck{g}_3(K_1 \# K_2)=1$. Since  the only knot with virtual $3$-genus $0$ is the unknot, it suffices to show that $\widecheck{g}_2(K_1 \#K_2) \ne 0$. For our example, the CSW polynomial of $K_1\#K_2$ is identically 0 and the Gordon-Litherland determinant test fails. Furthermore, since the diagram of $K_1 \# K_2$ has twenty classical crossings, computation of skein theoretic invariants such as the surface bracket are not feasible. The $U_q (\mathfrak{gl}(2|1))$ lower bound, however, gives the correct value of $\widecheck{g}_2(K_1 \#K_2)=2$. Hence, $1=\widecheck{g}_3(K_1 \# K_2) > \widecheck{g}_3(K_1)+\widecheck{g}_3(K_2)=0+0=0$.

\subsection{Examples $\&$ Calculations} To compare the $U_q(\mathfrak{gl}(m|n))$ bounds to other methods, further calculations are given in Section 9. There is an example of an almost classical satellite knot having CSW bound $0$ and arrow polynomial bound $0$, but whose $U_q(\mathfrak{gl}(2|1))$ correctly determines the virtual $2$-genus. There is also an example of a non-alternating checkerboard colorable knot such that the Gordon--Litherland determinant test fails, but whose $U_q(\mathfrak{gl}(1|1))$ and $U_q(\mathfrak{gl}(2|1))$ both recover the virtual $2$-genus. It is furthermore shown that the $U_q(\mathfrak{gl}(1|1))$ bound is sharp for an arbitrarily large virtual $2$-genus. This is done by showing that the $U_q(\mathfrak{gl}(1|1))$ bound is maximal for the infinite hyperbolic family of generalized $r$-Kishino knots. All together, these examples demonstrate the the $U_q(\mathfrak{gl}(m|n))$ bounds are generally effective at determining the minimal genus, even in cases where the minimal crossing number is either unknown or not easily inferred.


\subsection{Organization}  We begin in Section \ref{sec_csw} with a detailed account of the CSW polynomial. This will provide the necessary tools to relate the prismatic $U_q(\mathfrak{gl}(1|1))$ polynomial to the CSW polynomial. In Section \ref{sec_homotopy_zh}, we define prismatic links and the homotopy $\Zh$-construction and relate it to the operator group of Carter--Silver--Williams \cite{CSW}. In Section \ref{sec_uqglmn_rt}, we generalize to prismatic tangles and introduce the prismatic $U_q(\mathfrak{gl}(m|n))$ functor $\widetilde{Q}^{m|n}_{\Omega}$. The prismatic $U_q(\mathfrak{gl}(1|1))$ polynomial and the CSW polynomial are identified in Section \ref{sec_uqgl11_csw}. The proof of the main theorem appears in Section \ref{sec_gen_bounds}. Relations to the GAP are discussed in Section \ref{sec_GAP_and_homology_zh}. The surface bracket is similarly factored through the homotopy $\Zh$-construction in Section \ref{sec_homotopy_and_surface}.  Examples and applications, such as the non-additivity of the virtual 3-genus, are collected together in Section \ref{sec_applications}. 

\section{The infinite cyclic CSW polynomial} \label{sec_csw}


In this paper, we will use an infinite cyclic version of CSW polynomial. This section gives both a topological and computational perspective on its construction. This is needed to connect it to the quantum model ahead. Originally, the CSW polynomial was defined using operator groups. This terminology is reviewed in Section \ref{sec_operator}. A standard reference is Robinson \cite{robinson}. In Section \ref{sec_covering}, we define the covering group $\widetilde{\pi}_L$ of a link $L$ in a thickened surface as a finitely generated and presented operator group. The infinite cyclic cover which identifies the CSW polynomial as an Alexander polynomial is given explicitly in Section \ref{sec_infinite_cyclic}. Using this, we define the infinite cyclic CSW polynomial in Section \ref{sec_csw_poly}. The symplectic rank is introduced in Section \ref{sec_symp_rank}. A detailed example is worked out in Section \ref{sec_csw_example}.

\subsection{Review: Operator groups} \label{sec_operator} An \emph{operator group} is a triple $(\pi,\Gamma,\alpha)$, where $\pi$ is a group, $\Gamma$ is a set, and $\alpha: \pi \times \Gamma \to \pi$, $\alpha(a,\gamma)=a^\gamma$, is a function such that $\alpha(-,\gamma)$ is a group homomorphism for all $\gamma \in \Gamma$. When $\Gamma$ is a group, we assume that: 
\[
(\alpha^{\gamma})^{\eta}=\alpha^{\gamma \eta}.
\]
Observe that this is a \emph{right action}. A \emph{$\Gamma$-subgroup} of $(\pi,\Gamma,\alpha)$ is a subgroup $H$ of $\pi$ such that $h^{\gamma} \in H$ for all $h \in H$ and $\gamma \in \Gamma$. Given a normal $\Gamma$-subgroup of $\pi$, the quotient $\pi/N$ is also an operator group $(\pi/N,\Gamma,\bar{\alpha})$ where $\bar{\alpha}(Ng,\gamma)=Ng^{\gamma} $. A operator group will be called \emph{finitely generated (as an operator group)} if there are elements $g_1,\ldots,g_r$ of $\pi$ such that for every $g \in \pi$, $g=g_{i_1}^{\gamma_1}\cdot \ldots \cdot g_{i_n}^{\gamma_n}$ for $g_{i_j} \in \{g_1,\ldots,g_r \}$ and $\gamma_i \in \Gamma$. Likewise, an operator group is \emph{finitely presented} if there is a finite subset $R$ of its set of relations $S$ such that the orbit of the action of $\Gamma$ on $R$ is all of $S$. 

A \emph{homomorphism} of operator groups $(\pi_a,\Gamma_a,\alpha_a)$, $(\pi_b,\Gamma_b,\alpha_b)$, where $\Gamma_a,\Gamma_b$ are groups, is a pair of group homomorphisms $(f,\phi):(\pi_a,\Gamma_a) \to (\pi_b,\Gamma_b)$ such that $f(a^{\gamma})=f(a)^{\phi(\gamma)}$ for all $a \in \pi,\gamma \in \Gamma$. If both $f$ and $\phi$ are isomorphisms, then $(f,\phi)$ is an \emph{isomorphism of operator groups}. 

\subsection{The covering group} \label{sec_covering} Let $L \subset\Sigma \times [0,1]$ be a link where $\Sigma$ is a closed, connected, oriented surface of genus $g(\Sigma)=g>0$. Let $\widetilde{\Sigma}$ be the universal cover of $\Sigma$ and let $p:\widetilde{\Sigma} \to \Sigma$ denote the covering projection. Then $p \times \text{id}: \widetilde{\Sigma} \times [0,1] \to \Sigma \times [0,1]$ is also a covering. Set $\widetilde{L}=(p \times \text{id})^{-1}(L)$. Then there is a covering $\widetilde{X}=\widetilde{\Sigma} \times [0,1] \smallsetminus \widetilde{L}$ of $X=\Sigma \times [0,1] \smallsetminus L$. Abusing notation we will write the covering projection as $p:\widetilde{X} \to X$. Now, choose a base point $\infty$ on $\Sigma$ and a lift $\widetilde{\infty} \in p^{-1}(\Sigma)$. The \emph{covering group} of $L$ is defined to be $\widetilde{\pi}_L=\pi_1(\widetilde{\Sigma} \times [0,1] \smallsetminus \widetilde{L},(\widetilde{\infty},1))$. The covering group is neither finitely generated nor presented as a group. However, if we set $\Gamma\cong\pi_1(\Sigma,\infty)$ to be the group of covering transformations of $p$, $\widetilde{\pi}_L$ is both finitely generated and presented as a $\Gamma$-operator group. 

\begin{figure}[htb]
\[
\xymatrix{
\begin{array}{c}
\includegraphics[scale=.5]{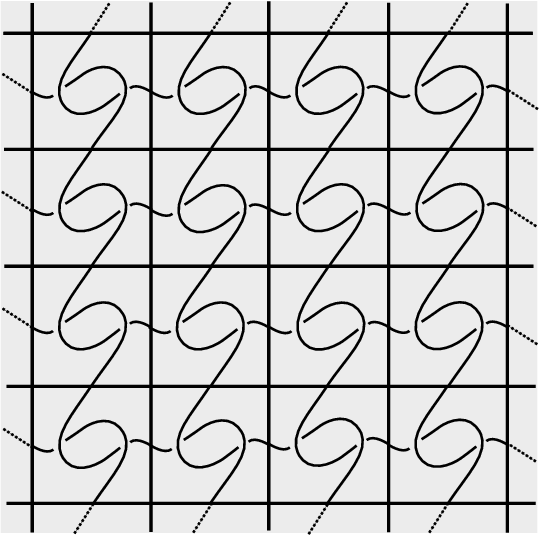}
\end{array}
\ar@/^5pc/[r]^-p & 
\begin{array}{c}
\includegraphics[scale=.6]{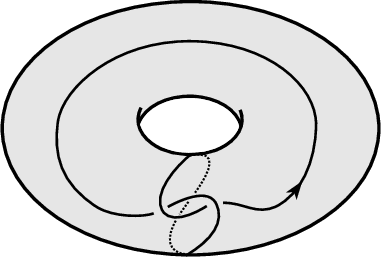}
\end{array}}
\]
\caption{Lifting a knot in $\Sigma \times [0,1]$ to the universal cover $\mathbb{R}^2 \times [0,1]$.} \label{fig_lifting}
\end{figure}

To write out this presentation, first choose a collection $x_1,y_1,\ldots,x_g,y_g$ of simple closed curves on $\Sigma$, intersecting only at $\infty$, which represent a symplectic basis of $H_1(\Sigma;\mathbb{Z})$. These are assumed to be oriented and ordered so that $x_i \cdot y_j=\delta_{ij}$ and $x_i \cdot x_j=y_i \cdot y_j=0$, where $-\cdot-$ to denotes the intersection pairing. Note that we will abuse notation throughout by confusing $x_i,y_i$ with their homology and homotopy classes. Then:
\begin{align}\label{eqn_sigma_pres}
\pi_1(\Sigma,\infty)\cong \left\langle x_1,y_1,\ldots,x_g,y_g| \prod_{i=1}^g[x_i,y_i] \right \rangle,
\end{align}
where $[a,b]=a^{-1}b^{-1}ab$.  Let $D$ be a link diagram on $\Sigma$ that intersects the symplectic basis transversely, but not at $\infty$. Choose a fundamental region $G$ of the action of $\Gamma$ on $\widetilde{\Sigma}$ such that $\infty \in \partial G$. Label the arcs of $p^{-1}(D) \subset \widetilde{\Sigma}$ according to the following rule:  if an arc $b$ is the image of an arc $a$ in $G$ by a covering transformation $\gamma \in \Gamma$, write $b=a^{\gamma}$. It follows that there are a finite number of arcs $a_1,\ldots,a_s$ lying in $G$ such that the set of orbits $a_1^{\gamma},\ldots,a_s^{\gamma}$ for $\gamma \in \Gamma$ generate $\widetilde{\pi}_{L}$ as an operator group. The Wirtinger relations for $\widetilde{\pi}_{L}$ can likewise all be written in terms of the generators $a_i^{\gamma}$. As these can be read off the fundamental domain $G$, only a finite number of them suffice to enumerate them all. Denote these by $r_1,\ldots,r_s$. We assume that the number of relators and generators coincide by \cite{CSW}, Lemma 2.3. Thus, we have an operator group presentation of $\widetilde{\pi}_L$:
\begin{align} \label{eqn_op_grp_pres}
\widetilde{P}=\langle a_1,\ldots,a_s| r_1,\ldots, r_s \rangle_{\Gamma}.
\end{align}

Note that $\widetilde{P}$ is both finitely generated and presented as an operator group. The covering group can be recovered from $\widetilde{P}$ by taking the group on the set of letters $A=\{a_1,\ldots,a_s\}$, freely generated as an operator group, and then taking the quotient by the normal $\Gamma$-subgroup $N$ generated by the set of relators $\{r_1,\ldots,r_s\}$. The presentation $\widetilde{P}$ represents the operator group $(A/N,\Gamma,\bar{\alpha})$, where $\bar{\alpha}$ is the action of $\Gamma$ induced on the quotient. This is summarized in the following theorem.

\begin{theorem} $\widetilde{\pi}_L \cong \langle a_1,\ldots,a_s| r_1,\ldots, r_s \rangle_{\Gamma}$ as operator groups. 
\end{theorem}  

Two links $L_1,L_2 \subset \Sigma \times [0,1]$ are said to be \emph{equivalent} if there is an orientation preserving diffeomorphism $h:(\Sigma \times [0,1],\Sigma \times \{0\}) \to (\Sigma \times [0,1],\Sigma \times \{0\})$ that maps $L_1$ to $L_2$ and preserves both the orientation and order of the components. In this case, $\widetilde{\pi}_{L_1} \cong \widetilde{\pi}_{L_2}$ as operator groups (\cite{CSW}, Theorem 3.4). To derive computable invariants from $\widetilde{\pi}_L$, the following lemma will be needed.

\begin{lemma} [Carter-Silver-Williams \cite{CSW}, Proposition 2.4]  \label{lemma_pres_from_op_pres} A presentation for $\pi_1(X, (\infty,1))$ can be obtained from $\widetilde{P}$ by inserting generators $x_1,y_1,\ldots,x_g,y_g$, inserting the relation $\prod_{i=1}^g[x_i,y_i]$, and replacing every symbol of the from $a^{\gamma}$ appearing in the relators with $\gamma a \gamma^{-1}$.   
\end{lemma} 


\subsection{The infinite cyclic cover} \label{sec_infinite_cyclic} For a link $L=K_1 \cup \cdots \cup K_d \subset \Sigma \times [0,1]$ with $g(\Sigma)=g$, the CSW polynomial $\Delta_L^0$ is an element of the group ring $\mathbb{Z}[\mathbb{Z}^d \times H_1(\Sigma)]$, which is well-defined up to multiplication by units. Thus, $\Delta_L^0$ is a  polynomial in the variables $t_1,\ldots,t_d,x_1,y_1,\ldots,x_g,y_g$. Our simplified infinite-cyclic version of the CSW polynomial has only a single ``$t$'' variable, but retains the variables $x_1,y_1,\ldots,x_g,y_g$. For knots ($d=1$), this is identical to $\Delta_L^0$. As in the classical case, Alexander polynomials with a single ``$t$'' variable can be defined via infinite cyclic covers.  

Let $\varepsilon: \widetilde{\pi}_L \to \langle t \rangle$ be the map which sends all positively oriented meridians of $\widetilde{L}$ to $t$. Then $\text{ker}(\varepsilon)$ corresponds to the infinite cyclic covering $p':\widetilde{X}' \to \widetilde{X}$. The space $\widetilde{X}'$ can be constructed explicitly in analogy with the classical case (see e.g. Rolfsen \cite{rolfsen}, Chapter 5.C). First note that $H_1(\Sigma \times [0,1], \Sigma \times \{1\}) \cong \{0\}$. Hence, all links bound in relative homology. Then there is a connected orientable surface $F \subset \Sigma \times [0,1]$ such that $\partial F=L\sqcup -C$, where $C$ is a set of oriented simple closed curves in $\Sigma \times \{1\}$. Let $\widetilde{F}=p^{-1}(F)$ and let $N:(\widetilde{F} \smallsetminus \widetilde{L}) \times (-1,1) \to \widetilde{\Sigma} \times [0,1]$ be a bicollar of $\widetilde{F} \smallsetminus \widetilde{L}$ such that $N((\widetilde{F} \smallsetminus \widetilde{L}) \times \{0\})=\widetilde{F} \smallsetminus \widetilde{L}$. Define $N^-=N((\widetilde{F} \smallsetminus \widetilde{L}) \times (-1,0))$, $N^+=N((\widetilde{F} \smallsetminus \widetilde{L}) \times (0,1))$, and $N=N((\widetilde{F} \smallsetminus \widetilde{L}) \times (-1,1))$. Take infinitely many copies $N_i$ of $N$ and $M_i$ of $\widetilde{\Sigma} \times [0,1] \smallsetminus \widetilde{F}$, $- \infty < i < \infty$. Then $\widetilde{X}'$ is the quotient space obtained from the disjoint union of the copies by identifying $N^+ \subset N_i$ with $N^+ \subset M_i$ and $N^- \subset N_i$ with $N^- \subset M_{i+1}$. The identification here is via the identity map.

\begin{figure}[htb]
\begin{tabular}{|c|c|} \hline
\xymatrix{\begin{tabular}{c}\includegraphics[scale=.65]{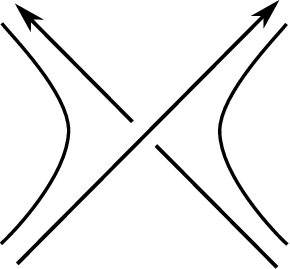} \end{tabular} \ar[r] & \begin{tabular}{c} \includegraphics[scale=.65]{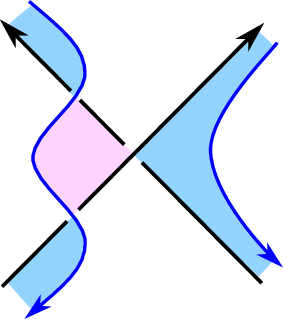} \end{tabular}} & \begin{tabular}{c} \\ \includegraphics[scale=.4]{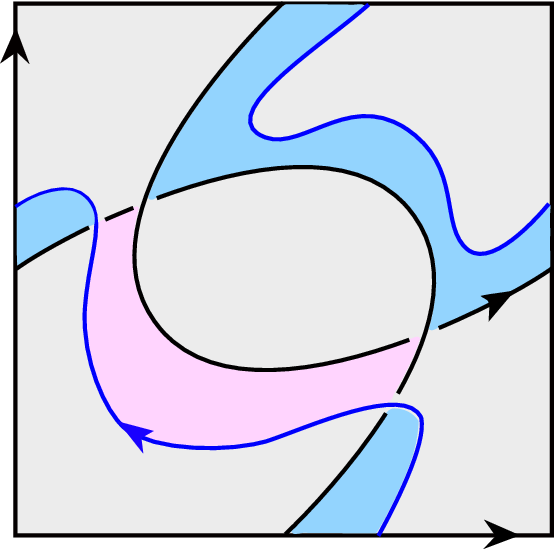} \\ \underline{The virtual trefoil:}  \end{tabular} \\  & \\ \hline
\end{tabular} 
\caption{Constructing a relative Seifert surface for a link $L \subset \Sigma \times [0,1]$.} \label{fig_rel_seif}
\end{figure}

We conclude this subsection with some remarks that will prove useful for understanding the $\Zh$-construction in Section \ref{sec_zh_review}. A surface $F \subset \Sigma \times [0,1]$, as above, such that $\partial F=L \sqcup-C$ for some family of disjoint curves $C \subset \Sigma \times \{1\}$ is called a \emph{relative Seifert surface for $L$}. A relative Seifert surface for $L$ can be constructed from a link diagram on $\Sigma$ as follows. First, perform the Seifert smoothing at the crossings of $L$ (see Figure \ref{fig_rel_seif}, left). Let $C$ denote the set of Seifert cycles. Then $F$ consists of the annuli $C \times [\tfrac{1}{2},1]$ together with half-twisted bands at the crossings $L$, attached along the annuli boundary $C \times \{ \tfrac{1}{2}\}$. This is illustrated at a crossing of $L$ in Figure \ref{fig_rel_seif}. A portion of the curve $-C$ is drawn in \textcolor{blue}{\textbf{blue}}.  For the knot in Figure \ref{fig_lifting}, the surface $F$ resembles Figure \ref{fig_rel_seif}, right. 

\subsection{The CSW polynomial} \label{sec_csw_poly} 
Our infinite cyclic CSW polynomial for links in thickened surfaces is defined following Kreinbihl \cite{kreinbihl}, who studied the case of knots. Some additional details are provided here to aid in calculation. The definition is given first, followed by an explanation.

\begin{definition}[Alexander module, CSW polynomial]\label{defn_csw} The \emph{Alexander module} $A(L)$ of a link $L \subset \Sigma \times [0,1]$ will be the group $H_1(\widetilde{X}',(p')^{-1}(\widetilde{\Sigma} \times \{1\});\mathbb{Z})$, considered as a $\mathbb{Z}[\mathbb{Z} \times H_1(\Sigma)]$-module. The \emph{Carter-Silver-Williams (CSW) polynomial} $\Delta_L^0$ is any generator of the zeroth-order ideal of $A(L)$, which is well-defined up to units of $\mathbb{Z}[\mathbb{Z} \times H_1(\Sigma)]$. 
\end{definition} 

To understand Definition \ref{defn_csw}, we must first describe the $\mathbb{Z}[\mathbb{Z} \times H_1(\Sigma)]$-module structure on $A(L)$. This is inherited from the operator group presentation of $\widetilde{\pi}_L$ as follows. Let $Y$ denote the $2$-complex associated to the presentation of $\pi_1(X,(\infty,1))$ from Lemma \ref{lemma_pres_from_op_pres}. This induces a cellular structure on $\widetilde{X}$, $\widetilde{X}'$, denoted $\widetilde{Y},\widetilde{Y}'$, respectively. Equation \ref{eqn_sigma_pres} defines a compatible cellular structure on $\Sigma \times \{1\}$, $\widetilde{\Sigma} \times \{1\}$, and $(p')^{-1}(\widetilde{\Sigma} \times \{1\})$, denoted $S, \widetilde{S}$, and $\widetilde{S}'$, respectively. Let $C_i(Z)$ denote the free abelian group on the set of $i$-cells in a cellular complex $Z$.  Then there is a commutative diagram:
\begin{equation} \label{eqn_cell_complex}
\xymatrix{C_2(\widetilde{S}') \ar[r]^-{\partial_2} \ar[d]_-{\iota_2} & C_1(\widetilde{S}') \ar[r]^-{\partial_1}  \ar[r] \ar[d]_-{\iota_1} & C_0(\widetilde{S}') \ar[d]_-{\iota_o} \ar[r]  & 0 \\ C_2(\widetilde{Y}') \ar[d]_-{\rho_2} \ar[r]^-{\partial_2} & C_1(\widetilde{Y}') \ar[r]^-{\partial_1} \ar[d]_-{\rho_1} \ar[r] & C_0(\widetilde{Y}') \ar[d]_-{\rho_0} \ar[r]  & 0 \\ \frac{C_2(\widetilde{Y}')}{C_2(\widetilde{S}')} \ar[r]^-{\partial_2} & \frac{C_1(\widetilde{Y}')}{C_1(\widetilde{S}')} \ar[r]^-{\partial_1}  \ar[r] & \frac{C_0(\widetilde{Y}')}{C_0(\widetilde{S}')} \ar[r] & 0 }
\end{equation}
In the above, $\rho_i$ denotes the natural surjection and $\iota_i$ the inclusion map. To describe the action of $\mathbb{Z}[\mathbb{Z} \times H_1(\Sigma)]$ on $C_i(\widetilde{Y}')/C_i(\widetilde{S}')$, we define a cellular complex from the operator group presentation (\ref{eqn_op_grp_pres}). This has one $0$-cell. For every $\gamma \in \Gamma$, and every $1 \le i \le s$, there is a 1-cell $a_i^{\gamma}$ and a $2$-cell $r_i^{\gamma}$. Clearly, $\mathbb{Z}[\Gamma]$ acts on the right of according to $a^{\gamma_1} \cdot \gamma_2=a^{\gamma_1 \gamma_2}$. This cellular complex can be obtained from $\widetilde{Y}'$ be collapsing all the $1$-cells corresponding to lifts of $x_1,y_1,\ldots,x_g,y_g$ to the $0$-cell. Thus, $C_i(\widetilde{Y}')/C_i(\widetilde{S}')$ inherits the $\mathbb{Z}[\Gamma]$-module structure from the operator group presentation. Furthermore, $C_i(\widetilde{Y}')/C_i(\widetilde{S}')$ has the structure of a $\mathbb{Z}[\mathbb{Z} \times \Gamma]$-module, where the action of $\mathbb{Z}=\langle t \rangle$ is that of the infinite cyclic cover. Lastly, observe that $\mathbb{Z}[\mathbb{Z} \times H_1(\Sigma)]$ has an obvious structure as a $\mathbb{Z}[\mathbb{Z} \times \Gamma]$-bimodule. Then it follows that for $i=0,1,2$, $C_i(\widetilde{Y}')/C_i(\widetilde{S}') \otimes_{\mathbb{Z}[\mathbb{Z} \times \Gamma]} \mathbb{Z}[\mathbb{Z} \times H_1(\Sigma)]$ is finitely-generated free $\mathbb{Z}[\mathbb{Z} \times H_1(\Sigma)]$-module. Using this, we can now calculate $A(L)$ and $\Delta_L^0$.

\begin{lemma} \label{lemma_alex_module} For a link $L \subset \Sigma \times [0,1]$, its Alexander module is given by:
\[
A(L) \cong \frac{C_1(\widetilde{Y}')/C_1(\widetilde{S}') \otimes_{\mathbb{Z}[\mathbb{Z} \times \Gamma]} \mathbb{Z}[\mathbb{Z} \times H_1(\Sigma)]}{ \partial_2\left(C_1(\widetilde{Y}')/C_1(\widetilde{S}') \otimes_{\mathbb{Z}[\mathbb{Z} \times \Gamma]} \mathbb{Z}[\mathbb{Z} \times H_1(\Sigma)] \right)}  
\]
\end{lemma}
\begin{proof} Since all $0$-cells are in $\widetilde{S}'$, $C_0(\widetilde{Y}')/C_0(\widetilde{S}') \cong \{0\}$. Then the result follows from the definition of $A(L)$ and the commutative diagram (\ref{eqn_cell_complex}).
\end{proof}

Next, we calculate $\Delta_L^0$ from the presentation of $\pi_1(X,(\infty,1))$ given in Lemma \ref{lemma_pres_from_op_pres}. First some notation is needed. Let $\varphi:\mathbb{Z}[\pi_1(X,(\infty,1))] \to \mathbb{Z}[\mathbb{Z} \times H_1(\Sigma)]$ be the map defined on generators by $\varphi(a_i)=t$, $\varphi(x_i)=x_i$, and $\varphi(y_i)=y_i$. Let $\psi:\mathbb{Z}[a_1,\ldots,a_s,x_1,y_1,\ldots,x_g,y_g] \to \mathbb{Z}[\pi_1(X,(\infty,1))]$ be the map induced on the group rings from the natural projection. For each $i$, $1 \le i \le s$, let $\frac{\partial}{\partial a_i}: \mathbb{Z}[a_1,\ldots,a_s,x_1,y_1,\ldots,x_g,y_g] \to \mathbb{Z}[a_1,\ldots,a_s,x_1,y_1,\ldots,x_g,y_g] $ be the Fox derivative: 
\begin{align} \label{eqn_fox}
\frac{\partial}{\partial a_i}a_j=\delta_{ij}, \quad \quad \frac{\partial}{\partial a_i}x_j=\frac{\partial}{\partial a_i}y_j=0, \quad \quad \frac{\partial}{\partial a_i} w_1 w_2=\frac{\partial }{\partial a_i} w_1+w_1\frac{\partial }{\partial a_i}w_2, \quad \quad \frac{\partial}{\partial a_i} a_i^{-1}=-a_i^{-1}.
\end{align}


\begin{theorem} \label{thm_csw_poly_jacobian} Let $P=\langle a_1,\ldots, a_s, x_1,y_1,\ldots, x_g,y_g| r_1,\ldots, r_s, \prod_{i=1}^g[x_i,y_i] \rangle$ be the presentation of $\pi_1(X,(\infty,1))$ from Lemma \ref{lemma_pres_from_op_pres}. Then up to multiplication by units of $\mathbb{Z}[\mathbb{Z} \times H_1(\Sigma)]$:
\[
\Delta_L^0(t,x_1,y_1,\ldots,x_g,y_g) \doteq \det\left( \varphi \circ \psi \left(\frac{\partial r_i}{\partial a_j} \right)\right)_{i,j=1,\ldots,s}.
\]
\end{theorem}
\begin{proof} By the preceding discussion, $C_1(\widetilde{Y}')/C_1(\widetilde{S}')$ is finitely-generated as a $\mathbb{Z}[\langle t \rangle \times H_1(\Sigma)]$-module by lifts of the elements $a_i$ and $C_2(\widetilde{Y}')/C_2(\widetilde{S}')$ is generated by lifts of the relators $r_j$. The relator $\prod_{i=1}^g[x_i,y_i]$ vanishes in the quotient, as do all generators corresponding to lifts of $x_i,y_i$. Then by the commutative diagram (\ref{eqn_cell_complex}) and Lemma \ref{lemma_alex_module}, $(\varphi \circ \psi (\partial r_i/ \partial a_j))_{i,j=1,\ldots,r}$ is a presentation matrix for $A(L)$. The result then follows by the definition of the CSW polynomial.
\end{proof}

The Jacobian in Theorem \ref{thm_csw_poly_jacobian} can be calculated directly from the operator group presentation (\ref{eqn_op_grp_pres}). This will be needed ahead in our $U_q(\mathfrak{gl}(1|1))$ model of the CSW polynomial $\Delta_L^0$. Below and in the sequel, we will write $\overline{z}$ for $z^{-1}$.

\begin{proposition} \label{prop_extend_fox} For all $a_i$, $1 \le i \le s$, and $\gamma \in \Gamma$, we have:
\begin{enumerate}
\item $\displaystyle{\frac{\partial }{\partial a_i} a_j^{\gamma}=\delta_{i j} \gamma}$.
\item $\displaystyle{\frac{\partial}{\partial a_i} \overline{a_i^{\gamma}}=-\gamma \overline{a_i^{\gamma}}}$
\end{enumerate}
\end{proposition}
\begin{proof} The first rule follows from the equations \ref{eqn_fox}  and Lemma \ref{lemma_pres_from_op_pres}:
\[
\frac{\partial}{\partial a_i}a_i^{\gamma}=\frac{\partial }{\partial a_i}\gamma a_i \overline{\gamma}=\frac{\partial }{\partial a_i}\gamma+\gamma \frac{\partial}{\partial a_i}a_i \overline{\gamma}=\gamma(1+a_i \cdot 0)=\gamma.
\]
The second rule then follows easily from the first.
\end{proof}

Lastly, it is necessary to discuss the effect of a symplectic change of basis on $\Delta_L^0$. Any orientation preserving diffeomorphism $h:\Sigma \to \Sigma$ induces on homology an element of the integral symplectic group $\text{Sp}(2g;\mathbb{Z})$. Starting with some symplectic basis of $H_1(\Sigma;\mathbb{Z})$, we may change to another one using the induced map $h_*:H_1(\Sigma) \to H_1(\Sigma;\mathbb{Z})$ of some $h:\Sigma \to \Sigma$. The map $h_*$ induces a map $h_{\sharp}:\mathbb{Z}[H_1(\Sigma;\mathbb{Z})] \to \mathbb{Z}[H_1(\Sigma)]$, which extends to a map $ h_{\sharp}:\mathbb{Z}[H_1(\Sigma;\mathbb{Z})][t^{\pm 1}] \to \mathbb{Z}[H_1(\Sigma)][t^{\pm 1}]$. The effect on $\Delta_K^0$ is to apply $h_{\sharp}$ to the coefficients of $t$. Polynomials $\Delta,\Delta' \in \mathbb{Z}[\langle t \rangle  \times H_1(\Sigma;\mathbb{Z})]$ are called \emph{equivalent}, if $\Delta=u \cdot \varphi_{\sharp}(\Delta')$ for some orientation preserving diffeomorphism $\varphi$ of $\Sigma$ and unit $u$ \cite{CSW}.

\subsection{The symplectic rank} \label{sec_symp_rank} Now consider the $\mathbb{Z}$-module $H_1(\Sigma)\cong H_1(\Gamma)$ and let $-\cdot-$ denote its intersection form. If $V$ is a submodule of $H_1(\Sigma)$, denote by $W$ the vector space $W=V \otimes \mathbb{R}$. Let $W^{\perp}$ denote the orthogonal complement of $W$:
\[
W^{\perp}=\{v \in H_1(\Sigma)|(\forall w \in W) v \cdot w=0 \}
\]
The following definition, due to Carter-Silver-Williams \cite{CSW}, is the means by which the virtual $2$-genus can be extracted from the covering group. 
\begin{definition}[Symplectic rank] In general, the \emph{symplectic rank} of a submodule $W$ is the rank of $W/(W \cap W^{\perp})$. This specializes to the following definitions for submodules of $H_1(\Sigma)$.
\begin{enumerate}
\item The \emph{symplectic rank} of an operator group presentation $\widetilde{P}=\langle A|R \rangle_{\Gamma}$ of $\widetilde{\pi}_L$ is the symplectic rank of the submodule $W_{\widetilde{P}}$ generated by elements $[\gamma] \in H_1(\Sigma;\mathbb{R})$ for which $\gamma \in \Gamma$ appears in the set of relators $R$. 
\item The \emph{symplectic rank} $\text{rk}_s(\widetilde{\pi}_L)$ of $\widetilde{\pi}_L$ is the minimum symplectic rank taken over all its operator group presentations $\widetilde{P}$. 
\item The \emph{symplectic rank} $\text{rk}_s(\Delta^0_L)$ of $\Delta_L^0=\Delta^0_L(t,x_1,y_1,\ldots,x_g,y_g)$ is the symplectic rank of the submodule $W_{\Delta}$ generated by the quotients of the nonzero coefficients of $t$ in $\Delta_L^0$.  
\end{enumerate}
\end{definition}

\begin{remark} In practice, $\text{rk}_s(\Delta^0_L)$ is calculated by writing the coefficients of $\Delta_L^0$ additively, determining a basis for this submodule, finding the dimension of $W_{\Delta} \cap W_{\Delta}^{\perp}$, and then subtracting this value from the dimension of $W_{\Delta}$. It is oftentimes the case that one can calculate the symplectic rank using this method just by a quick inspection of the coefficients.
\end{remark}

The main theorem of \cite{CSW} can now be stated as follows.

\begin{theorem}[Carter-Silver-Williams \cite{CSW}, Theorem 6.1] \label{thm_csw} If $L \subset \Sigma \times [0,1]$ is any representative of a non-split virtual link, then $\text{rk}_s(\widetilde{\pi}_L)=2 \cdot \check{g}_2(L)$. 
\end{theorem}

It is difficult in general to compute $\text{rk}_s(\widetilde{\pi}_L)$, as this would require minimizing over all operator group presentations of $\widetilde{\pi}_L$. However, $\Delta_L^0$ is independent of any presentation and $\text{rk}_s(\Delta_K^0)$ is preserved by any symplectic change of basis. Hence, it follows that (see \cite{CSW}, Corollary 6.2):
\begin{equation} \label{eqn_csw_poly_bound}
\frac{\text{rk}_s(\Delta_L^0)}{2} \le \check{g}_2(L).
\end{equation}
An example calculation of $\widetilde{\pi}_L$, $\Delta_L^0$, and $\text{rk}_s(\Delta_L^0)$ is given in the following subsection. Later, it will be shown how to calculate $\Delta_L^0$ using the prismatic $U_q(\mathfrak{gl}(1|1))$ polynomial.

\begin{figure}[htb]
\begin{tabular}{cc}
\begin{tabular}{c}
\includegraphics[scale=.5]{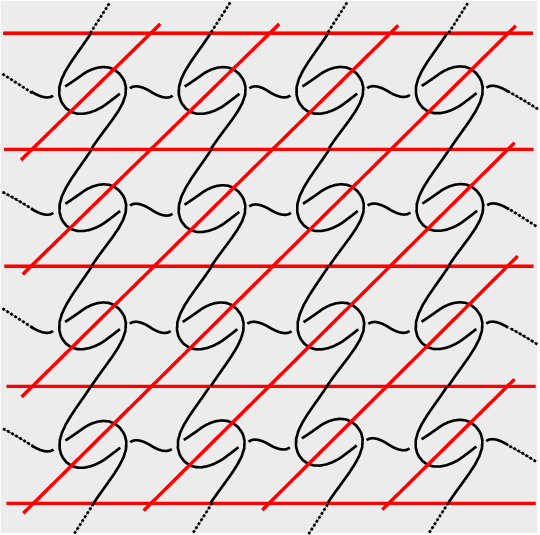}
\end{tabular} 
&
\begin{tabular}{c}
\includegraphics[scale=.5]{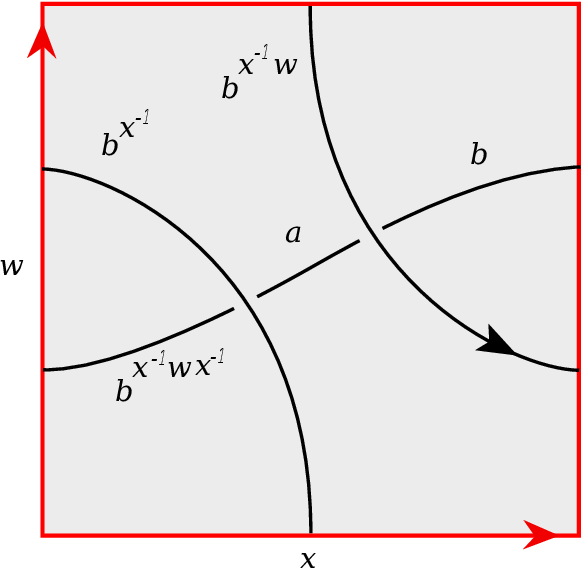}
\end{tabular}
\end{tabular}
\caption{A first choice of symplectic basis and labeling of the fundamental domain.} \label{fig_first_choice}
\end{figure}

\subsection{Example calculation} \label{sec_csw_example} Two choices of symplectic basis for the right-handed virtual trefoil are shown in Figure \ref{fig_first_choice} and Figure \ref{fig_second_choice}.  We will calculate the CSW polynomial for each and relate them by a symplectic change of basis. This example also appears in \cite{CSW} (see Example 4.3).

\begin{figure}[htb]
\begin{tabular}{cc}
\begin{tabular}{c}
\includegraphics[scale=.5]{virt_tref_lift.eps}
\end{tabular} 
&
\begin{tabular}{c}
\includegraphics[scale=.5]{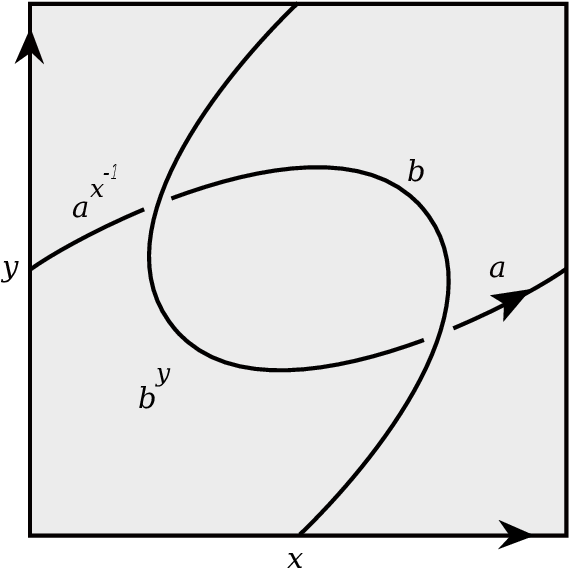}
\end{tabular}
\end{tabular}
\caption{A second choice of the symplectic basis and labeling of the fundamental domain.} \label{fig_second_choice}
\end{figure}

For the choice in Figure \ref{fig_first_choice}, the two crossing relations are:
\[
b^{x^{-1}}a\overline{b^{x^{-1}}}\,\overline{b^{x^{-1}wx^{-1}}}, \quad \quad b^{x^{-1}w}b\overline{b^{x^{-1}w}} \overline{a}
\]
The Jacobian matrix and its determinant work out to be:
\[
\begin{pmatrix}
t & x^{-1}-x^{-1}t-w x^{-2} \\
-1 & x^{-1}w+t-tx^{-1}w
\end{pmatrix} \implies \Delta_K^0(t,x,w)=(1-x^{-1}w) t^2+(x^{-1}w-x^{-1}) t+(x^{-1}-wx^{-2}).
\]
Written additively, $W_{\Delta}=\langle -x+w,-x,-2x+w \rangle=\langle x,w\rangle$. Hence, $\text{rk}_s(\Delta_K^0(t,x,w))=2$ and $\check{g}_2(K)=1$. For the right hand side of Figure \ref{fig_second_choice}, the crossing relations are:
\[
ba\overline{b} \,\overline{b^{y}}, \quad \quad b^{y}b \overline{b^{y}} \overline{a^{x^{-1}}},
\] 
In this case, the Jacobian matrix and its determinant are given by:
\[
\begin{pmatrix}
t & 1-t-y \\
-x^{-1} & y+t-yt
\end{pmatrix} \implies \Delta_K^0(t,x,y)=(1-y) t^2+(y-x^{-1}) t+(x^{-1}-x^{-1} y).
\]
Here we get $W_{\Delta}=\langle y,-x,y-x\rangle=\langle x,y \rangle$. Again it follows that $\text{rk}_s(\Delta_K^0(t,x,y))=2$. Observe that setting $w=xy$ in the first equation gives the second equation:
\[
(1-x^{-1}(xy)) t^2+(x^{-1}(xy)-x^{-1}) t+(x^{-1}-(xy)x^{-2})=(1-y) t^2+(y-x^{-1}) t+(x^{-1}-x^{-1} y).
\]

\section{The homotopy\!\! $\Zh$-construction} \label{sec_homotopy_zh}
The homotopy $\Zh$-construction is the initial step of the bookkeeping method for quantum invariants (see Section \ref{sec_main_idea}). This will be shown in Section \ref{sec_GAP_and_homology_zh} to be a generalization of the Bar-Natan $\Zh$-construction. We begin with a review of the original $\Zh$-construction in Section \ref{sec_zh_review}. Prismatic links are defined in Section \ref{sec_prismatic_links}. The homotopy $\Zh$-construction is defined in Section \ref{sec_homotopy_zh_defn}. The relation between prismatic links and the covering group is given in Section \ref{sec_prismatic_link_group}. It will be assumed throughout that the reader is familiar with virtual link diagrams, extended Reidemeister moves, and Gauss diagrams. For details, see Kauffman \cite{kauffman_vkt}. For a minimalist review, see \cite{CP}.

\subsection{The Bar-Natan\!\! $\Zh$-construction}\label{sec_zh_review} First, we define a codomain for the $\Zh$-construction. A \emph{semi-welded link diagram} is an $(r+1)$-component virtual link diagram with a distinguished last component. This is called the $\omega$-component and will be colored \textcolor{blue}{\textbf{blue}} in figures. The other components are called the $\alpha$-part, which are drawn in \textbf{black}. Semi-welded links are considered equivalent up to extended Reidemeister moves and the \emph{semi-welded move} (see Figure \ref{fig_sw_move_and_zh}, left). Note that in the semi-welded move, the over-crossing arcs are both in the $\omega$-component and the under-crossing arcs are in the $\alpha$-part. 

\begin{figure}[htb]
\[
\begin{tabular}{|c||cc|} \hline & & \\
$\begin{array}{c} \def\svgwidth{1.1in} \tiny
\begingroup%
  \makeatletter%
  \providecommand\color[2][]{%
    \errmessage{(Inkscape) Color is used for the text in Inkscape, but the package 'color.sty' is not loaded}%
    \renewcommand\color[2][]{}%
  }%
  \providecommand\transparent[1]{%
    \errmessage{(Inkscape) Transparency is used (non-zero) for the text in Inkscape, but the package 'transparent.sty' is not loaded}%
    \renewcommand\transparent[1]{}%
  }%
  \providecommand\rotatebox[2]{#2}%
  \newcommand*\fsize{\dimexpr\f@size pt\relax}%
  \newcommand*\lineheight[1]{\fontsize{\fsize}{#1\fsize}\selectfont}%
  \ifx\svgwidth\undefined%
    \setlength{\unitlength}{234.89887566bp}%
    \ifx\svgscale\undefined%
      \relax%
    \else%
      \setlength{\unitlength}{\unitlength * \real{\svgscale}}%
    \fi%
  \else%
    \setlength{\unitlength}{\svgwidth}%
  \fi%
  \global\let\svgwidth\undefined%
  \global\let\svgscale\undefined%
  \makeatother%
  \begin{picture}(1,0.65155205)%
    \lineheight{1}%
    \setlength\tabcolsep{0pt}%
    \put(0,0){\includegraphics[width=\unitlength]{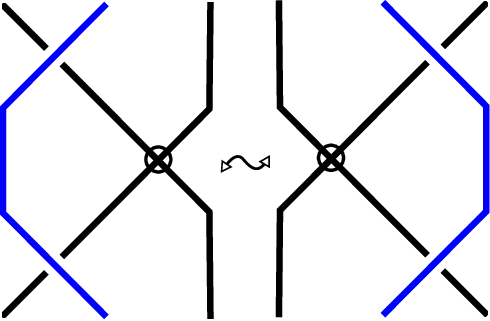}}%
    \put(0.44323135,0.34193985){\color[rgb]{0,0,0}\makebox(0,0)[lt]{\lineheight{40.54999924}\smash{\begin{tabular}[t]{l}$sw$\end{tabular}}}}%
  \end{picture}%
\endgroup%
 \end{array}$ & \xymatrix{
\begin{array}{c}
\def\svgwidth{.3in} \tiny 
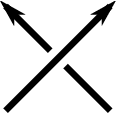
\end{array} \ar[r]^-{\Zh} & \begin{array}{c}\def\svgwidth{1.1in} \tiny 
\begingroup%
  \makeatletter%
  \providecommand\color[2][]{%
    \errmessage{(Inkscape) Color is used for the text in Inkscape, but the package 'color.sty' is not loaded}%
    \renewcommand\color[2][]{}%
  }%
  \providecommand\transparent[1]{%
    \errmessage{(Inkscape) Transparency is used (non-zero) for the text in Inkscape, but the package 'transparent.sty' is not loaded}%
    \renewcommand\transparent[1]{}%
  }%
  \providecommand\rotatebox[2]{#2}%
  \newcommand*\fsize{\dimexpr\f@size pt\relax}%
  \newcommand*\lineheight[1]{\fontsize{\fsize}{#1\fsize}\selectfont}%
  \ifx\svgwidth\undefined%
    \setlength{\unitlength}{236.86404916bp}%
    \ifx\svgscale\undefined%
      \relax%
    \else%
      \setlength{\unitlength}{\unitlength * \real{\svgscale}}%
    \fi%
  \else%
    \setlength{\unitlength}{\svgwidth}%
  \fi%
  \global\let\svgwidth\undefined%
  \global\let\svgscale\undefined%
  \makeatother%
  \begin{picture}(1,0.6607376)%
    \lineheight{1}%
    \setlength\tabcolsep{0pt}%
    \put(0,0){\includegraphics[width=\unitlength]{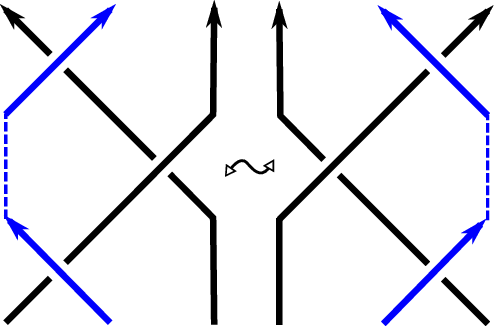}}%
    \put(0.44820787,0.34479521){\color[rgb]{0,0,0}\makebox(0,0)[lt]{\lineheight{40.54999924}\smash{\begin{tabular}[t]{l}$sw$\end{tabular}}}}%
  \end{picture}%
\endgroup%
 \end{array}} & \xymatrix{ \begin{array}{c}
\def\svgwidth{.3in} \tiny 
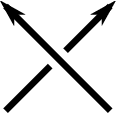
\end{array} \ar[r]^-{\Zh} & \begin{array}{c}\def\svgwidth{1.1in} \tiny 
\begingroup%
  \makeatletter%
  \providecommand\color[2][]{%
    \errmessage{(Inkscape) Color is used for the text in Inkscape, but the package 'color.sty' is not loaded}%
    \renewcommand\color[2][]{}%
  }%
  \providecommand\transparent[1]{%
    \errmessage{(Inkscape) Transparency is used (non-zero) for the text in Inkscape, but the package 'transparent.sty' is not loaded}%
    \renewcommand\transparent[1]{}%
  }%
  \providecommand\rotatebox[2]{#2}%
  \newcommand*\fsize{\dimexpr\f@size pt\relax}%
  \newcommand*\lineheight[1]{\fontsize{\fsize}{#1\fsize}\selectfont}%
  \ifx\svgwidth\undefined%
    \setlength{\unitlength}{239.48080902bp}%
    \ifx\svgscale\undefined%
      \relax%
    \else%
      \setlength{\unitlength}{\unitlength * \real{\svgscale}}%
    \fi%
  \else%
    \setlength{\unitlength}{\svgwidth}%
  \fi%
  \global\let\svgwidth\undefined%
  \global\let\svgscale\undefined%
  \makeatother%
  \begin{picture}(1,0.6515535)%
    \lineheight{1}%
    \setlength\tabcolsep{0pt}%
    \put(0,0){\includegraphics[width=\unitlength]{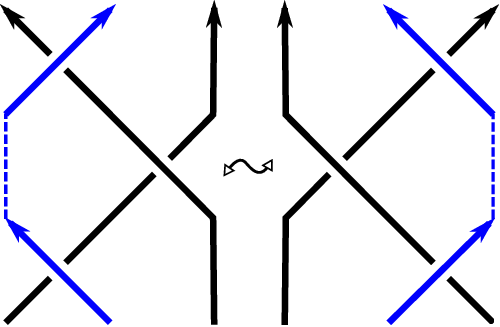}}%
    \put(0.44978272,0.34246882){\color[rgb]{0,0,0}\makebox(0,0)[lt]{\lineheight{40.54999924}\smash{\begin{tabular}[t]{l}$sw$\end{tabular}}}}%
  \end{picture}%
\endgroup%
 \end{array}
} \\\underline{The semi-welded move:}  & \underline{$\Zh$ for $\oplus$-crossings:} & \underline{$\Zh$ for $\ominus$-crossings:} \\ & & \\ \hline
\end{tabular}
\]
\caption{The semi-welded move (left) and the Bar-Natan $\Zh$-construction (right).} \label{fig_sw_move_and_zh}
\end{figure}
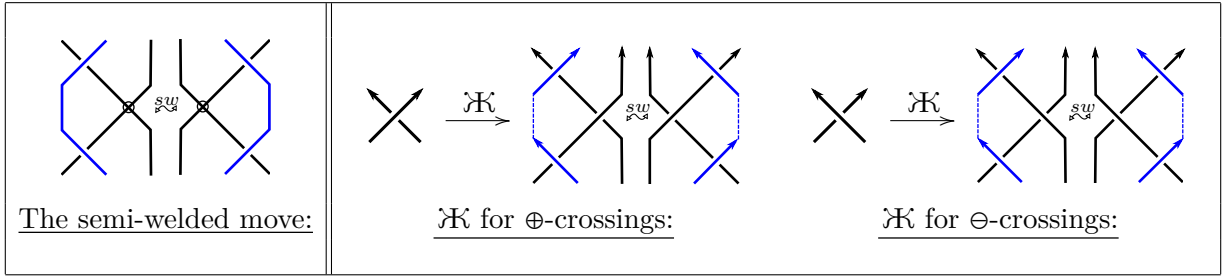

Now, let $L$ be an $r$-component virtual link diagram. The \emph{$\Zh$-construction} associates an $(r+1)$-component semi-welded link diagram $\Zh(L)=L \cup \omega$ as follows. First, each classical crossing of $L$ is flanked by two over-crossing blue arcs, as shown in Figure \ref{fig_sw_move_and_zh}, right. The next step is to connect all of these blue arcs together (arbitrarily) into a single component. If new intersections between blue arcs and black arcs (or with other blue arcs) are created during this process, these transversal intersections are marked as virtual crossings. This is well-defined due to the semi-welded move. Indeed, traveling along the blue arc, the semi-welded move allows any adjacent over-crossings with the $\alpha$-part to be exchanged. Since transpositions generate permutations, any desired order of the $\alpha$-part along $\omega$ can be achieved. Furthermore, if $L_1,L_2$ are equivalent virtual link diagrams, then $\Zh(L_1), \Zh(L_2)$ are semi-welded equivalent. These basic facts about the $\Zh$-construction were first proved by Bar-Natan \cite{bar_natan_talk}. Careful proofs can be found in \cite{boden_chrisman_21}. For an example, see Figure \ref{fig_zh_virt_tref}. Note that some blue arcs are drawn dashed to emphasize their arbitrary placement.

\begin{figure}[hbt]
\begin{tabular}{|c|} \hline \\
$\xymatrix{\begin{array}{c}\includegraphics[scale=.3]{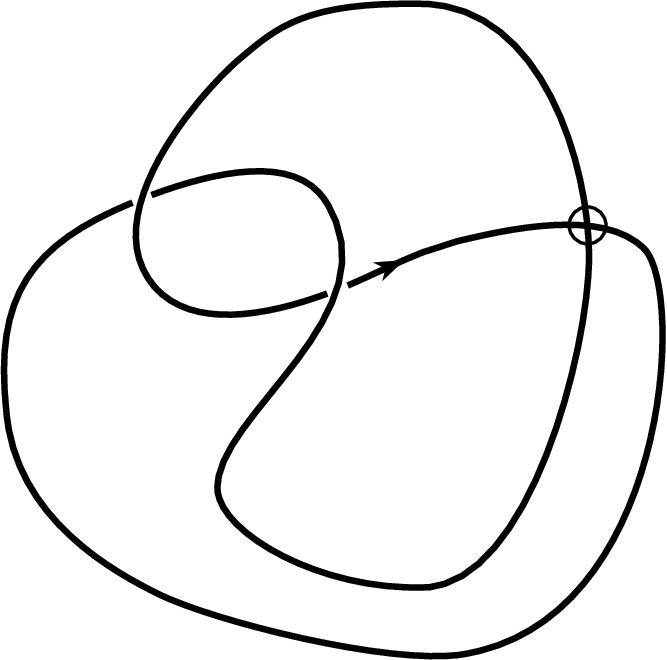} \end{array} \ar[rr]^-{\Zh} & & \begin{array}{c} \includegraphics[scale=.3]{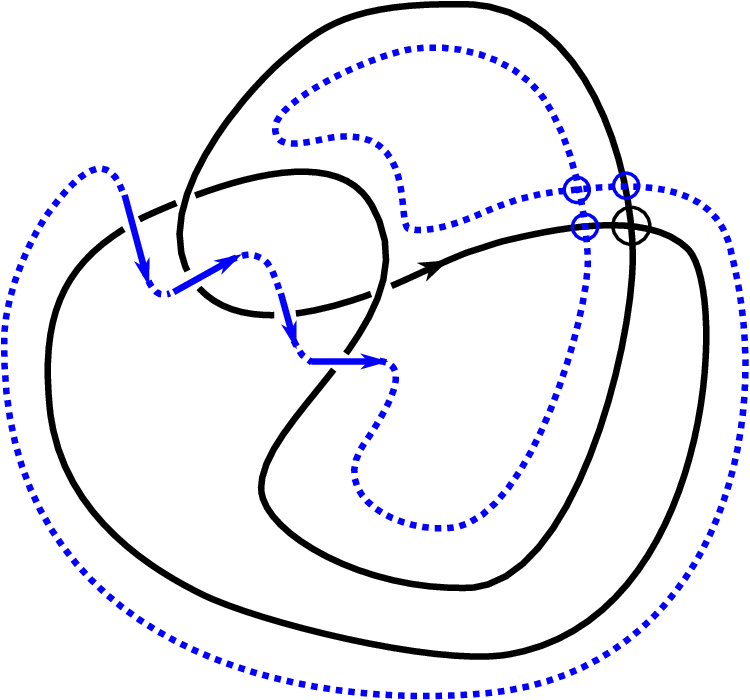} \end{array}}$
 \\ \\ \hline 
\end{tabular}
\caption{$\Zh$-construction for a right virtual trefoil. Compare Figure \ref{fig_rel_seif}.} \label{fig_zh_virt_tref}
\end{figure}

The topological interpretation of the $\Zh(L)$ is that $\omega$ depicts the homology class $[D] \in H_1(\Sigma;\mathbb{Z})$, where $D$ is any diagram representing $L$ on some surface $\Sigma$. To see this, compare Figure \ref{fig_rel_seif} with Figure \ref{fig_sw_move_and_zh}. Observe that the former may be obtained from the latter by changing the orientation of $\omega$ in Figure \ref{fig_sw_move_and_zh}. When the orientation of $\omega$ is changed, it is called the \emph{$\Zh^{\text{op}}$-construction}, denoted $\Zh^{\text{op}}(L)=L \cup \omega^{\text{op}}$. Then $\Zh^{\text{op}}(L)$ is the virtual link diagram corresponding to the boundary of a relative Seifert surface for $D$, as discussed in Section \ref{sec_infinite_cyclic}. Hence, $\omega$ is a diagrammatic representation of $[D] \in H_1(\Sigma;\mathbb{Z})$. The semi-welded move also has a topological interpretation: it records the fact that $H_1(\Sigma;\mathbb{Z})$ is abelian. For a more precise description of the  topological interpretation of the $\Zh(L)$, see \cite{chrisman_todd_23}, Section 3.3. To avoid confusion with the homotopy $\Zh$-construction defined ahead, we will henceforth refer to the original $\Zh$-construction as the \emph{homology $\Zh$-construction}. We now proceed to our multicolor generalization of semi-welded links. 

\subsection{Prismatic links} \label{sec_prismatic_links} The homotopy $\Zh$-construction is a multicolor generalization of the homology $\Zh$-construction where the additional components represent generators of $\pi_1(\Sigma,\infty)$. The set of allowed colors will be called the \emph{symplectic palette}, which is an ordered set of letters $\Omega=\{x_1,y_1,\ldots,x_g,y_g\}$. In addition to the colors $\Omega$, there is a distinguished color $\alpha$, which is shown as \textbf{black} in figures. A \emph{$\Omega$-prismatic link diagram} is a virtual link diagram $D$ whose components are colored in the set $\{\alpha\} \cup \Omega$. The \textbf{black} components are again called the $\alpha$-part and the other components are called the $\Omega$-part. Furthermore, it is required that in any classical crossing involving the $\Omega$-part, the over-crossing strand is in the $\Omega$-part and the under-crossing strand is in the $\alpha$-part\footnote{This is not strictly necessary for a well-defined theory (see e.g. \cite{CP}), but it more closely matches the geometric interpretation that the $\Omega$-part represents curves in $\Sigma \times \{1\}$, so that such components lie above the $\alpha$-part.}. If every color in $\{\alpha\} \cup \Omega$ appears in $D$ and each color in $\Omega$ appears exactly once, $D$ is called \emph{complete}. 

By an \emph{$\Omega$-semi-welded move}, we mean a semi-welded move where the over-crossing arc in Figure \ref{fig_sw_move_and_zh}, left, can carry any of the colors in $\Omega$. The requirement that the two under-crossing arcs are in the $\alpha$-part remains in force. By an \emph{$\Omega$-commutator move}, we mean the move shown in Figure \ref{fig_commutator_move}. Note that a commutator move can only be applied to a complete prismatic link. Also note that the over-crossing arcs of the same color are in the same component. Furthermore, observe that the under-crossing $\alpha$-arc is not oriented, so that the move applies to both possible orientations. The $\Omega$-commutator move is a diagrammatic realization of the relation $\prod_{i=1}^g [x_i,y_i]=1$ in $\pi_1(\Sigma,\infty)$. 

\begin{definition}[$\Omega$-prismatic equivalence] Two $\Omega$-prismatic links $W_1,W_2$ are said to be \emph{equivalent} if they are related by a finite sequence of extended Reidemeister moves, $\Omega$-semi-welded moves, and $\Omega$-commutator moves, where all intermediate diagrams are $\Omega$-prismatic links.
\end{definition}

\begin{figure}[htb]
\begin{tabular}{|ccccc|}  \hline & & & & \\ &
\includegraphics[scale=.6]{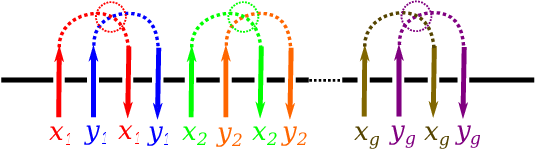} & \begin{tabular}{c} $\longleftrightarrow$ \\ \\ \\ \\ \end{tabular} & \includegraphics[scale=.6]{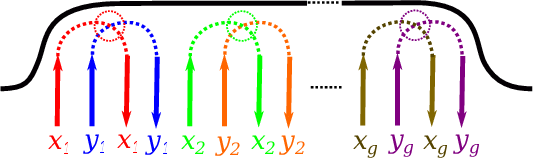} & \\  \hline
\end{tabular}
\caption{The $\Omega$-commutator move with $\Omega=\{x_1,y_1,\ldots,x_g,y_g\}$.} \label{fig_commutator_move}
\end{figure}

Note that in Figure \ref{fig_commutator_move}, the dashed arcs in between between solid arcs of the same color can be placed arbitrarily. This is again due to the $\Omega$-semi-welded move, which allows the over-crossings by any $\Omega$-colored component to be arranged along that component in any order. See Section \ref{sec_zh_review} for a discussion of the single color case. This fact is recorded below for future reference.

\begin{lemma} \label{lemma_rearrange} Let $W$ be an $\Omega$-prismatic link and $z \in \Omega$. Let $J$ be a component in $W$ colored $z$. Let $W'$ be a an $\Omega$-prismatic link obtained from $W$ by rearranging the virtual crossings in $J$ such that the classical crossings along $J$ occur in any arbitrary order. Then $W$ and $W'$ are equivalent. 
\end{lemma}

The $\Omega$-commutator move can appear in many different forms. The next result (and its proof) shows that any algebraic rearrangement of $\prod_{i=1}^g [x_i,y_i]=1$ is also an $\Omega$-commutator move.

\begin{lemma} \label{lemma_commutator_permute} The $\Omega$-commutator move holds after a cyclic permutation of its over-crossing arcs.
\end{lemma}
\begin{proof} Write the $\Omega$-commutator relation in shorthand as $x_1^{-1} y_1^{-1} x_1 y_1 \cdots x_g^{-1} y_g^{-1} x_g y_g=1$. This means that going from left to right on the under-crossing arc, we first see $x_1$ pointing up ($-1$ exponent), then $y_1$ pointing up ($-1$ exponent), then $x_1$ pointing down ($+1$ exponent), and so on. Turning the paper upside down, reading left to right, the relation is: $y_g^{-1} x_g^{-1} y_g x_g \cdots y_1^{-1} x_1^{-1} y_1 x_1=1$. Then:
\[
x_1^{-1}= 1 \cdot x_1^{-1} =y_g^{-1} x_g^{-1} y_g x_g \cdots y_1^{-1} x_1^{-1} y_1 x_1 x_1^{-1}
\]
By Lemma \ref{lemma_rearrange}, the arcs of the component colored $x_1$ can be rearranged so that $x_1^{-1}$ is passed immediately after $x_1$. Then $x_1 x_1^{-1}$ can be removed by a Reidemeister $2$ move. This gives:
\[
x_1^{-1}=y_g^{-1} x_g^{-1} y_g x_g \cdots y_1^{-1} x_1^{-1} y_1.
\]
Using this identity, and employing Lemma \ref{lemma_rearrange} as necessary to rearrange the arcs, we have:
\[
y_1^{-1} x_1 y_1 \cdots x_g^{-1} y_g^{-1} x_g y_g x_1^{-1}=y_1^{-1} x_1 y_1 \cdots x_g^{-1} y_g^{-1} x_g y_g y_g^{-1} x_g^{-1} y_g x_g \cdots y_1^{-1} x_1^{-1} y_1=1
\]
Continuing in this fashion, any cyclic permutation of the over-crossing arcs is realizable.
\end{proof}

\subsection{The homotopy\!\!\! $\Zh$-construction} \label{sec_homotopy_zh_defn} Let $D$ be a diagram of a link $L \subset \Sigma \times [0,1]$, where $g(\Sigma)=g$. Choose a symplectic basis $\Omega=\{x_1,y_1,\ldots,x_g,y_g\}$ of $H_1(\Sigma;\mathbb{Z})$, with base point $\infty$, which is ordered and oriented so that $\prod_{i=1}^g[x_i,y_i]=1$ in $\pi_1(\Sigma,\infty)$. Suppose that the arcs of $D$ intersects the curves in $\Omega$ transversely,  never at the crossings of $D$, and so that the base point at $\infty$ is avoided. The \emph{homotopy $\Zh$-construction} associates to $D$ an $\Omega$-prismatic link diagram $\Zh_{\Omega}(D)$ as follows. First, let $\widecheck{D}$ be a virtual link diagram having the same Gauss diagram as $D$ on $\Sigma$. This can be constructed, for example, by cutting $\Sigma$ along the curves in $\Omega$ and laying the resulting $4g$-gon in $\mathbb{R}^2$. Then the arcs of $D$ can be connected together appropriately, inserting virtual crossings as necessary. Moving along an arc $c$ of $D$, it may intersect an element $z \in \Omega$. Let $\widecheck{c}$ denote the arc corresponding to $c$ in $\widecheck{D}$. For each such intersection on $c$, we draw a small arc $\widecheck{b}$ over-crossing $\widecheck{c}$ which satisfies:
\begin{enumerate}
\item the color of $\widecheck{b}$ is $x_i$, if $z=y_i$, and the color of $\widecheck{b}$ is $y_i$, if $z=x_i$.
\item the orientation of $\widecheck{b}$ agrees with that of $z$ if $z=y_i$ and disagrees if $z=x_i$. 
\end{enumerate} 
Subsequent intersections along $c$ with an element of $\Omega$ are drawn on $\widecheck{c}$ in the same order as they appear on $c$. To complete $\Zh_{\Omega}(D)$, connect all of the added over-crossings arcs $\widecheck{b}$ having the same color into a single component. This can be done arbitrarily, as in the $\Zh$-construction. If any new intersections are created while connecting the over-crossing arcs $\widecheck{b}$ together, these are marked as virtual crossings. If some color of $\Omega$ is missing, draw an unknot disjoint from the rest of the diagram having that color. That $\Zh_{\Omega}(D)$ is a well-defined complete $\Omega$-prismatic link for a fixed diagram $D$ follows immediately from Lemma \ref{lemma_rearrange}. An example of $\Zh_{\Omega}(D)$ is drawn in Figure \ref{fig_mvzh_example}.

\begin{figure}[htb]
\begin{tabular}{|c|} \hline \\
$\xymatrix{
\begin{tabular}{c} \includegraphics[scale=.45]{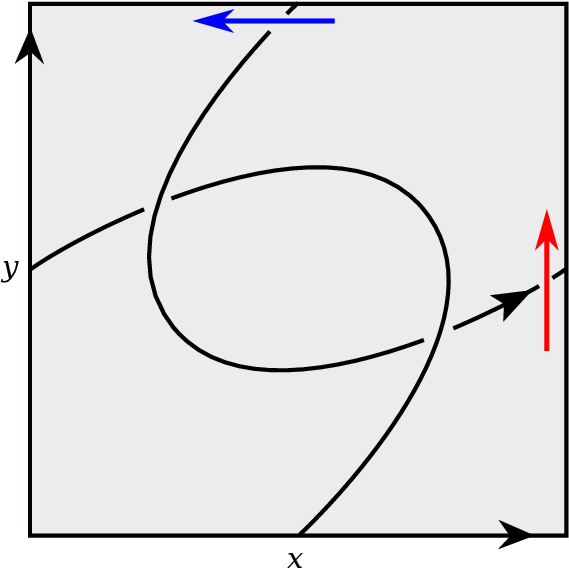} \end{tabular} \ar[rr]^-{\Zh_{\Omega}} & & \begin{tabular}{c} \includegraphics[scale=.35]{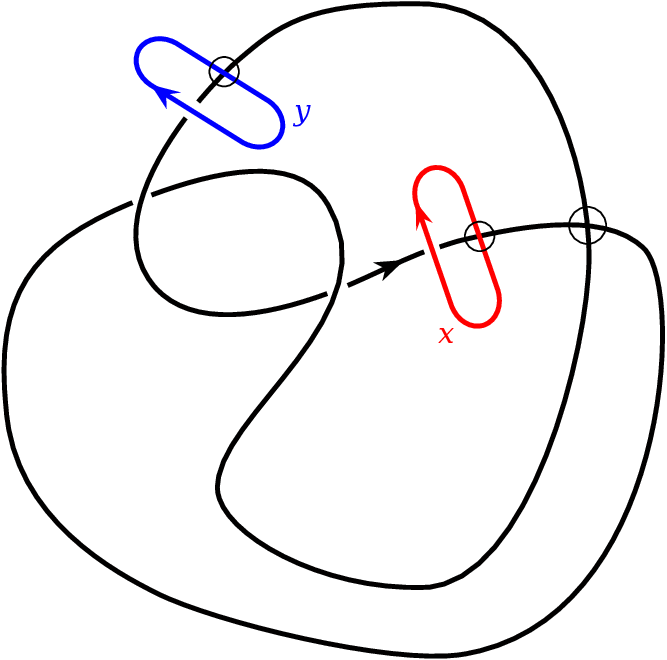} \end{tabular}
}$ \\ \\ \hline
\end{tabular}
\caption{Example of the prismatic\! $\Zh$-construction with $\Omega=\{x,y\}$.} \label{fig_mvzh_example}
\end{figure}

The homotopy $\Zh$-construction remembers the interaction between a link diagram on $\Sigma$ and the fixed generators of $\pi_1(\Sigma,\infty)$. The orientation conventions are chosen so that the $\Omega$-part records the homology of $-[D]$ in the coordinates of the symplectic basis. In other words, the homotopy $\Zh$-construction coordinatizes the $\Zh^{\text{op}}$-construction (see Section \ref{sec_zh_review}). This is proved in the following proposition. First, recall the definition of the virtual linking number. If $L=J \cup K$ is a two-component virtual link diagram, the \emph{virtual linking number} $\vlk(J,K)$ is the number of times, counted with crossing sign, that $J$ crosses over $K$. Beware: $\vlk(J,K)$ is not symmetric!

\begin{proposition} \label{prop_coordinatize_zh} Let $D$ be a link diagram on a surface $\Sigma$ and let $\Omega=\{x_1,y_1,\ldots,x_g,y_g\}$ be a choice of symplectic basis. For $z \in \Omega$, let $\widecheck{z}$ denote the $z$-colored component of $\Zh_{\Omega}(D)$. Then $[D] \in H_1(\Sigma;\mathbb{Z})$ in terms of $[x_i],[y_i] \in H_1(\Sigma;\mathbb{Z})$ is:
\[
[D]= \sum_{i=1}^g -\vlk\left(\widecheck{x}_i,\widecheck{D}\right) [x_i]-\vlk\left(\widecheck{y}_i,\widecheck{D}\right)[y_i].
\]
\end{proposition}
\begin{proof} Since the choice of symplectic basis decomposes $\Sigma$ into a connect sum of tori, it may be assumed that $g=1$ and $\Omega=\{x,y\}$. Let $[D]= \alpha [x]+\beta [y]$. Then $[D] \cdot [x]$ and  $[D]\cdot [y]$ are given by:
\[
\begin{pmatrix}
\alpha & \beta
\end{pmatrix} \begin{pmatrix}
0 & 1 \\ -1 & 0
\end{pmatrix} \begin{pmatrix}
1 \\ 0
\end{pmatrix}=-\beta, \quad \quad \begin{pmatrix}
\alpha & \beta
\end{pmatrix} \begin{pmatrix}
0 & 1 \\ -1 & 0
\end{pmatrix} \begin{pmatrix}
0 \\ 1
\end{pmatrix}=\alpha
\]
The following assertions are evident after consultation with Figure \ref{fig_mvzh_example}: if $D$ has a positive (resp., negative) intersection with $y$, the contribution to $\vlk(\widecheck{x},\widecheck{D})$ is $-1$ (resp. $+1$); if $D$ has a positive (resp., negative) intersection with $x$, the contribution to $\vlk(\widecheck{y},\widecheck{D})$ is $1$ (resp. $-1$). Then the above calculation implies that $\alpha=-\vlk(\widecheck{x},\widecheck{D})$ and $-\beta=\vlk(\widecheck{y},\widecheck{D})$.  
\end{proof}

Next we show that the homotopy $\Zh$-construction is an isotopy invariant of links in $\Sigma \times [0,1]$.

\begin{theorem} \label{thm_homotopy_zh_invariance} Let $\Omega=\{x_1,y_1,\ldots,x_g,y_g\}$ be a symplectic basis for $\Sigma$. If $D_1,D_2 \subset \Sigma \times [0,1]$ are Reidemeister equivalent diagrams on $\Sigma$, then $\Zh_{\Omega}(D_1)$, $\Zh_{\Omega}(D_2)$ are equivalent $\Omega$-prismatic links.
\end{theorem}
\begin{proof} Cut $\Sigma$ along the curves in $\Omega$ to give a $4g$-gon $G$. Suppose that $D_1,D_2$ are related by a single Reidemeister move. A Reidemeister move occurs in a small ball $B$ on $\Sigma$. First consider the case that $B$ can be chosen so small that $\infty \not\in B$ but the move can nonetheless still be performed. Then $B$ can be contracted to so that $B \subset G$. When drawn in $\mathbb{R}^2$, this contraction may involve extended Reidemeister moves, but this does not change the equivalence class of $\Zh_{\Omega}(D_1)$. Performing the Reidemeister move in $B$ then gives an $\Omega$-prismatic link equivalent to $\Zh_{\Omega}(D_2)$.

The remaining case is that every ball $B$ defining the Reidemeister move contains $\infty$. That is, $\infty$ is in the interior of the triangle defining a Reidemeister 3 move, the bigon defining a Reidemeister 2 move, or the monogon defining a Reidemeister 1 move. To use the trick from the previous case, an isotopy must be applied to move one of the arcs involved in the move over the base point $\infty$. The effect on $\Zh_{\Omega}(D_1)$ of such an isotopy can be realized by an $\Omega$-commutator move. Indeed, if an arc intersects an element of $\Omega$, say $x_1$, and needs to be moved, an isotopy over the base point will now force an intersection with all the colors in $\Omega$. For example, in Figure \ref{fig_bp_moves_preserve_sr} an isotopy over the base point causes the replacement: 
\[
y_1=x_1y_1x_2^{-1}y_2^{-1}x_2y_2x_1^{-1}.
\] 
Here we are using the notation of Lemma \ref{lemma_commutator_permute}, where an upwards pointing arc is counted as $z^{-1}$ and a downwards one as $z$. Using the method of the proof of Lemma \ref{lemma_commutator_permute}, this relation is seen to be a consequence of the $\Omega$-semi-welded relations and commutator moves. In general, it is equivalent to $\prod_{i=1}^g [x_i,y_i]$. Hence, the isotopy over the base point results in an equivalent $\Omega$-prismatic link. The desired Reidemeister move may now be performed inside a ball $B$ such that $\infty \not\in B$. Hence, if follows from the previous case that $\Zh_{\Omega}(D_1)$ and $\Zh_{\Omega}(D_2)$ are equivalent $\Omega$-prismatic links.
\end{proof}

\begin{figure}[htb]
\begin{tabular}{|c|} \hline \\
\xymatrix{
\begin{tabular}{c} \includegraphics[scale=.5]{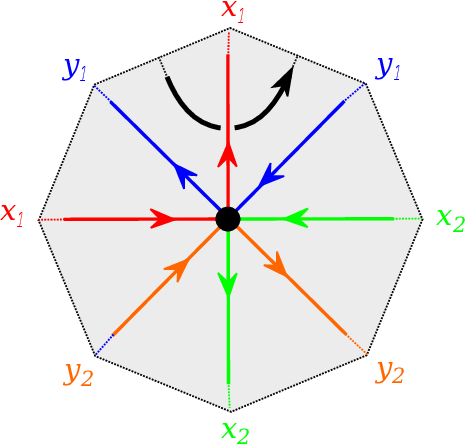} \end{tabular} \ar[r] & \begin{tabular}{c} \includegraphics[scale=.5]{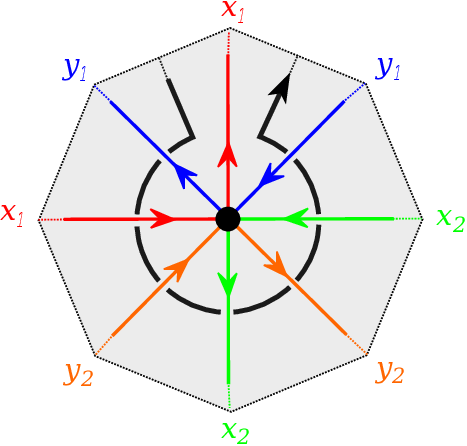} \end{tabular}
} \\ \\ \hline
\end{tabular}
\caption{An isotopy over a base point introduces an $\Omega$-commutator move in $\Zh_{\Omega}(D)$.} \label{fig_bp_moves_preserve_sr}
\end{figure}

Lastly, we discuss the effect of an orientation preserving diffeomorphism $h:\Sigma \to \Sigma$ on the homotopy $\Zh$-construction. Given an initial choice of symplectic basis, $\Omega=\{x_1,y_1,\ldots,x_g,y_g\}$ with base point $\infty$, $h$ maps $\Omega$ to a symplectic basis $h(\Omega)=\{h(x_1),h(y_1),\ldots,h(x_g),h(y_g)\}$ with common base point $h(\infty)$. This follows from the fact that the intersection number is preserved by orientation preserving diffeomorphisms. For the same reason, if $D$ is a link diagram on $\Sigma$, $h(D)$ is a link diagram on $\Sigma$ having the same underlying Gauss diagram as $D$. Then $\Zh_{h(\Omega)}(h(D))$ can be drawn by relabeling:

\begin{lemma} \label{lemma_homotopy_zh_and_diffeo} For $h: \Sigma \to \Sigma$ an orientation preserving diffeomorphism, $D$ a link diagram on $\Sigma$, and $\Omega=\{x_1,y_1,\ldots,x_g,y_g\}$ a symplectic basis, $\Zh_{h(\Omega)}(h(D))$ equivalent to the $h(\Omega)$-prismatic link obtained from $\Zh_{\Omega}(D)$ by relabeling the $z$-colored component by $h(z)$ for all $z \in \Omega$. 
\end{lemma}

\subsection{The prismatic link group $\&$ the operator group} \label{sec_prismatic_link_group} To relate the operator group to the homotopy $\Zh$-construction, we introduce the prismatic link group. It naturally generalizes the fundamental group of a virtual link to prismatic links. Recall that the \emph{fundamental group} $\pi_1(D)$ of a virtual link diagram $D$ has one generator for every arc of $D$ and one Wirtinger relation for every classical crossing of $D$. See Figure \ref{fig_wirt_rels}.

\begin{figure}[htb]
\begin{tabular}{|ccccccc|} \hline & & & & & & \\
& \begin{tabular}{c}\includegraphics[scale=.6]{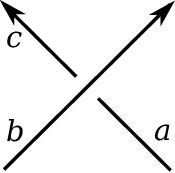} \\ $c=b^{-1}ab$ \end{tabular} & & \begin{tabular}{c}\includegraphics[scale=.6]{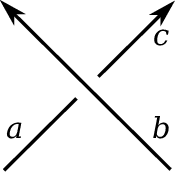} \\ $c=bab^{-1}$ \end{tabular} & & \begin{tabular}{c}\includegraphics[scale=.6]{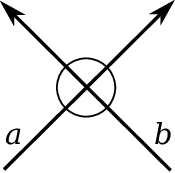}\\ none \end{tabular} & \\& & & & & & \\ \hline 
\end{tabular}
\caption{The Wirtinger relations for the group of a virtual link.} \label{fig_wirt_rels}
\end{figure}

The prismatic link group is the quotient of $\pi_1(D)$ which trivializes the $\Omega$-commutator relation. To define it, some notation must first be established. If $D$ is a $\Omega$-prismatic link diagram and $z \in \Omega$, the $z$-colored component of $D$ has only classical over-crossings or virtual crossings with the $\alpha$-part and only virtual crossings with the $\Omega$-part.  This means that the $z$-colored component is a continuous unbroken arc. If $D$ is also complete, then there is only one such arc for each $z \in \Omega$. Thus, in $\pi_1(D)$, each $z$-colored arc corresponds to a unique generator. This generator of $\pi_1(D)$ will also be denoted by its color $z$. With this convention in place, the prismatic link group can now be defined.

\begin{definition}[Prismatic link group] Let $\Omega=\{x_1,y_1,\ldots,x_g,y_g\}$ be a symplectic palette and $W$ a complete $\Omega$-prismatic link. The $\Omega$-\emph{prismatic group} $\pi_{1}^{\Omega}(W)$ of $W$ is $\pi_1(W)/\langle \prod_{i=1}^g [x_i,y_i] \rangle$. 
\end{definition}

\begin{theorem} If $W_1$, $W_2$ are equivalent $\Omega$-prismatic links, then $\pi_1^{\Omega}(W_1) \cong \pi_1^{\Omega}(W_2)$.
\end{theorem}
\begin{proof} The group of a virtual link is invariant under the welded move \cite{kauffman_vkt}, and hence is also invariant under any semi-welded move, regardless of the color of the over-crossing arc. For the $\Omega$-commutator move, consider traveling along the under-crossing arc $a$ in Figure \ref{fig_commutator_move}. Then the word in the group $\pi_1(W)$ of $W$ that corresponds to the left-hand side of Figure \ref{fig_commutator_move} is:
\[
   y_g^{-1} x_g^{-1} y_g x_g \cdots y_1^{-1} x_1^{-1} y_1 x_1 a x_1^{-1} y_1^{-1} x_1 y_1 \cdots x_g^{-1} y_g^{-1} x_g y_g=\left( \prod_{i=1}^g [x_i,y_i] \right)^{-1} a \left(\prod_{i=1}^g [x_i,y_i] \right)
\]
On the right hand side, the word is simply $a$. Then in $\pi_1^{\Omega}(W)$ this reduces to $a=a$.
\end{proof}

The next result relates the prismatic link group to the operator group. In particular, it implies that $\pi_1(X,(\infty,1))$ and $\pi_1^{\Omega}(\Zh_{\Omega}(L))$ only differ in the order of conjugation in the action of $\pi_1(\Sigma,\infty)$.
\begin{theorem} \label{thm_prismatic_link_group} Let $L \subset \Sigma \times [0,1]$ be a link and $\Omega=\{x_1,y_1,\ldots,x_g,y_g\}$ a symplectic basis for $\Sigma$. A presentation for  $\pi_1^{\Omega}(\Zh_{\Omega}(L))$ can be obtained from $\widetilde{\pi}_L$ by adding the generators $\Omega$, adding the relation $\prod_{i=1}^g [x_i,y_i]$ and replacing $a^{\gamma}$ with $\gamma^{-1} a \gamma$ in every relator.
\end{theorem}

\begin{proof} Let $D$ be a diagram of $L$ on $\Sigma$ and suppose that an arc $c$ of $D$ intersects the symplectic basis element $z \in \Omega$. There are four options for how this can occur, depending on whether $z=x_i,y_i$ and on whether it intersects from left to right or right to left. These are labeled $\raisebox{.5pt}{\textcircled{\raisebox{-.9pt} {1}}}$, $\raisebox{.5pt}{\textcircled{\raisebox{-.9pt} {2}}}$, $\raisebox{.5pt}{\textcircled{\raisebox{-.9pt} {3}}}$, and $\raisebox{.5pt}{\textcircled{\raisebox{-.9pt} {4}}}$ in Figure \ref{fig_zh_and_pi_1}. The corresponding crossings in $\Zh_{\Omega}(D)$ for these four cases are shown in Figure \ref{fig_zh_and_pi_2}. The Wirtinger relations for these crossings are given as follows:
\begin{align*}
\raisebox{.5pt}{\textcircled{\raisebox{-.9pt} {1}}} \quad y_i c=d y_i & \implies d=y_i c y_i^{-1}, \\
\raisebox{.5pt}{\textcircled{\raisebox{-.9pt} {2}}} \quad y_i d=c y_i &\implies d=y_i^{-1} c y_i,\\
\raisebox{.5pt}{\textcircled{\raisebox{-.9pt} {3}}} \quad x_i c=dx_i &\implies d=x_i c x_i^{-1}, \\
\raisebox{.5pt}{\textcircled{\raisebox{-.9pt} {4}}} \quad x_i d=cx_i &\implies d=x_i^{-1} c x_i. 
\end{align*}
Now consider the lift of these pictures to the infinite cyclic cover $\widetilde{\Sigma}$ and let $G$ be the fundamental domain. Write $G^{\gamma}$ for image of $G$ under the action of $\gamma \in \pi_1(\Sigma,\infty)$. The arc  $x_i$ can be viewed as lying in the intersection of $G$ and $G^{y_i}$. The arc $c$ can be thought of as an arc exiting $G$ and the arc $d$ represents the same arc reentering $G$. In case $
\raisebox{.5pt}{\textcircled{\raisebox{-.9pt} {1}}}$, $d$ is entering from $G^{y_i^{-1}}$ and hence it is the continuation of the arc $c^{y_i^{-1}}$ into $G$. Therefore, $d=c^{y_i^{-1}}$. In case $\raisebox{.5pt}{\textcircled{\raisebox{-.9pt} {2}}}$, $d$ is entering from $G^{y_i}$ and hence it is a continuation of the arc $c^{y_i}$. This implies $d=c^{y_i}$. For cases $\raisebox{.5pt}{\textcircled{\raisebox{-.9pt} {3}}}$ and $\raisebox{.5pt}{\textcircled{\raisebox{-.9pt} {4}}}$, a similar argument gives $d=c^{x_i^{-1}}$ and $d=c^{x_i}$, respectively. By the above calculation, these four cases exactly match the Wirtinger relations corresponding to the $\Omega$-colored over-crossing arcs. Hence, making all these substitutions into the presentation $\widetilde{P}$ of $\widetilde{\pi}_L$ gives a presentation for $\pi_1^{\Omega}(\Zh_{\Omega}(D))$. 
\end{proof}

\begin{figure}[htb]
\begin{tabular}{|ccccc|} \hline & & & & \\ 
& \begin{tabular}{c} 
\includegraphics[scale=.7]{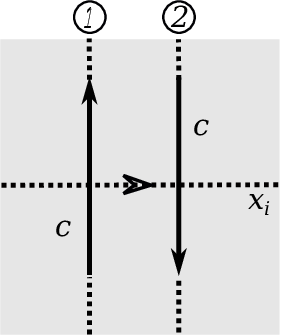}
\end{tabular} & & \begin{tabular}{c} 
\includegraphics[scale=.7]{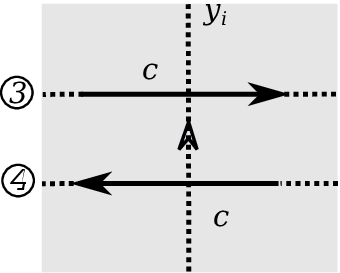}
\end{tabular} & \\ & & & &  \\ \hline
\end{tabular}
\caption{Four ways an arc $c$ of $D$ on $\Sigma$ can intersect a symplectic basis $\Omega$.} \label{fig_zh_and_pi_1}
\end{figure}

\begin{figure}[htb]
\begin{tabular}{|ccc|} \hline & & \\ 
& \includegraphics[scale=.7]{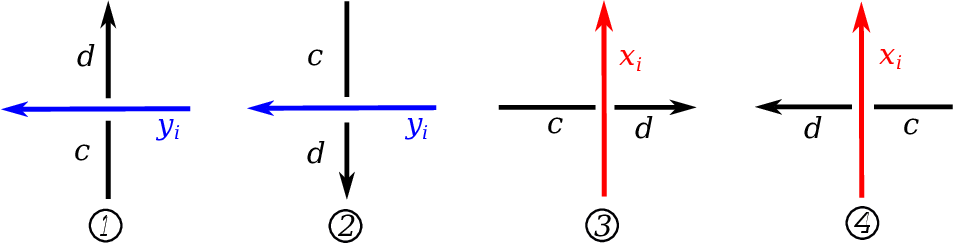} & \\ & & \\ \hline
\end{tabular}
\caption{Appearance of the four cases in Figure \ref{fig_zh_and_pi_1} in $\Zh_{\Omega}(D)$.} \label{fig_zh_and_pi_2}
\end{figure}

\section{The prismatic $U_q(\mathfrak{gl}(m|n))$-Reshetikhin-Turaev functor} \label{sec_uqglmn_rt}

The second part of the bookkeeping method is a $U_q(\mathfrak{gl}(m|n))$ Reshetikhin-Turaev functor. First, in Section \ref{sec_prismatic_tangles}, we give a brief review of categories of virtual tangles and give a generalization of them to the prismatic setting. A bare bones introduction to Lie superalgebras appears in Section \ref{sec_lie_sup_prelim}. The prismatic $U_q(\mathfrak{gl}(m|n))$ Reshetikhin-Turaev functor is defined in Section \ref{sec_uqglmn_rt}.

\subsection{Prismatic tangles} \label{sec_prismatic_tangles} Here we will follow the conventions for virtual tangles from \cite{CP}, Section 2.1. The \emph{category of virtual tangles}, denoted $\mathcal{VT}$, has as objects (possibly empty) sequences in the set of symbols $\{ \boxplus,\boxminus\}$. Morphisms in $\mathcal{VT}$ are \emph{oriented virtual tangle diagrams}. Any virtual tangle diagram can be drawn as a composition and tensor product of the elementary virtual tangles shown in Figure \ref{fig_cross}. Given a virtual tangle diagram, $T$, its domain $\text{dom}(T)$ is the sequence of free ends at the bottom, where a free end is assigned $\boxplus$ if the incident strand is oriented away from the bottom and $\boxminus$ if it is oriented towards the bottom. The codomain $\text{cod}(T)$ is defined likewise as the sequence of free ends at the top of $T$, but with the opposite orientation convention.  Virtual tangle diagrams are considered equivalent up to a sequence of classical Reshetikhin-Turaev moves (see \cite{CP}, Figure 6) and virtual tangle moves. The virtual tangle moves are repeated below in Figure \ref{fig_virtual_tangle_moves}  as they are likely less familiar. $\mathcal{VT}$ is then a monoidal category with the usual tensor product.

\begin{figure}[htb]
\begin{tabular}{|c|cccc|c|c|c|c|} \hline
\multirow[l]{2}*{\rotatebox{90}{ $\stackrel{\text{classical}}{ \text{crossings}}$ \hspace{.35cm}}}  & & & & & & \multirow[c]{2}*{\rotatebox{-90}{\hspace{.05cm} $\stackrel{\text{virtual}}{\text{crossings}}$ }} & & \multirow[c]{2}*{\rotatebox{-90}{ \hspace{-.25cm} $\stackrel{\text{identity}}{\text{morphisms}}$}} \\
& \begin{tabular}{c} 
\def\svgwidth{.5in}
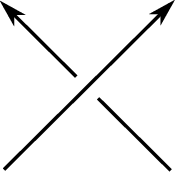 \end{tabular} & & \begin{tabular}{c} \def\svgwidth{.5in}
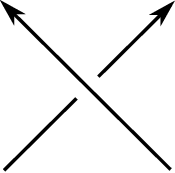 \end{tabular} & & \begin{tabular}{c} \def\svgwidth{.5in}
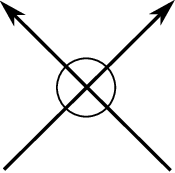 \end{tabular} & & \begin{tabular}{c} \def\svgwidth{.5in}
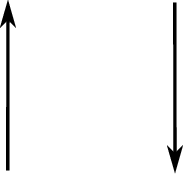 \end{tabular} & \\
& $\oplus$ & & $\ominus$ & & & & & \\\hline 
\multirow[l]{2}*{\rotatebox{90}{ caps }} & \multicolumn{3}{c}{\begin{tabular}{ccc} \def\svgwidth{.5in}
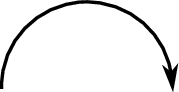 & & \def\svgwidth{.5in}
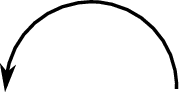 \end{tabular}} & & \multicolumn{3}{c|}{\begin{tabular}{ccc} \def\svgwidth{.5in}
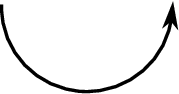 & & \def\svgwidth{.5in}
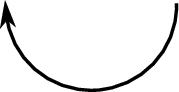  \end{tabular}} & \multirow[c]{2}*{\rotatebox{-90}{\hspace{-.3cm} cups }} \\ & & & & & \multicolumn{3}{c|}{} & \\ \hline
\end{tabular}
\caption{Elementary virtual tangles.} \label{fig_cross}
\end{figure}

\begin{figure}[htb]
    \begin{tabular}{|cc|} \hline & \\ 
\multicolumn{2}{|c|}{\begin{tabular}{cccc} \begin{tabular}{c} \def\svgwidth{1.1in}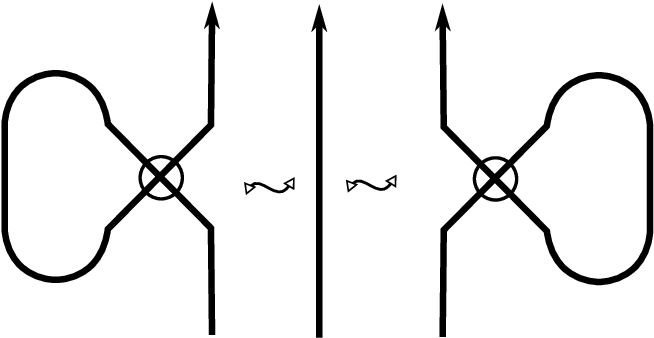 \\ \underline{$VT_1$} \end{tabular} & \begin{tabular}{c} \def\svgwidth{.67in}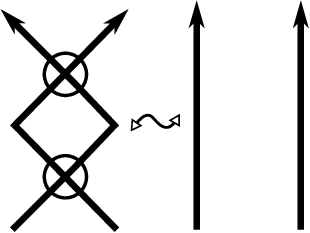 \\ \underline{$VT_2$} \end{tabular} & \begin{tabular}{c} \def\svgwidth{.9in}
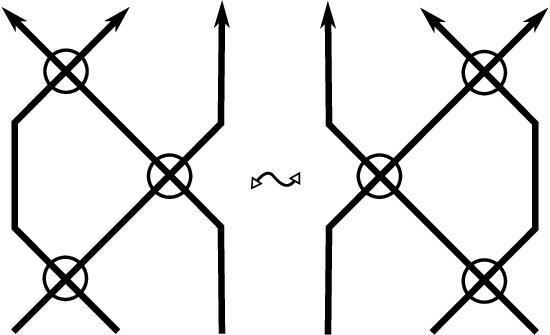 \\ \underline{$VT_3$} \end{tabular} & \begin{tabular}{c} \def\svgwidth{.9in}
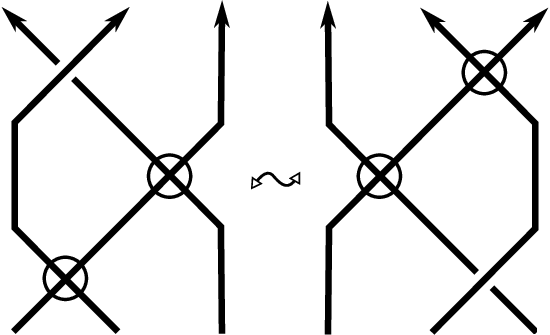 \\ \underline{$VT_4$} \end{tabular} \end{tabular}} \\ & \\
\multicolumn{2}{|c|}{\begin{tabular}{ccc} \begin{tabular}{c} \def\svgwidth{.85in}
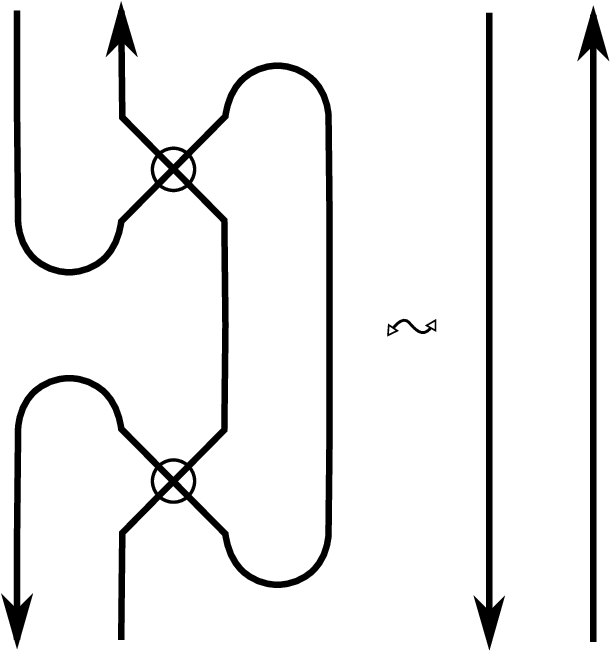 \\ \underline{$VT_5$} \end{tabular} & \begin{tabular}{c} \def\svgwidth{.85in}
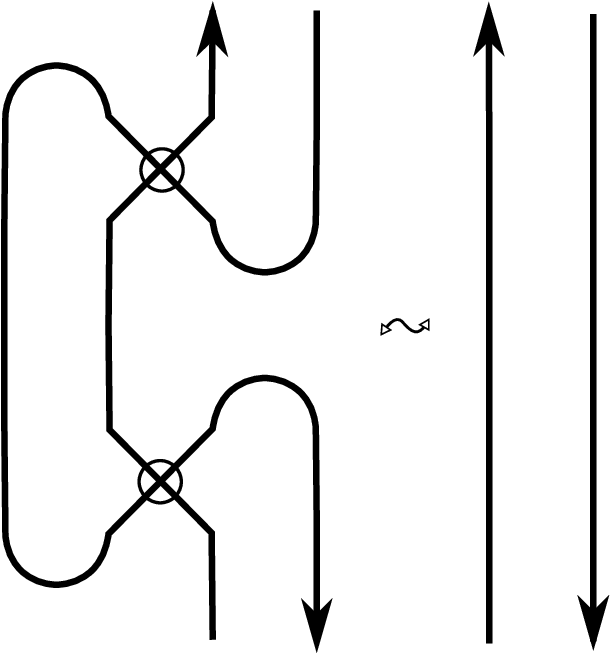  \\ \underline{$VT_6$} \end{tabular} & \begin{tabular}{c} \def\svgwidth{1.6in}
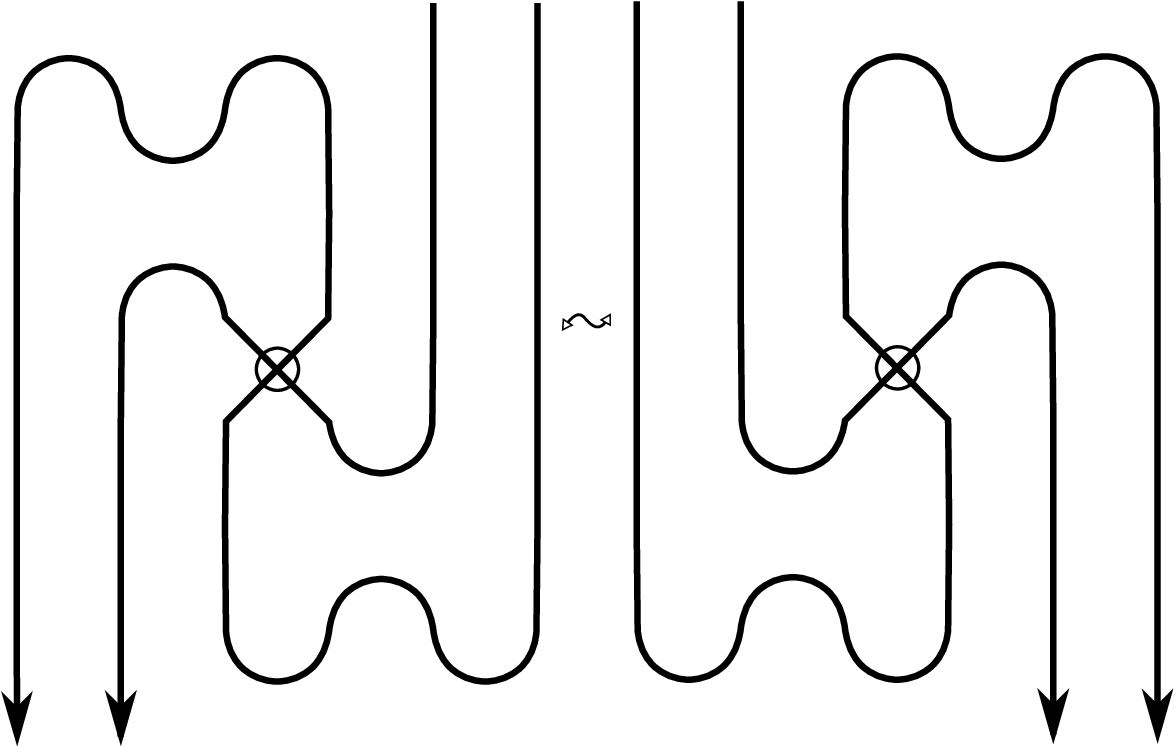 \\ \underline{$VT_7$} \end{tabular} \end{tabular}} \\ & \\ \hline
\end{tabular}
\caption{Virtual tangle moves.} \label{fig_virtual_tangle_moves}
\end{figure}

Let $\Omega$ be a symplectic palette. The objects of the category $\mathcal{PT}_{\Omega}$ of $\Omega$-prismatic tangles the same as those in $\mathcal{VT}$ except that they take colors in the set $\{\alpha\} \cup \Omega$. For $z \in \Omega$, we write $\boxplus^z$, $\boxminus^z$ to denote the color. An \emph{$\Omega$-prismatic tangle diagram} is a virtual tangle diagram $T$ whose components take colors in the set $\{\alpha\} \cup \Omega$ and such that $\text{dom}(T)$, $\text{cod}(T)$ are compatibly colored with the the strands in $T$. Furthermore, the $\Omega$-part of $T$ has only classical over-crossings or virtual crossings with the $\alpha$-part, and only virtual crossings with itself. Morphisms in $\mathcal{PT}_{\Omega}$ are considered equivalent up to classical tangle moves, virtual tangle moves, $\Omega$-semi-welded moves, and $\Omega$-commutator moves. For a set of generating tangle moves for the general $\Omega$-semi-welded move, see \cite{CP}, Figure 12.  

As observed by Kauffman \cite{kauffman_vkt}, additional categories are needed for quantum invariants of virtual tangles. Denote by $\mathcal{VT}^{fr}$ the category of framed virtual tangles, where the Reidemeister 1 move is disallowed but curls of opposite sign cancel. If the move $VT_1$ is disallowed, we have the category of \emph{rotational virtual tangles}, denoted $\mathcal{VT}^{rot}$. Note that the virtual framing relation $T_1^{rot}$ in Figure \ref{fig_framing} is automatically satisfied in $\mathcal{VT}^{rot}$; it is drawn here for emphasis. In the category $\mathcal{VT}^{fr,rot}$ of \emph{framed rotational tangles}, both classical and virtual framing occur. Likewise we have the categories $\mathcal{PT}^{fr},\mathcal{PT}^{rot},\mathcal{PT}^{fr,rot}$ of framed, rotational, and framed rotational prismatic tangles. For prismatic tangles, the disallowed moves in each case only apply to the $\alpha$-part of a prismatic tangle. The move $VT_1$ is allowed on the $\Omega$-part. Reidemeister 1 moves cannot occur on the $\Omega$ part in any event as  classical self-crossings are disallowed by hypothesis. 

\begin{figure}[htb]
\begin{tabular}{|c|c|} \hline & \\
\begin{tabular}{c} $\underline{\mathit{T}_1^{\mathit{rot}}}$:  \end{tabular}   &  \begin{tabular}{c} 
\includegraphics[scale=.4]{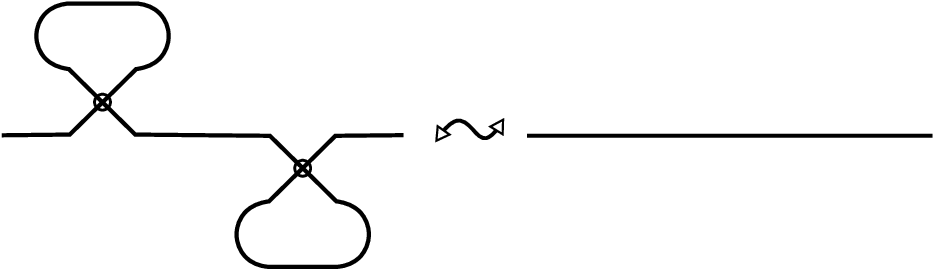} \\ \\ \end{tabular} \\ \hline
\end{tabular}
\caption{The rotational (i.e. virtual framing) relation.} \label{fig_framing}
\end{figure}

\subsection{Preliminaries on Lie superalgebras} \label{sec_lie_sup_prelim} A Lie superalgebra is a $\mathbb{Z}_2$-graded vector space $\mathfrak{g}=\mathfrak{g}_{\bar{0}} \oplus \mathfrak{g}_{\bar{1}}$, together with a bilinear map $[ \cdot, \cdot]: \mathfrak{g} \oplus \mathfrak{g} \to \mathfrak{g}$ (called the \emph{superbracket}) which satisfies the following graded anti-commutativity and graded Jacobi relations:
\begin{eqnarray*}
    [x,y] &=& -(-1)^{|x||y|}[y,x]\\
    {[x,[y,z]]} &=& [[x,y],z] + (-1)^{|x||y|}[y,[x,z]\,],
\end{eqnarray*}
where $x,y$ are homogeneous and $|z| \in \mathbb{Z}_2$ denotes the grading of a homogeneous element $z$. Elements in $\mathfrak{g}_{\bar{0}}$ are said to be \emph{even} and those in $\mathfrak{g}_{\bar{1}}$ are said to be \emph{odd}. Let $\mathbb{C}^{m|n}$ be the graded complex vector space whose even subspace has basis $x_1,\ldots,x_m$ and whose odd subspace has basis $x_{m+1},\ldots,x_{m+n}$. The vector space of endomorphisms of $\mathbb{C}^{m|n}$, denoted $\mathfrak{gl}(m|n)$, has the structure of a Lie superalgebra. The even and odd parts take the following forms:
\[
\underline{\text{Even}:}\,\, \left(\begin{array}{c|c}
    A & \textbf{0}_{m \times n} \\ \hline
    \textbf{0}_{n \times m} & D
\end{array}\right), \quad \underline{\text{Odd}:}\
\,\, \left(\begin{array}{c|c}
    \textbf{0}_{m \times m} & B\\ \hline
    C & \textbf{0}_{n\times n}
\end{array}\right),
\]
where $\textbf{0}_{j,k}$ denotes the $j \times k$ matrix of all zeros. The superbracket is: $[X,Y]=XY-(-1)^{|X||Y|}YX$. Analogous to the Lie algebra case, the universal enveloping algebra of $\mathfrak{gl}(m|n)$ can be deformed to obtain the \emph{quantized universal enveloping superalgebra} $U_q(\mathfrak{gl}(m|n))$ over the field $\mathbb{C}(q)$ \cite{zhang,zhang_2}. By the \emph{vector representation} $V$ of $U_q(\mathfrak{gl}(m|n))$, we mean the representation inherited from the defining action of $\mathfrak{gl}(m|n)$ on $\mathbb{C}^{m|n}$. The superalgebra $U_q(\mathfrak{gl}(m|n))$ furthermore has the structure of a \emph{ribbon Hopf superalgebra}. In particular, it has a ribbon element and a universal $R$-matrix. For our purposes, only the values for the cup, cap, and crossing maps for the vector representation of $U_q(\mathfrak{gl}(m|n))$ will be required. These can be easily found in the literature (see e.g. \cite{queffelec_19,queffelec_sartori}).

\subsection{The prismatic $U_q(\mathfrak{gl}(m|n))$-Reshetikhin-Turaev functor} \label{sec_prismatic_rt_defn} The prismatic Reshetikhin-Turaev functor is defined analogously to the extended Reshetikhin-Turaev functor of semi-welded tangles from \cite{CP}. Let $\Omega=\{x_1,y_1,\ldots,x_g,y_g\}$ be a symplectic palette and let $m,n \in \mathbb{N}$. Define $\mathbb{F}=\mathbb{C}(q,x_1,y_1\ldots,x_g,y_g)$ to be the field of rational functions in $q,x_1,y_1,\ldots,x_g,y_g$ over $\mathbb{C}$. The vector representation $V$ of $U_q(\mathfrak{gl}(m|n))$ over $\mathbb{C}(q)$ extends to a representation over $\mathbb{F}$, which we will also denote by $V$. The \emph{prismatic $U_q(\mathfrak{gl}(m|n))$ Reshetikhin-Turaev functor $\widetilde{Q}_{\Omega}^{m|n}:\mathcal{PT}_{\Omega}^{fr,rot} \to \textbf{Vect}_{\mathbb{F}}$} is defined on objects as follows. First set $\widetilde{Q}_{\Omega}^{m|n}(\varnothing)=\mathbb{F}$. Then for objects $(\varepsilon)$ of length $1$, define:
\[
\widetilde{Q}_{\Omega}^{m|n}((\varepsilon))=\left\{\begin{array}{cl} V & \text{if } \varepsilon=\boxplus \\ V^* & \text{if } \varepsilon=\boxminus \\ \mathbb{F} & \text{if } \varepsilon=\boxplus^{z} \text{ or } \boxminus^{z}, \text{ and } z \in \Omega \end{array} \right.
\] 
Note that for objects in the $\alpha$-part, we are using the vector representation but for objects in the $\Omega$-part, we are using a $1$-dimensional representation. For an object $a=(\varepsilon_1, \ldots,\varepsilon_k)$, define $\widetilde{Q}_{\Omega}^{m|n}(a)$ to be the tensor product of the representations for each $\varepsilon_i$. Next we define $\widetilde{Q}_{\Omega}^{m|n}(T)$ for $T$ an elementary virtual tangle. For elementary tangles involving only that classical $\alpha$-part, we use the same values as the $U_q(\mathfrak{gl}(m|n))$ Reshetikhin-Turaev functor on classical tangles (see e.g. Queffelec \cite{queffelec_19}).
\begin{align*} \label{eqn_r_matrices}
\widetilde{Q}^{m|n}_{\Omega} \left(\begin{array}{c} \includegraphics[height=.35in]{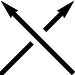} \end{array}\right) & : x_i \otimes x_j \to \left\{ \begin{array}{cl} 
q x_i \otimes x_i                                  & \text{if } i=j \le m      \\
(-1)^{|i||j|} x_j \otimes x_i+(q-q^{-1}) x_i \otimes x_j                           & \text{if } i<j         \\
(-1)^{|i||j|} x_j \otimes x_i & \text{if } i>j \\
-q^{-1} x_i \otimes x_i                                      & \text{if } i=j>m           
\end{array} \right. \\
\widetilde{Q}^{m|n}_{\Omega} \left(\begin{array}{c} \includegraphics[height=.35in]{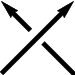} \end{array} \right)& : x_i \otimes x_j \to \left\{ \begin{array}{cl} 
q^{-1} x_i \otimes x_i                                  & \text{if } i=j \le m      \\
(-1)^{|i||j|} x_j \otimes x_i                           & \text{if } i<j         \\
(-1)^{|i||j|} x_j \otimes x_i+(q^{-1}-q)x_i \otimes x_j & \text{if } i>j \\
-q x_i \otimes x_i                                      & \text{if } i=j>m           
\end{array} \right. 
\end{align*}
\begin{align*}
\widetilde{Q}_{\Omega}^{m|n}\left(\begin{array}{c}  \includegraphics[height=.25in]{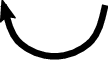} \end{array}\right) &: 1 \to \sum_{k=1}^{m+n} x_k \otimes x_k^*,\\  \widetilde{Q}_{\Omega}^{m|n}\left(\begin{array}{c}  \includegraphics[height=.25in]{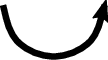} \end{array}\right) & : 1 \to q^{m-n} \left( \sum_{k=1}^{m} q^{1-2k} x_k^* \otimes x_k -\sum_{k=m+1}^{m+n} q^{-4m-1+2k} x_k^* \otimes x_k \right) 
\end{align*}
\[
\widetilde{Q}_{\Omega}^{m|n}\left(\begin{array}{c}  \includegraphics[height=.25in]{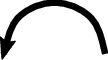} \end{array}\right) : x_k^* \otimes x_k \to 1, \widetilde{Q}_{\Omega}^{m|n}\left( \begin{array}{c} \includegraphics[height=.25in]{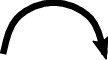} \end{array}\right) : x_k \otimes x_k^* \to \left\{\begin{array}{cl} q^{-m+n-1+2k} & \text{if } k \le m \\ -q^{3m+n+1-2k} & \text{if } k >m \end{array} \right.
\]
These are repeated from \cite{CP}, but with values of positive and negative crossings switched. This is to match the conventions of the CKS polynomial in the next section. In \cite{CP}, the values were chosen instead to match those of  Jaeger, Kauffman, and Saleur \cite{kauffman_saleur_91, kauffman_saleur_92, jaeger_kauffman_saleur_94}. For the virtual crossing, we will use the \emph{same} value as the extended $U_q(\mathfrak{gl}(m|n))$ functor from \cite{CP}, which is chosen to match the Alexander polynomial of an almost classical knot when $m=n=1$ (see \cite{CP}, Remarks 3.2.1-3.2.3).
\begin{align*} 
\widetilde{Q}^{m|n}_{\Omega} \left(\begin{array}{c} \includegraphics[height=.35in]{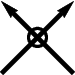} \end{array}\right) & : x_i \otimes x_j \to \left\{ \begin{array}{cl} x_j \otimes x_i & \text{if } i,j \le m \\ 
 q x_j \otimes x_i & \text{if } j \le m, i>m \\
 q^{-1} x_j \otimes x_i & \text{if } i \le m, j>m \\
-x_j \otimes x_i & \text{if } i,j > m \\
\end{array} \right.
\end{align*}
For the classical crossings involving the $\Omega$-part, the associated value must be a linear transformation $V \to V$, since $V \otimes \mathbb{F} \cong V \cong \mathbb{F} \otimes V$. Following \cite{CP}, we define the action of the $\Omega$-part to be multiplication on the odd subspace by a scalar. Then for each $z \in \Omega$, we have the following assignments:
\[
\widetilde{Q}^{m|n}_{\Omega}\left(\begin{array}{c} \includegraphics[height=.35in]{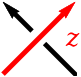} \end{array}\right) :x_i \to \left\{ \begin{array}{cl} x_i & \text{if } i \le m \\ z^{-1} \cdot x_i & \text{if } i>m \end{array} \right., \quad 
\widetilde{Q}^{m|n}_{\Omega}\left(\begin{array}{c} \includegraphics[height=.35in]{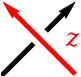} \end{array}\right):x_i \to \left\{ \begin{array}{cl} x_i & \text{if } i \le m \\ z \cdot x_i & \text{if } i>m \end{array} \right.,
\]
Next consider a virtual crossing involving the $\Omega$-part. From the topological perspective, it is expected that only classical over-crossings should record the action of $\pi_1(\Sigma,\infty)$ on the lift of the diagram to the universal cover. Such crossings should therefore be assigned to an identity matrix:
\[ \widetilde{Q}^{m|n}_{\Omega}\left(\begin{array}{c} \includegraphics[height=.35in]{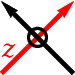} \end{array} \text{ or }\begin{array}{c} \includegraphics[height=.35in]{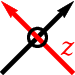} \end{array} \right) \,\,: x_i \to x_i, \quad \widetilde{Q}^{m|n}_{\Omega}\left(
\begin{array}{c} \includegraphics[height=.35in]{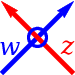} \end{array} \right)\,\, : 1 \to 1
\] 
Lastly, for the evaluation and coevaluation maps, note that $\mathbb{F} \otimes \mathbb{F}^* \cong \mathbb{F}$. These are thus homomorphisms $\mathbb{F} \to \mathbb{F}$. Hence, all such maps are also assigned to the appropriate identity map:    \[
\widetilde{Q}^{m|n}_{\Omega}\left(\begin{array}{c} \includegraphics[height=.25in]{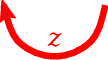} \end{array} \text{ or }  
\begin{array}{c} \includegraphics[height=.25in]{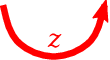} \end{array}\right): 1 \to 1 \otimes 1, \quad \widetilde{Q}^{m|n}_{\Omega}\left( \begin{array}{c} \includegraphics[height=.25in]{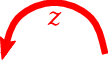} \end{array} \text{ or } 
\begin{array}{c} \includegraphics[height=.25in]{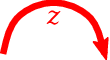} \end{array} \right): 1 \otimes 1 \to 1,
\]
\begin{theorem}\label{thm_mv_functor_is_wd} $\widetilde{Q}^{m|n}_{\Omega}:\mathcal{PT}_{\Omega}^{fr,rot} \to \textbf{Vect}_{\mathbb{F}}$ is a well-defined functor of $\Omega$-prismatic tangles in the framed rotational category. If $m=n=1$, this holds for $\mathcal{PT}_{\Omega}$, up to multiples of $q^k$.
\end{theorem} 
\begin{proof} This is essentially the same as the proof for the extended $U_q(\mathfrak{gl}(m|n))$ Reshetikhin-Turaev functor from \cite{CP}, which can be viewed as a single color version (i.e. $\Omega=\{\textcolor{blue}{\textbf{\text{blue}}}\}$) of the prismatic functor. In particular, the $\Omega$-semi-welded move is satisfied for any color of the over-crossing arc. For details of the proofs of invariance under the various moves, see \cite{CP}, Section 3.4. The only rule that needs to be checked is the $\Omega$-commutator move. Going from left to right in the left hand side of Figure \ref{fig_commutator_move}, the effect is to multiply the odd part of $V$ by $x_1 y_1 x_1^{-1} y_1^{-1} \cdot \ldots \cdot x_g y_g x_g^{-1} y_g^{-1}$. In $\mathbb{F}$, this simplifies to $1$. Hence, the $\Omega$-commutator move is satisfied for all $m,n$. If $m=n=1$, $\dim(V)=2$. A classical curl evaluates to the identity matrix $I_2$ and a virtual curl evaluates to $q^{\pm 1} I_2$ (see again \cite{CP}, Section 3.4). Hence we have a true invariant of virtual tangles, up to multiples of $q^k$.    
\end{proof}

\section{The prismatic $U_q(\mathfrak{gl}(1|1))$ invariant and the CSW polynomial} \label{sec_uqgl11_csw}

Having established the two parts of the bookkeeping method, we proceed to showing that their composition is equivalent to the CSW polynomial when $m=n=1$. This is accomplished by generalizing Kauffman and Saleur's proof identifying the classical Alexander polynomial with the $U_q(\mathfrak{gl}(1|1))$ quantum invariant \cite{kauffman_saleur_91}. First, in Section \ref{sec_prismatic_braid_group}, we define a prismatic extension of the virtual braid group. In Section \ref{sec_prismatic_burau}, we define a representation of this prismatic braid group, called the prismatic Burau representation. Finally, in Section \ref{sec_quantum_CSW}, we identify the CSW polynomial and the prismatic $U_q(\mathfrak{gl}(1|1))$ invariant via the exterior power of the prismatic Burau representation.

\subsection{The prismatic braid group} \label{sec_prismatic_braid_group} The $\Omega$-prismatic braid group extends the virtual braid group by adding one strand for each color in $\Omega$. First, let's recall the definition of the virtual braid group $VB_N$ on $N$ strands. The generators $\sigma_i$, $\chi_i$ of $\mathit{VB}_N$ are are shown in Figure \ref{fig_braid_gen}. The relations are:
\begin{align} 
\label{rel_vbr1} \text{(Classical Relations)} \quad  & \left\{\begin{array}{cl} \sigma_i \sigma_i^{-1}=\sigma_i^{-1}\sigma_i=1 & \\
\sigma_i \sigma_j=\sigma_j \sigma_i &  \text{if } |i-j|>1 \\
\sigma_i \sigma_{i+1} \sigma_i=\sigma_{i+1} \sigma_i \sigma_{i+1}
  \end{array} \right. \\
\label{rel_vbr2} \text{(Virtual Relations)} \quad  & \left\{\begin{array}{cl} \chi_i^2=1 & \\
\chi_i \chi_j=\chi_j \chi_i &  \text{if } |i-j|>1 \\
\chi_i \chi_{i+1} \chi_i=\chi_{i+1} \chi_i \chi_{i+1}
  \end{array} \right. \\ 
\label{rel_vbr3} \text{(Mixed Relations)} \quad  & \left\{\begin{array}{cl}
\sigma_i \chi_j=\chi_j \sigma_i &  \text{if } |i-j|>1 \\
\sigma_i \chi_{i+1} \chi_i=\chi_{i+1} \chi_i \sigma_{i+1}
  \end{array} \right. 
\end{align}
Now, let $\Omega=\{x_1,y_1,\ldots,x_g,y_g\}$ be a fixed symplectic palette. Let $\beta \in \mathit{VB}_{N+2g}$ and let $\pi_{\beta} \in \mathbb{S}_{N+2g}$ be the permutation of the strands $\{1,\ldots,N,N+1,\ldots N+2g\}$ defined by $\beta$. Suppose that the restriction of $\pi_{\beta}$ to $\{N+1,\ldots,N+2g\}$ is the identity. In other words, $\beta$ is pure on the last $2g$ strands. We then consider the last $2g$ strands to be colored, left to right, as $x_1,y_1,\ldots,x_g,y_g$. These will be called the \emph{$\Omega$-strands} while the first $N$ strands will be called that \emph{$\alpha$-strands}. Suppose furthermore that the $\Omega$-strands have only classical over-crossings or virtual crossings with the $\alpha$-strands and only virtual crossings with other $\Omega$-strands. Then we will say that $\beta$ is an \emph{$\Omega$-prismatic braid on $N$ strands}. When applying relations (\ref{rel_vbr1}), (\ref{rel_vbr2}), and (\ref{rel_vbr3}) to $\Omega$-prismatic braids, only moves in which both sides of the equality are $\Omega$-prismatic braids are allowed. Also allowed between $\Omega$-prismatic braids are $\Omega$-semi-welded moves (Figure \ref{fig_sw_move_and_zh}) and $\Omega$-commutator moves (Figure \ref{fig_commutator_move}). All together, these moves generate the equivalence relation for $\Omega$-prismatic braids. 

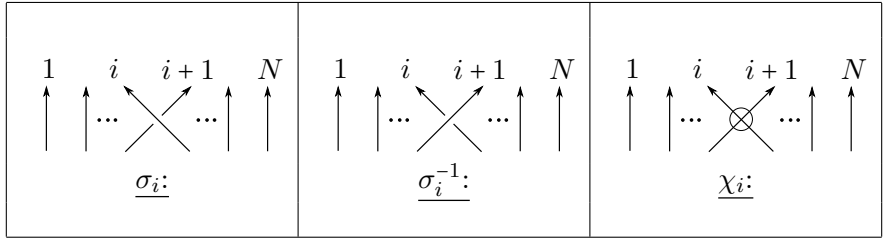
\begin{figure}[htb]
\begin{tabular}{|c|c|c|} \hline & & \\
\begin{tabular}{c}\\ \def\svgwidth{1.2in}
\begingroup%
  \makeatletter%
  \providecommand\color[2][]{%
    \errmessage{(Inkscape) Color is used for the text in Inkscape, but the package 'color.sty' is not loaded}%
    \renewcommand\color[2][]{}%
  }%
  \providecommand\transparent[1]{%
    \errmessage{(Inkscape) Transparency is used (non-zero) for the text in Inkscape, but the package 'transparent.sty' is not loaded}%
    \renewcommand\transparent[1]{}%
  }%
  \providecommand\rotatebox[2]{#2}%
  \newcommand*\fsize{\dimexpr\f@size pt\relax}%
  \newcommand*\lineheight[1]{\fontsize{\fsize}{#1\fsize}\selectfont}%
  \ifx\svgwidth\undefined%
    \setlength{\unitlength}{291.82369523bp}%
    \ifx\svgscale\undefined%
      \relax%
    \else%
      \setlength{\unitlength}{\unitlength * \real{\svgscale}}%
    \fi%
  \else%
    \setlength{\unitlength}{\svgwidth}%
  \fi%
  \global\let\svgwidth\undefined%
  \global\let\svgscale\undefined%
  \makeatother%
  \begin{picture}(1,0.35400402)%
    \lineheight{1}%
    \setlength\tabcolsep{0pt}%
    \put(0,0){\includegraphics[width=\unitlength]{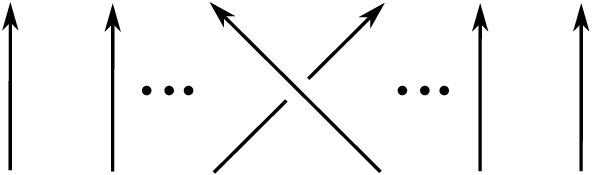}}%
    \put(0.93743895,0.32293176){\color[rgb]{0,0,0}\makebox(0,0)[lt]{\lineheight{40.54999924}\smash{\begin{tabular}[t]{l}$N$\end{tabular}}}}%
    \put(0.52,0.32125272){\color[rgb]{0,0,0}\makebox(0,0)[lt]{\lineheight{40.54999924}\smash{\begin{tabular}[t]{l}$i+1$\end{tabular}}}}%
    \put(0.29099413,0.32628995){\color[rgb]{0,0,0}\makebox(0,0)[lt]{\lineheight{40.54999924}\smash{\begin{tabular}[t]{l}$i$\end{tabular}}}}%
    \put(-0.00284444,0.32796899){\color[rgb]{0,0,0}\makebox(0,0)[lt]{\lineheight{40.54999924}\smash{\begin{tabular}[t]{l}$1$\end{tabular}}}}%
  \end{picture}%
\endgroup%
 \\ 
\end{tabular} & \begin{tabular}{c} \\ \def\svgwidth{1.2in}
\begingroup%
  \makeatletter%
  \providecommand\color[2][]{%
    \errmessage{(Inkscape) Color is used for the text in Inkscape, but the package 'color.sty' is not loaded}%
    \renewcommand\color[2][]{}%
  }%
  \providecommand\transparent[1]{%
    \errmessage{(Inkscape) Transparency is used (non-zero) for the text in Inkscape, but the package 'transparent.sty' is not loaded}%
    \renewcommand\transparent[1]{}%
  }%
  \providecommand\rotatebox[2]{#2}%
  \newcommand*\fsize{\dimexpr\f@size pt\relax}%
  \newcommand*\lineheight[1]{\fontsize{\fsize}{#1\fsize}\selectfont}%
  \ifx\svgwidth\undefined%
    \setlength{\unitlength}{291.82369523bp}%
    \ifx\svgscale\undefined%
      \relax%
    \else%
      \setlength{\unitlength}{\unitlength * \real{\svgscale}}%
    \fi%
  \else%
    \setlength{\unitlength}{\svgwidth}%
  \fi%
  \global\let\svgwidth\undefined%
  \global\let\svgscale\undefined%
  \makeatother%
  \begin{picture}(1,0.35400402)%
    \lineheight{1}%
    \setlength\tabcolsep{0pt}%
    \put(0,0){\includegraphics[width=\unitlength]{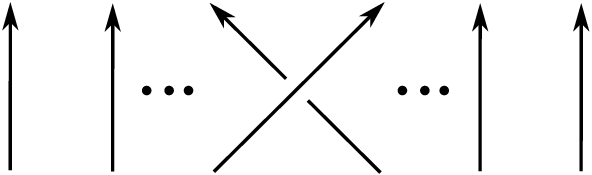}}%
    \put(0.93743895,0.32293176){\color[rgb]{0,0,0}\makebox(0,0)[lt]{\lineheight{40.54999924}\smash{\begin{tabular}[t]{l}$N$\end{tabular}}}}%
    \put(0.52,0.32125272){\color[rgb]{0,0,0}\makebox(0,0)[lt]{\lineheight{40.54999924}\smash{\begin{tabular}[t]{l}$i+1$\end{tabular}}}}%
    \put(0.29099413,0.32628995){\color[rgb]{0,0,0}\makebox(0,0)[lt]{\lineheight{40.54999924}\smash{\begin{tabular}[t]{l}$i$\end{tabular}}}}%
    \put(-0.00284444,0.32796899){\color[rgb]{0,0,0}\makebox(0,0)[lt]{\lineheight{40.54999924}\smash{\begin{tabular}[t]{l}$1$\end{tabular}}}}%
  \end{picture}%
\endgroup%
  \\ \end{tabular} & \begin{tabular}{c} \\  \def\svgwidth{1.2in}
\begingroup%
  \makeatletter%
  \providecommand\color[2][]{%
    \errmessage{(Inkscape) Color is used for the text in Inkscape, but the package 'color.sty' is not loaded}%
    \renewcommand\color[2][]{}%
  }%
  \providecommand\transparent[1]{%
    \errmessage{(Inkscape) Transparency is used (non-zero) for the text in Inkscape, but the package 'transparent.sty' is not loaded}%
    \renewcommand\transparent[1]{}%
  }%
  \providecommand\rotatebox[2]{#2}%
  \newcommand*\fsize{\dimexpr\f@size pt\relax}%
  \newcommand*\lineheight[1]{\fontsize{\fsize}{#1\fsize}\selectfont}%
  \ifx\svgwidth\undefined%
    \setlength{\unitlength}{291.82369523bp}%
    \ifx\svgscale\undefined%
      \relax%
    \else%
      \setlength{\unitlength}{\unitlength * \real{\svgscale}}%
    \fi%
  \else%
    \setlength{\unitlength}{\svgwidth}%
  \fi%
  \global\let\svgwidth\undefined%
  \global\let\svgscale\undefined%
  \makeatother%
  \begin{picture}(1,0.35400402)%
    \lineheight{1}%
    \setlength\tabcolsep{0pt}%
    \put(0,0){\includegraphics[width=\unitlength]{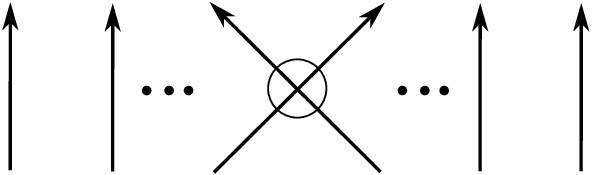}}%
    \put(0.93743895,0.32293176){\color[rgb]{0,0,0}\makebox(0,0)[lt]{\lineheight{40.54999924}\smash{\begin{tabular}[t]{l}$N$\end{tabular}}}}%
    \put(0.52,0.32125272){\color[rgb]{0,0,0}\makebox(0,0)[lt]{\lineheight{40.54999924}\smash{\begin{tabular}[t]{l}$i+1$\end{tabular}}}}%
    \put(0.29099413,0.32628995){\color[rgb]{0,0,0}\makebox(0,0)[lt]{\lineheight{40.54999924}\smash{\begin{tabular}[t]{l}$i$\end{tabular}}}}%
    \put(-0.00284444,0.32796899){\color[rgb]{0,0,0}\makebox(0,0)[lt]{\lineheight{40.54999924}\smash{\begin{tabular}[t]{l}$1$\end{tabular}}}}%
  \end{picture}%
\endgroup%
 \\ \end{tabular} \\  \underline{$\sigma_i$:} & \underline{$\sigma_i^{-1}$:} & \underline{$\chi_i$:} \\ & & \\ \hline
\end{tabular}
\caption{Generators of the virtual braid group $\mathit{VB}_N$ on $N$ strands.} \label{fig_braid_gen}
\end{figure} 

\begin{definition}[$\Omega$-Prismatic braid group on $N$ strands] The \emph{$\Omega$-prismatic braid group} $\mathit{VB}_{N,\Omega}$ is the set of equivalence classes of $\Omega$-prismatic braids on $N$ strands. For $\beta_1,\beta_2 \in \mathit{VB}_{N,\Omega}$ the operation $\beta_2 \circ \beta_1$ is defined to be $\beta_2$ stacked on top of $\beta_1$.
\end{definition}

A generating set for $\mathit{VB}_{N,\Omega}$ can be described as follows. Let $1_{2g}$ denote the the identity braid on $2g$ strands. For $\gamma \in \mathit{VB}_N$, let $\gamma \otimes 1_{2g} \in \mathit{VB}_{N+2g}$ denote the virtual braid by the horizontal juxtaposition of $1_{2g}$ on the right of $\gamma$. Then for $1 \le i \le N-1$, we have that $\sigma_i^{\pm 1} \otimes 1_{2g}, \chi_i \otimes 1_{2g} \in \mathit{VB}_{N,\Omega}$. Clearly, any $\Omega$-prismatic braid whose $\Omega$-strands have no classical crossings with the $\alpha$-strands are in the subgroup generated by $\sigma_i^{\pm 1} \otimes 1_{2g}, \chi_i \otimes 1_{2g}$. To obtain the remaining braids, it suffices to include as generators the $\Omega$-prismatic braids $\lambda_{j,w_k}$ shown in Figure \ref{fig_lambda_gens}, where $1 \le j \le N$ and $w_k \in \Omega$. Indeed, if the $j$-th strand of $\beta$ is over-crossed by the $w_k$-strand, the $w_k$-strand can be diverted nearby the crossing so that it appears as exactly as in the left or right of Figure \ref{fig_lambda_gens}. This can be done using only the relations (\ref{rel_vbr2}). Since the $w_k$-strand is pure, it can be assumed to be stationary between all such classical over-crossings. Thus, $\mathit{VB}_{N,\Omega}$ is generated by $\sigma_i\otimes 1_{2g}$, $\chi_i \otimes 1_{2g}$,$\lambda_{i,w_k}$, $1 \le i \le N$, $w_k \in \Omega$. 

\begin{figure}[htb]
\begin{tabular}{|cccc|} \hline & & & \\
& \begin{tabular}{c}
\includegraphics[scale=.7]{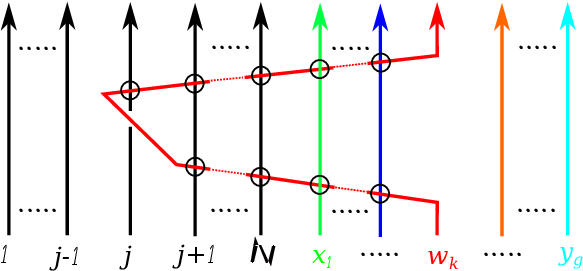} \\ \underline{$\lambda_{j,w_k}$:}
\end{tabular} &
\begin{tabular}{c}
\includegraphics[scale=.7]{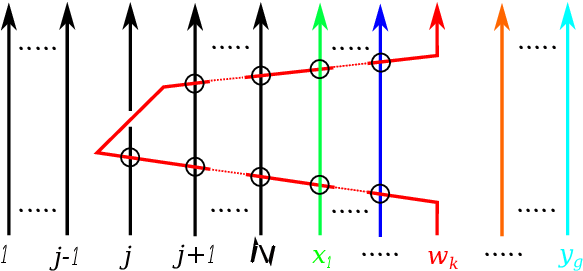} \\
\underline{$\lambda_{j,w_k}^{-1}$:}
\end{tabular} &
\\ & & & \\ \hline
\end{tabular}
\caption{The generators $\lambda_{j,w_k}$ of $\mathit{VB}_{N,\Omega}$ for $1 \le j \le N$ and $w_k \in \Omega$.} \label{fig_lambda_gens}
\end{figure}

\begin{theorem} \label{thm_prismatic_braid_closure} Every complete $\Omega$-prismatic link is the closure of an $\Omega$-prismatic braid.  
\end{theorem}
\begin{proof} Let $L$ be a complete $\Omega$-prismatic link diagram. Let $A$ denote the $\alpha$-part of $L$. By Kamada \cite{kamada_v_braid}, Proposition 3, there is an $N \ge 1$ and virtual braid $\beta \in \mathit{VB}_N$ such that $A$ is equivalent to $\widehat{\beta}$. Here we will take $\widehat{\beta}$ to be the right closure of $\beta$ (see Figure \ref{fig_pt_at_infty}). Furthermore, it may be assumed from \cite{kamada_v_braid} that $A$ and $\widehat{\beta}$ have the same Gauss diagram. Hence, the classical over-crossing arcs of the $\Omega$-part of $L$ can be drawn in the same positions on arcs of $\beta$. These $\Omega$-colored arcs can be assumed to be drawn at different heights. Denote this new partial diagram by $\vec{\beta}$. Now consider $\vec{\beta} \otimes 1_{2g}$. Let $\widecheck{b}$ denote the first $\Omega$-colored over-crossing arc in $\vec{\beta}$ that appears when traveling from bottom to top of $\vec{\beta}$. If $\widecheck{\beta}$ is colored $w_k$, we divert the $w_k$-strand of $\vec{\beta} \otimes 1_{2g}$ as shown in Figure \ref{fig_lambda_gens}, according to the sign of the crossing. Moving upwards, continue this process until all of the $\Omega$-colored arcs of $\vec{\beta}$ are accounted for. Let $\beta' \in \mathit{VB}_{N,\Omega}$ denote the final result. By Lemma \ref{lemma_rearrange},  $\widehat{\beta'}$ is equivalent to $L$.
\end{proof}
\subsection{The prismatic Burau representation} \label{sec_prismatic_burau} Recall that to any $\beta \in \mathit{VB}_N$, we can associate an automorphism of the free group $F_N$ on $N$ letters $s_1,\ldots,s_N$. For $\sigma_i$, $\chi_i$, the automorphisms are given by (see e.g. Bardakov-Bellingeri \cite{bardakov_bellingeri}): 
\begin{align} \label{eqn_virt_braid_auto}
\sigma_i: \left\{ \begin{array}{ccc} s_i & \longrightarrow & s_i s_{i+1} s_i^{-1} \\ s_{i+1} & \longrightarrow & s_i \\ s_k & \longrightarrow & s_k \,\, (k \ne i,i+1) \end{array} \right., &\quad \chi_i: \left\{ \begin{array}{ccc} s_i & \longrightarrow &  s_{i+1} \\ s_{i+1} & \longrightarrow & s_i \\ s_k & \longrightarrow & s_k \,\, (k \ne i,i+1) \end{array} \right. 
\end{align} 
Now, let $\Omega=\{x_1,y_1,\ldots,x_g,y_g\}$ be a symplectic palette and let $F_{N+2g}$ be the free group on the letters $s_1,\ldots,s_N,x_1,y_1\ldots,x_g,y_g$. For the generators of $\mathit{VB}_{N,\Omega}$, it follows from Equation \ref{eqn_virt_braid_auto} that:
\begin{align}  
\sigma_i \otimes 1_{2g} &: \left\{ \begin{array}{ccc} s_i & \longrightarrow & s_i s_{i+1} s_i^{-1} \\ s_{i+1} & \longrightarrow & s_i \\ s_k & \longrightarrow & s_k \,\, (k \ne i,i+1) \\ w_k & \longrightarrow & w_k \,\,  (w_k \in \Omega) \end{array} \right., \quad \chi_i \otimes 1_{2g}: \left\{ \begin{array}{ccc} s_i & \longrightarrow &  s_{i+1} \\ s_{i+1} & \longrightarrow & s_i \\ s_k & \longrightarrow & s_k \,\, (k \ne i,i+1) \\ w_k & \longrightarrow & w_k \,\,  (w_k \in \Omega) \end{array} \right., \\ \label{eqn_lambda} \lambda_{i,w_j}&: \left\{ \begin{array}{ccc} s_i & \longrightarrow & w_j s_{i} w_j^{-1} \\ s_k & \longrightarrow & s_k \,\, (k \ne i) \\ w_k & \longrightarrow & w_k \,\,  (w_k \in \Omega) \end{array}\right., \quad \quad \quad\quad \lambda_{i,w_j}^{-1}:\left\{\begin{array}{ccc} s_i & \longrightarrow & w_j^{-1} s_{i} w_j \\ s_k & \longrightarrow & s_k \,\, (k \ne i) \\ w_k & \longrightarrow & w_k \,\,  (w_k \in \Omega) \end{array} \right. .
\end{align}
Then every $\beta \in \mathit{VB}_{N,\Omega}$ gives an element of $\text{Aut}(F_{N+2g})$, also denoted $\beta$. The Jacobian formula for computing the CSW polynomial (see Theorem \ref{thm_csw_poly_jacobian}) now suggests the following definition for a Burau representation of $\mathit{VB}_{N,\Omega}$. Below, $\phi:\mathbb{Z}[F_{N+2g}] \to \mathbb{Z}[\mathbb{Z} \times H_1(\Sigma)]$ is defined by $\phi(s_i)=t$ for all $i$ and $\phi(w_i)=w_i$ for all $w_i \in \Omega$. The maps $\partial/\partial s_i$ denote the usual Fox derivative (see Section 
\ref{sec_csw_poly}).
\begin{definition}[$\Omega$-Prismatic Burau representation]
The \emph{$\Omega$-prismatic Burau representation} $\rho_{N,\Omega}:\mathit{VB}_{N,\Omega} \to GL(N, \mathbb{C}(t,x_1,y_1,\ldots,x_g,y_g))$ is defined by:
\[
\rho_{N,\Omega}(\beta)=\left(\phi\left(\frac{\partial \beta(s_j)}{\partial s_i}\right) \right)_{1 \le i,j \le N}
\]
\end{definition}

Note that $\rho_{N,\Omega}(\beta)$ is an $N \times N$ matrix and that there are no derivatives with respect to any $w_i \in \Omega$. This agrees with the Jacobian formula for the CSW polynomial from Theorem \ref{thm_csw_poly_jacobian}. For the generators $\sigma_i \otimes 1_{2g}$ and $\chi_i \otimes 1_{2g}$, this recovers the usual virtual Burau representation: 
\begin{align} \label{eqn_rho_alpha}
\rho_{N,\Omega}(\sigma_i \otimes 1_{2g})= \left(\begin{array}{c|c|c|c}
I_{i-1} & \textbf{0} & \textbf{0} & \textbf{0} \\ \hline
\textbf{0} & 1-t & t & \textbf{0} \\ \hline
\textbf{0} & 1   & 0 & \textbf{0} \\ \hline
\textbf{0} & \textbf{0}   & \textbf{0} & I_{N-i-1} \\ 
\end{array} \right), & \quad \rho_{N,\Omega}(\chi_i \otimes 1_{2g})= \left(\begin{array}{c|c|c|c}
I_{i-1} & \textbf{0} & \textbf{0} & \textbf{0} \\ \hline
\textbf{0} & 0 & 1 & \textbf{0} \\ \hline
\textbf{0} & 1   & 0 & \textbf{0} \\ \hline
\textbf{0} & \textbf{0}   & \textbf{0} & I_{N-i-1} \\ 
\end{array} \right),
\end{align}
where $I_k$ is the $k \times k$ identity matrix and $\textbf{0}$ is a zero matrix of the appropriate dimensions. For the generators $\lambda_{i,w_j}$, the value is given by:
\begin{align} \label{eqn_rho_omega}
\rho_{N,\Omega}(\lambda_{i,w_j})=\left(\begin{array}{c|c|c}
I_{i-1} & \bm{0} & \bm{0} \\ \hline
 \bm{0} & w_j & \bm{0} \\ \hline
 \bm{0} & \bm{0} & I_{N-i}
\end{array} \right).
\end{align}
\begin{proposition} The map $\rho_{N,\Omega}$ is a representation of the $\Omega$-prismatic braid group $\mathit{VB}_{N,\Omega}$.
\end{proposition}

\begin{proof} That $\rho_{N,\Omega}$ is a homomorphism (i.e.  $\rho_{N,\Omega}(\beta_1 \circ \beta_2)=\rho_{N,\Omega}(\beta_1)\rho_{N,\Omega}(\beta_2)$) follows from the Chain Rule for Jacobians (see e.g. \cite{bz}, Section 9.B). Using Equations \ref{eqn_rho_alpha} and \ref{eqn_rho_omega}, it is straightforward to check that $\rho_{N,\Omega}$ satisfies the set of relations (\ref{rel_vbr1}), (\ref{rel_vbr2}), (\ref{rel_vbr3}) on $\Omega$-prismatic braids. The only relations that need to be checked are the $\Omega$-semi-welded move and the $\Omega$-commutator move. The $\Omega$-semi-welded move shown in Figure \ref{fig_sw_move_and_zh} can be written in terms of generators as $\lambda_{i,w_j}^{-1}(\chi_i \otimes 1_{2g})\lambda_{i,w_j}=\lambda_{(i+1),w_j}(\chi_i \otimes 1_{2g} )\lambda_{(i+1),w_j}^{-1}$. Applying $\rho_{N,\Omega}$ the left and right hand sides of this equation yields:
\begin{align*}
\left(\begin{array}{c|c|c}
I_{i-1} & \bm{0} & \bm{0} \\ \hline
 \bm{0} & w_j^{-1} & \bm{0} \\ \hline
 \bm{0} & \bm{0} & I_{N-i}
\end{array} \right)\left(\begin{array}{c|c|c|c}
I_{i-1} & \textbf{0} & \textbf{0} & \textbf{0} \\ \hline
\textbf{0} & 0 & 1 & \textbf{0} \\ \hline
\textbf{0} & 1   & 0 & \textbf{0} \\ \hline
\textbf{0} & \textbf{0}   & \textbf{0} & I_{N-i-1} \\ 
\end{array} \right) \left(\begin{array}{c|c|c}
I_{i-1} & \bm{0} & \bm{0} \\ \hline
 \bm{0} & w_j & \bm{0} \\ \hline
 \bm{0} & \bm{0} & I_{N-i}
\end{array} \right) &= \left(\begin{array}{c|c|c|c}
I_{i-1} & \textbf{0} & \textbf{0} & \textbf{0} \\ \hline
\textbf{0} & 0 & w_j^{-1} & \textbf{0} \\ \hline
\textbf{0} & w_j   & 0 & \textbf{0} \\ \hline
\textbf{0} & \textbf{0}   & \textbf{0} & I_{N-i-1} \\ 
\end{array} \right) \\ 
\left(\begin{array}{c|c|c}
I_{i} & \bm{0} & \bm{0} \\ \hline
 \bm{0} & w_j & \bm{0} \\ \hline
 \bm{0} & \bm{0} & I_{N-i-1}
\end{array} \right)\left(\begin{array}{c|c|c|c}
I_{i-1} & \textbf{0} & \textbf{0} & \textbf{0} \\ \hline
\textbf{0} & 0 & 1 & \textbf{0} \\ \hline
\textbf{0} & 1   & 0 & \textbf{0} \\ \hline
\textbf{0} & \textbf{0}   & \textbf{0} & I_{N-i-1} \\ 
\end{array} \right) \left(\begin{array}{c|c|c}
I_{i} & \bm{0} & \bm{0} \\ \hline
 \bm{0} & w_j^{-1} & \bm{0} \\ \hline
 \bm{0} & \bm{0} & I_{N-i-1}
\end{array} \right) &= \left(\begin{array}{c|c|c|c}
I_{i-1} & \textbf{0} & \textbf{0} & \textbf{0} \\ \hline
\textbf{0} & 0 & w_j^{-1} & \textbf{0} \\ \hline
\textbf{0} & w_j   & 0 & \textbf{0} \\ \hline
\textbf{0} & \textbf{0}   & \textbf{0} & I_{N-i-1} \\ 
\end{array} \right)
\end{align*}  
Now consider the $\Omega$-commutator move (see Figure \ref{fig_commutator_move}). Suppose that the black strand is the $i$-th one in the $\alpha$-part, $1 \le i \le N$. Then the effect of applying $\rho_{N,\Omega}$ is to multiply the $(i,i)$ entry of $I_N$ by $\prod_{i=1}^g [x_i,y_i]$. Since this is evaluated in $\mathbb{Z}[\mathbb{Z} \times H_1(\Sigma)]$, this coefficient reduces to $1$.    \end{proof}

Finally, it can be proved that the CSW polynomial can be calculated from a determinant formula via the representation $\rho_{N,\Omega}$ of the $\Omega$-prismatic braid group .

\begin{theorem} \label{thm_CSW_from_braid} Let $D$ be a link diagram on a surface $\Sigma$ and let $\Omega$ be a symplectic basis for $\Sigma$. Represent $\Zh_{\Omega}(D)$ as $\widehat{\beta}$ for some $\beta \in \mathit{VB}_{N,\Omega}$. Then the CSW polynomial for $L$ is given by:
\[
\Delta^0_D(t,x_1^{-1},y_1^{-1},\ldots,x_g^{-1},y_g^{-1}) \doteq \det\left(\rho_{N,\Omega}(\beta)-I_N\right),
\]
where ``$\doteq$'' means up to multiples of units in $\mathbb{Z}[\mathbb{Z} \times H_1(\Sigma)]$.
\end{theorem}
\begin{proof} The $\Omega$-prismatic braid $\beta$ determines an automorphism of the free group $F_{N+2g}$ on the generators $s_1,\ldots,s_{N},x_1,y_1,\ldots,x_g,y_g$.  The automorphism is given by:
\begin{align*}
\beta &: \left\{ \begin{array}{cccc} s_i & \longrightarrow & l_i s_{\pi(i)} l_i^{-1} & (1 \le i \le N) \\ x_i & \longrightarrow & x_i & (1 \le i \le g) \\ y_i & \longrightarrow & y_i & (1 \le i \le g) \end{array} \right.,
\end{align*}
where $l_i$ is a word in the generators of $F_{N+2g}$ and $\pi \in \mathbb{S}_{N}$ is the permutation of $1,\ldots,N$ defined by the $\alpha$-part of $\beta$. Note that $\beta$ acts on $x_i,y_i$ as the identity because the $\Omega$-colored components only have classical over-crossings with the $\alpha$-part, and all other crossings are virtual. Then the $\Omega$-prismatic link group $\pi_1^{\Omega}(\Zh_{\Omega}(D))$ has a presentation of the form:
\[
\left\langle s_1,\ldots,s_{N},x_1,y_1,\ldots,x_g,y_g| s_i^{-1}l_is_{\pi(i)} l_i^{-1} , 1 \le i \le N, \prod_{i=1}^g [x_i,y_i] \right\rangle,
\]
Combining Lemma \ref{lemma_pres_from_op_pres} and Theorems \ref{thm_csw_poly_jacobian} and \ref{thm_prismatic_link_group}, it follows that a presentation matrix for $A(D)$ can be obtained from this by applying the Fox calculus to the relations $s_i^{-1}l_is_{\pi(i)} l_i^{-1}$ and performing the substitution $w_i \to w_i^{-1}$ for all $w_i \in \Omega$. On the other hand, this calculation gives $\rho_{N,\Omega}(\beta)-I_{N}$. Thus, as claimed, the determinant of $\rho_{N,\Omega}(\beta)-I_{N}$ is the CSW polynomial, after the change of variables and up to multiples of units in $\mathbb{Z}[\mathbb{Z} \times H_1(\Sigma)]$.
\end{proof}

\subsection{Quantum model for the CSW polynomial} \label{sec_quantum_CSW} Starting with Theorem \ref{thm_CSW_from_braid}, the quantum model for the CSW polynomial can now be obtained from the well-known method of Kauffman--Saleur \cite{kauffman_saleur_91}, i.e. by passing to the exterior algebra. Let $D$ be a link diagram on a surface $\Sigma$ of genus $g$ and $\Omega=\{x_1,y_1,\ldots,x_g,y_g\}$ a symplectic basis. By Theorem \ref{thm_prismatic_braid_closure}, we may write $\Zh_{\Omega}(D)=\widehat{\beta}$ for some $\beta \in \mathit{VB}_{N,\Omega}$. Next, apply the determinant-trace formula to Theorem \ref{thm_CSW_from_braid}. This gives:
\begin{equation} \label{eqn_prismatic_trace_det}
\det(\rho_{N,\Omega}(\beta)-I_N)=(-1)^N \sum_{k=0}^N (-1)^k \text{tr}\left(\bigwedge\,\!\!^k \rho_{N,\Omega}(\beta)\right),
\end{equation}
where $\bigwedge\,\!\!^k$ denotes the $k$-th exterior power. On the other hand, there is also the trace formula for a Reshetikhin-Turaev functor (see e.g. \cite{jackson_moffatt}, Theorem 7.31):
\begin{equation} \label{eqn_prismatic_rt_trace_det}
\widetilde{Q}^{1|1}_{\Omega}(\widehat{\beta})=\text{tr}\left(\widetilde{Q}^{1|1}_{\Omega}(\beta) \circ \mu^{\otimes N} \right),
\end{equation}
where $\mu:V \to V$ is the cup-cap map. For the vector representation of $U_q(\mathfrak{gl}(1|1))$, this is given by:
\[
\mu=\begin{pmatrix}
q & 0 \\ 0 & -q
\end{pmatrix}
\]
Then $\mu^{\otimes N}(x_{i_1} \otimes \cdots \otimes x_{i_N})=q^N (-1)^{\sum_{i_j} |i_j|} x_{i_1} \otimes \cdots \otimes x_{i_N}$. Observe that there is a factor of $-1$ exactly with $x_{i_1} \otimes \cdots \otimes x_{i_N}$ is odd. Now, let $U$ be the $N$ dimensional vector space on which the $\rho_{N,\Omega}$ acts and let $\{u_1,\ldots,u_N\}$ denote its basis. Perform the substitution $t \to q^{-2}$. We may then identify $\bigwedge\!^* U$ with $V^{\otimes N}$ by setting $u_{i_1} \wedge \cdots \wedge u_{i_k}$, $i_1 <\cdots<i_k$, equal to the basis element $x_{l_1} \otimes \cdots \otimes x_{l_N}$ of $V^{\otimes N}$ having $x_2$ in the $(i_j+1) \pmod N$ position, $1  \le j \le N$, and $x_1$ in all other positions. In particular, the target element $x_{l_1} \otimes \cdots \otimes x_{l_N}$ is odd if and only if $k$ is odd. Then (\ref{eqn_prismatic_trace_det}) may be rewritten as:
\begin{equation} \label{eqn_rewrite_det_trace}
(-1)^N q^{-N} \text{tr}\left(\bigwedge\!^* \rho_{N,\Omega}(\beta) \circ \mu^{\otimes N} \right)
\end{equation}
Comparing (\ref{eqn_prismatic_rt_trace_det}) and (\ref{eqn_rewrite_det_trace}), we see that the result will follow if we can relate $\widetilde{Q}^{1|1}_{\Omega}$ and $\bigwedge\!^* \rho_{N,\Omega}(\beta)$. Since $\bigwedge\!^*$ is functorial, it suffices to consider the cases $N=2$, $\Omega=\varnothing$ and $N=1, \Omega=\{x_1,y_1\}$. For the first case, we first perform a change of basis using the matrix:
\[
\begin{pmatrix}
q^{-1/2} & 0 \\ 0 & q^{1/2}
\end{pmatrix}
\]
For $N=2$, $\Omega=\varnothing$, Equation (\ref{eqn_rho_alpha}) implies that the $\Omega$-prismatic braid representation is given by:
\[
\rho_{2,\varnothing}(\sigma_1)=\begin{pmatrix}
    1-\tfrac{1}{q^2} & \tfrac{1}{q} \\
    \tfrac{1}{q} & 0
    \end{pmatrix}, \quad \rho_{2,\varnothing}(\sigma_1^{-1})=\begin{pmatrix}
    0 & q \\
    q & 1-q^2
    \end{pmatrix}, \quad \rho_{2,\varnothing}(\chi_1)=\begin{pmatrix}
    0 & q \\
    \tfrac{1}{q} & 0
    \end{pmatrix}. 
\]
If $N=1, \Omega=\{x_1,y_1\}$, Equation (\ref{eqn_rho_omega}) implies that the $\Omega$-prismatic braid representation is given by:
\[
\rho_{1,\Omega}(\lambda_{1,w_i})=(w_i), \quad \rho_{1,\Omega}(\lambda_{1,w_i}^{-1})=(w_i^{-1}),
\]
where $w_i \in \Omega$. Next, apply $\bigwedge\!^*$ to $\rho_{2,\varnothing}$ and $\rho_{1,\Omega}$. The result is:
\small
\[
\bigwedge \!^*\rho_{2,\varnothing}(\sigma_1)=\begin{pmatrix}
    1 & 0 & 0 & 0 \\
    0 & 1-\tfrac{1}{q^2} & \tfrac{1}{q} & 0 \\
    0 & \tfrac{1}{q} & 0 & 0 \\
    0 & 0 & 0 & -\tfrac{1}{q^2}
\end{pmatrix}, \bigwedge \!^*\rho_{2,\varnothing}(\sigma_1^{-1})= \begin{pmatrix}
    1 & 0 & 0 & 0 \\
    0 & 0 & q & 0 \\
    0 & q & 1-q^2 & 0 \\
    0 & 0 & 0 & -q^2
\end{pmatrix}, \bigwedge \!^*\rho_{2,\varnothing}(\chi_1)=
\begin{pmatrix}
    1 & 0 & 0 & 0 \\
    0 & 0 & q & 0 \\
    0 & \tfrac{1}{q} & 0 & 0 \\
    0 & 0 & 0 & -1
\end{pmatrix},
\]
\[
\bigwedge \!^* \rho_{1,\Omega}(\lambda_{1,w_i})=\begin{pmatrix}
1 & 0 \\ 0 & w_i
\end{pmatrix}, \quad 
\bigwedge \!^* \rho_{1,\Omega}(\lambda_{1,w_i}^{-1})=\begin{pmatrix}
1 & 0 \\ 0 & w_i^{-1}
\end{pmatrix},
\]
\normalsize
Comparing these values with prismatic $U_q(\mathfrak{gl}(1|1))$ functor in Section \ref{sec_prismatic_rt_defn}, we get: 
\begin{equation} \label{eqn_final_Burau_to_quantum_I}
\bigwedge \!^*\rho_{2,\varnothing}(\sigma_1)=q^{-1} \cdot \widetilde{Q}^{1|1}_{\varnothing} (\sigma_1), \quad \bigwedge \!^*\rho_{2,\varnothing}(\sigma_1^{-1})=q \cdot \widetilde{Q}^{1|1}_{\varnothing} (\sigma_1^{-1}), \quad \bigwedge \!^*\rho_{2,\varnothing}(\chi_1)=\widetilde{Q}^{1|1}_{\varnothing} (\chi_1),
\end{equation}
\begin{equation}  \label{eqn_final_Burau_to_quantum_II}
\bigwedge \!^* \rho_{1,\Omega}(\lambda_{1,w_i})=\widetilde{Q}^{1|1}_{\Omega}(\lambda_{1,w_i}), \quad \bigwedge \!^* \rho_{1,\Omega}(\lambda_{1,w_i}^{-1})=\widetilde{Q}^{1|1}_{\Omega}(\lambda_{1,w_i}^{-1})
\end{equation}
Thus we have the following theorem relating the CSW polynomial to the prismatic functor $\widetilde{Q}^{1|1}_{\Omega}$.

\begin{theorem} \label{thm_quantum_csw_model} For $D$ be a link diagram on $\Sigma$ and $\Omega=\{x_1,y_1,\ldots,x_g,y_g\}$ a symplectic basis. Then after the change of variables $\Delta_D^0(q^{-2},x_1^{-1},y_1^{-1},\ldots,x_g^{-1},y_g^{-1})$ and a change of basis, the CSW polynomial coincides with $\widetilde{Q}^{1|1}_{\Omega}(\Zh_{\Omega}(D))$, up to multiples of units in $\mathbb{Z}[\langle q\rangle \times H_1(\Sigma)]$.
\end{theorem}
\begin{proof} We will use the notation of the current subsection, so that $D=\widehat{\beta}$ for some $\beta \in \mathit{VB}_{N,\Omega}$. Now, since $\bigwedge\!^*$ is functorial, Equations (\ref{eqn_final_Burau_to_quantum_I}) and (\ref{eqn_final_Burau_to_quantum_II}) imply that  we may replace $\bigwedge\!^* \rho_{N,\Omega}(\beta)$ with $\widetilde{Q}^{1|1}_{\Omega}(\beta)$ in Equation (\ref{eqn_rewrite_det_trace}). In doing so, there is a cost of $\pm q^{k}$, where $k$ is the writhe of of the $\alpha$-part of $\beta$. Such a factor is allowed, due to the symbol ``$\doteq$''. The claim then follows by Theorem \ref{thm_CSW_from_braid}.
\end{proof}

\section{$U_q(\mathfrak{gl}(m|n))$ lower bounds on the virtual $2$-genus} \label{sec_gen_bounds}
In Section \ref{sec_quantum_CSW}, the CSW polynomial of a link $L \subset \Sigma \times [0,1]$, for a fixed symplectic basis $\Omega$, was identified with the composition of the prismatic $U_q(\mathfrak{gl}(1|1))$ Reshetikhin-Turaev functor and the homotopy $\Zh$-construction. Hence, the symplectic rank of the prismatic $U_q(\mathfrak{gl}(1|1))$ Reshetikhin-Turaev polynomial of a link gives a lower bound on $\widecheck{g}_2(L)$. To generalize this bound to every $U_q(\mathfrak{gl}(m|n))$, we study the extent to which this composition for $U_q(\mathfrak{gl}(m|n))$ gives an invariant of knots in thickened surfaces. Additional care is required in this case since quantum invariants of virtual links are only invariant in the rotational category for $(m,n) \ne (1,1)$. To remedy this, we first refine the homotopy $\Zh$-construction so that defined relative to a fixed ``screen''. The link diagram on $\Sigma$ is projected to a virtual link diagram using this screen (see Section \ref{sec_homotopy_zh_screen}). Invariants of links in thickened surfaces and once-punctured surfaces are discussed in Section   \ref{sec_invar_in_thick_surf}. The $U_q(\mathfrak{gl}(m|n))$ bounds on $\widecheck{g}_2(L)$ are given in Section \ref{sec_gen_glmn_bounds}.
 
\subsection{Homotopy $\Zh$-construction relative to a screen} \label{sec_homotopy_zh_screen}
 Let $D$ be a link diagram on $\Sigma$ and $\Omega=\{x_1,y_1,\ldots,x_g,y_g\}$ is a symplectic basis with base point $\infty$. Here we assuming, as in the usual setup for the homotopy $\Zh$-construction, that $D$ intersects $\Omega$ transversely, never at a crossing, and never at $\infty$. Let $U \approx D^2$ be a small neighborhood of $\infty$ and set $\Sigma_{\infty}=\overline{\Sigma \smallsetminus U}$. Cutting $\Sigma_{\infty}$ along the curves in $\Omega$ gives a $4g$-gon $P$, which is identified with a fixed region in the plane $\mathbb{R}^2$. Up to now, this has agreed with the usual $\Zh_{\Omega}$. Now, attach $2g$ 1-handles $H^1_{x_1},H^1_{y_1},\ldots,H^1_{x_g},H^1_{y_g}$ to the sides of $P$ so that $H^1_{x_i}$ (resp. $H^1_{y_i}$) is attached along the sides of $P$ labeled $x_i$ (resp. $y_i$). See Figure \ref{fig_screen}. We will assume a specific choice of such ``screen'' on which to project link diagrams for a given choice of $\Sigma$ and $\Omega$, and we denote this choice by $S$. For the placement of the $1$-handles, it is required that the the cores of the one handles $H^1_{x_i},H^1_{y_i}$ have exactly one transversal intersection, the intersection $H^1_{x_i} \cap H^1_{y_i}$ is a disc $O_i \approx D^2$ for $1 \le i \le g$, and $H^1_{x_i} \cap H^1_{x_j}=H^1_{y_i} \cap H^1_{y_j}=\varnothing$ for $1 \le i \ne j \le g$. The disc $O_i$ is called that \emph{virtual region} of $H^1_{x_i}$ and $H^1_{y_i}$. Note also that $\partial S$ is itself a virtual knot diagram.

\begin{figure}[htb]
\begin{tabular}{|c|} \hline \\
\xymatrix{
\begin{tabular}{c}
\includegraphics[scale=.52]{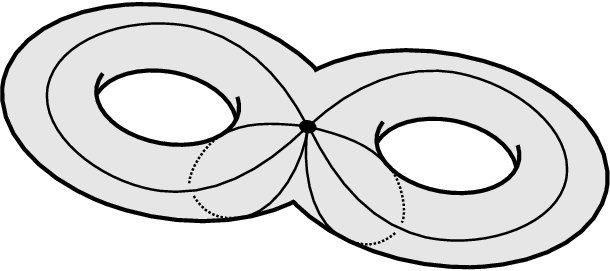}
\end{tabular} \ar@/^5pc/[r]^-{\text{puncture $\&$ flatten}}
&
\begin{tabular}{c}
\includegraphics[scale=.55]{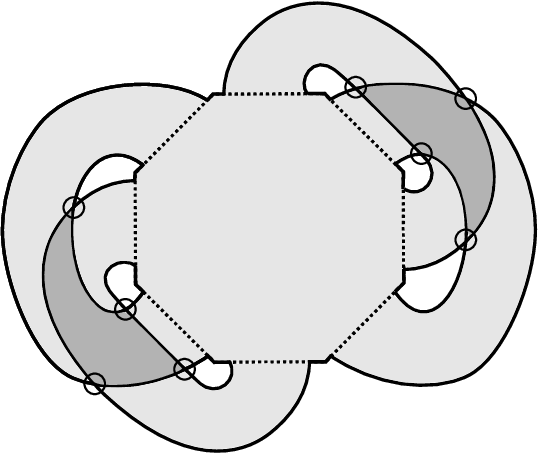}
\end{tabular}
} \\ \\ \hline
\end{tabular}
\caption{A fixed ``screen'' for a two-holed torus.} \label{fig_screen}
\end{figure}

Now the link diagram $D$ on $\Sigma$ may be drawn as a virtual link diagram $\widecheck{D}$ in $\mathbb{R}^2$. All classical crossings of $D$ appear on the interior of $P$. The crossings of $P$ are joined by arcs either lying in the interior of $P$, or arcs which intesect the boundary of $P$. Suppose in the latter case that an arc intersects a side labeled $z_i$ at a point. Then there must be another arc of $D$ intersecting the same point of the other side of $P$ labeled $z_i$.  Hence, they may be connected by a simple arc in $H^1_{z_i}$ that is parallel to the sides of $H^1_{z_i}$. The set of all such simple arcs added to the figure may have intersections, but all will lie in virtual regions $O_1,\ldots,O_g$. Any intersections in these overlapping regions are marked as virtual crossings of the virtual link diagram $\widecheck{D}$. Note here the emphasis that virtual crossings are artifacts of how one chooses to flatten the surface $\Sigma$ on which the real diagram $D$ lives. With a fixed screen, virtual crossings can only ever appear in its virtual regions $O_1,\ldots,O_g$. 

\begin{lemma} \label{lemma_screen_rot_equiv} If $D_1,D_2$ are Reidemeister equivalent link diagrams on $\Sigma_{\infty}=\overline{\Sigma\smallsetminus U}$, then $\widecheck{D}_1, \widecheck{D}_2$ are rotationally equivalent virtual link diagrams
\end{lemma}
\begin{proof} A Reidemeister move on $\Sigma_{\infty}$ occuring entirely within the the interior of $P$ appears in $S$ as a classical Reidemeister move. Hence, in this case, $\widecheck{D}_1$, $\widecheck{D}_2$ are rotationally equivalent. Suppose the arcs of a Reidemeister move intersect $\Omega$. For each classical Reidemeister move of type 1, 2, or 3, the arcs involved on at least one of its left and right hand sides enclose a region of $\Sigma$ homeomorphic to a disc $B$. By hypothesis, this disc cannot contain $\infty$. Hence, $B$ can be contracted so that it lies inside $P$. This implies that the second case reduces to the first. The contraction of $B$ appears $S$ as pushing some of its arcs over some of the handles $H^1_{x_1},H^1_{y_1},\ldots,H^1_{x_g},H^1_{y_g}$. In passing some arc over a $1$-handle $H_{z_i}^1$, virtual Reidemeister moves may be required. If $H^1_{z_i'}$ is the unique $1$-handle having non-empty intersection with $H^1_{z_i}$, these moves involve passing some arcs parallel to the sides of $H^1_{z_i'}$ over the arcs in the move. But along these arcs, there are only virtual crossings in $O_i$. After the contraction of $B$, the number of arcs of $\widecheck{D}$ parallel to the sides of $H^1_{z_i}$ and $H^1_{z_i'}$ is fixed. In particular, the number of virtual crossings of $\widecheck{D}$ has not changed. Then the contraction avoids virtual Reidemeister 1 moves ($VT_1$ in Figure \ref{fig_virtual_tangle_moves}), so that $\widecheck{D}_1$, $\widecheck{D}_2$ are rotationally equivalent.
\end{proof}

\begin{figure}[htb]
\begin{tabular}{|c|} \hline \\
\xymatrix{
\begin{tabular}{c}
\includegraphics[scale=.6]{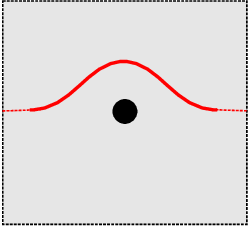} \end{tabular} \ar[r] &
\begin{tabular}{c}
\includegraphics[scale=.6]{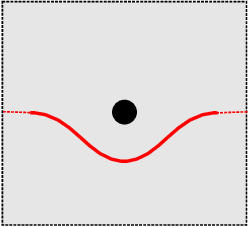}
\end{tabular}
} \\ \\ \hline
\end{tabular}
\caption{An isotopy over the base point $\infty$.} \label{fig_isotopy_over_basepoint}
\end{figure}

If $D_1,D_2$ are Reidemeister equivalent by a Reidemeister move on $\Sigma$ but not on $\Sigma_{\infty}$, an arc of $D_1$ must have passed over the base point $\infty$. See Figure \ref{fig_isotopy_over_basepoint}. The effect of passing over the base point can be realized in the screen $S$ by taking a connected sum of $\widecheck{D}$ with a parallel copy of the virtual knot diagram $\partial S$, appropriately oriented. See Figure \ref{fig_isotopy_basepoint}. This is called a \emph{base-point operation}.

\begin{figure}[htb]
\begin{tabular}{|c|} \hline 
\\ 
\xymatrix{
\begin{tabular}{c}
\includegraphics[scale=.65]{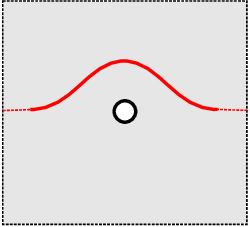}\end{tabular}
\ar[r] &
\begin{tabular}{c} \includegraphics[scale=.7]{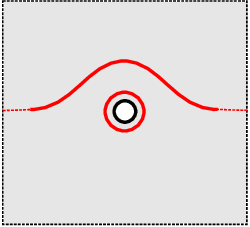} \end{tabular} \ar[r]
&
\begin{tabular}{c}
\includegraphics[scale=.65]{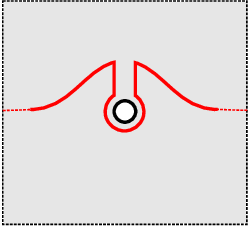}
\end{tabular}
\ar[r]
&
\begin{tabular}{c}
\includegraphics[scale=.65]{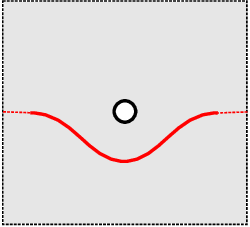} \end{tabular}} \\ \\ \hline
\end{tabular}

\caption{The effect of an isotopy over the base point is the same as a connect sum with a boundary parallel component.}\label{fig_isotopy_basepoint}
\end{figure}

\begin{example} \label{ex_basepoint_op_torus} A base-point operation applied to our running example is shown in Figure \ref{fig_basepoint_example}. The connected sum with the screen boundary is shown in the middle of Figure \ref{fig_basepoint_example}. Deleting the screen and performing a rotational isotopy in the plane, it can be easily checked that this is the right-handed virtual trefoil (Figure \ref{fig_basepoint_example}, right). For $g=1$, the screen boundary is rotationally trivial.
\end{example}

\begin{figure}[htb]
\begin{tabular}{|ccc|} \hline
 & & \\
\begin{tabular}{c} \includegraphics[scale=.4]{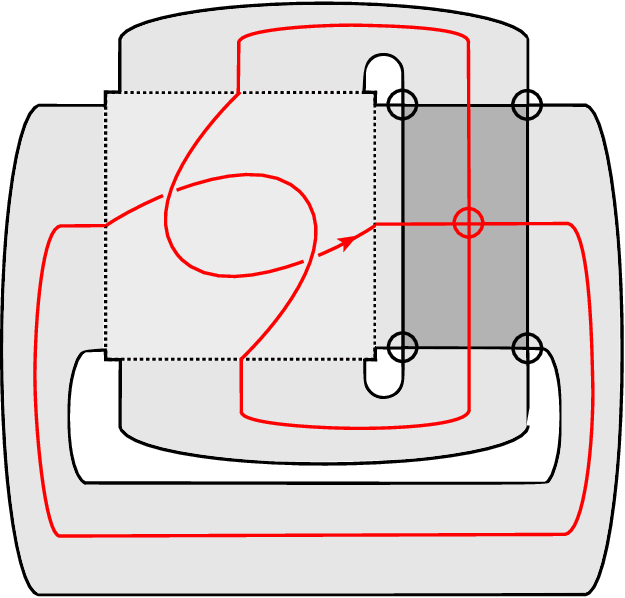} \end{tabular} & \begin{tabular}{c} \includegraphics[scale=.4]{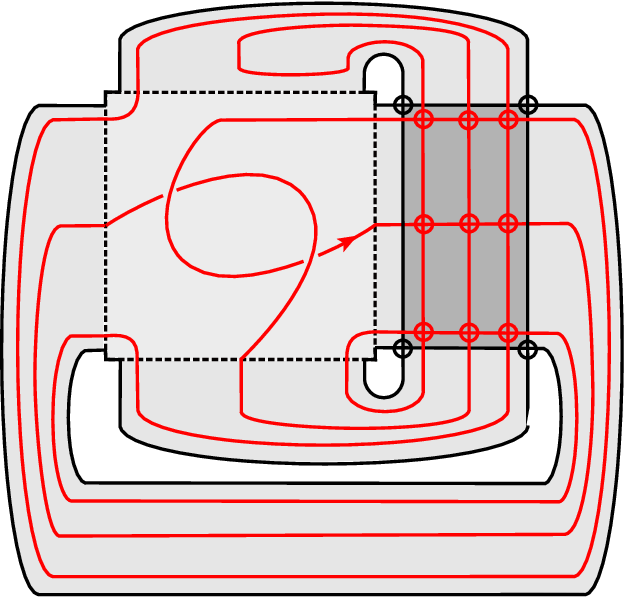} \end{tabular} & \begin{tabular}{c} \includegraphics[scale=.4]{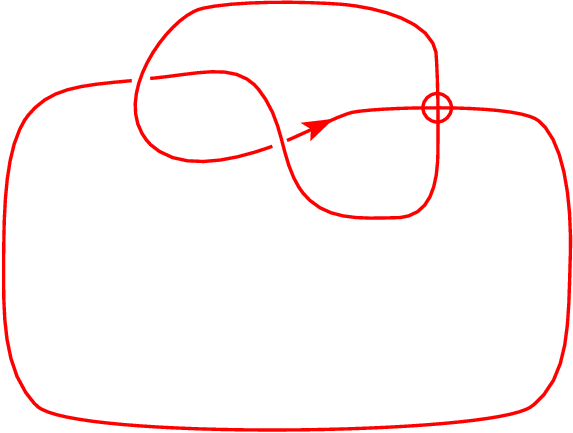}
\end{tabular} \\ & & \\ \hline \end{tabular}
\caption{The effect of base-point move on a surface of genus $g=1$.} \label{fig_basepoint_example}
\end{figure}

\begin{example} \label{ex_basepoint_op_rot_number} For $g \ne 1$, the screen boundary need not be rotationally trivial. Consider the screen for a surface $\Sigma$ of genus $2$ shown on the left in Figure 
\ref{fig_basepoint_rotation_number}. A parallel copy of its boundary is drawn in red. The virtual crossings for this knot have been suppressed for aesthetic reasons. When the screen is deleted, we have the rotational virtual knot shown in the center of Figure \ref{fig_basepoint_rotation_number}. This is equivalent to the knot shown on the right in Figure \ref{fig_basepoint_rotation_number} after a rotational isotopy. Although this is virtually equivalent to the unknot, it is not rotationally equivalent to the unknot (Kauffman \cite{kauffman_cobordism}).
\end{example}

\begin{figure}
\begin{tabular}{|c|} \hline
\\
\includegraphics[scale=.75]{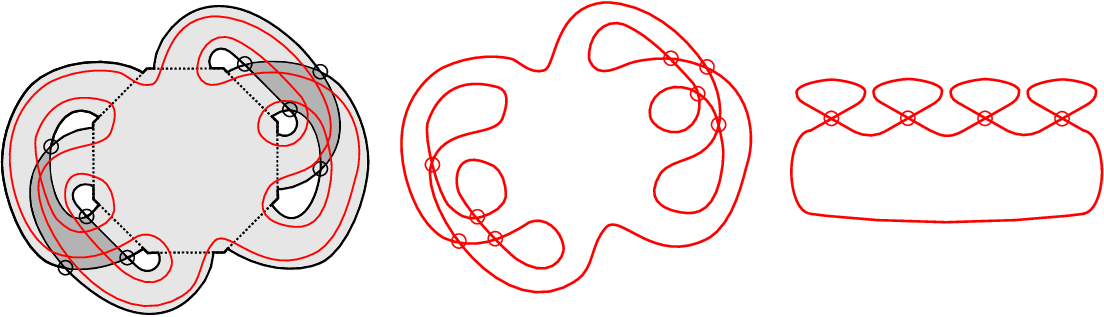}
\\ \\ \hline
\end{tabular}
\caption{Base-point operations and the rotation number.} \label{fig_basepoint_rotation_number}
\end{figure}

\begin{example} \label{ex_basepoint_op_order_two} Suppose a base-point operation is performed on a link diagram $D$ on $\Sigma$ in Figure \ref{fig_isotopy_basepoint}, and then the arc is immediately passed back over the basepoint again. The result is that $D$ is returned to its original position. On the screen, this is equivalent to taking a connect sum with the boundary and then a second connect sum where the orientation of the boundary is reversed. Then it is easy to see that the original diagram of $\widecheck{D}$ is recovered after a rotational isotopy. 
\end{example}

\begin{lemma} \label{lemma_rot_and_base} If $D_1,D_2$ are Reidemeister equivalent link diagrams on $\Sigma$, then $\widecheck{D}_1, \widecheck{D}_2$ are equivalent up to  rotational isotopy and base-point operations.
\end{lemma}
\begin{proof} This follows immediately from Lemma \ref{lemma_screen_rot_equiv} and the preceding discussion.
\end{proof}

Now suppose that $D$ is a link diagram on $\Sigma$, $\Omega$ is a symplectic basis at $\infty$, and $S \subset \mathbb{R}^2$ is a screen. Let $\widecheck{D}$ denote the corresponding virtual link diagram in $S$. Perform the homotopy $\Zh$-construction by overlaying the arcs of $\widecheck{D}$ with the $\Omega$-colored over-crossing arcs, as described in Section \ref{sec_homotopy_zh_defn}. This called the \emph{homotopy $\Zh$-construction relative to $S$}, denoted $\Zh_{\Omega}^S(D)$.

\begin{lemma} \label{lemma_prism_zh_rel_screen} If $D_1,D_2$ are Reidemeister equivalent link diagrams on $\Sigma$, then $\Zh_{\Omega}^S(D_1)$, $\Zh_{\Omega}^S(D_2)$ are rotationally equivalent $\Omega$-prismatic links, up to base-point operations.
\end{lemma}
\begin{proof} If $D_1,D_2$ are Reidemeister equivalent on $\Sigma_{\infty}$, then the result follows as in Lemma \ref{lemma_screen_rot_equiv} and Theorem \ref{thm_homotopy_zh_invariance}. If $D_1,D_2$ are related by a Reidemeister move that encloses $\infty$, then as in previous proofs, we can reduce this case to the first by first moving an arc of $D_1$ over the base point. Hence, it suffices to prove the claim when $D_1,D_2$ are related by an isotopy over $\infty$, as shown in Figure \ref{fig_isotopy_over_basepoint}. In this event, Lemma \ref{lemma_rot_and_base} implies that $\widecheck{D}_1$ and $\widecheck{D}_2$ are related by a base-point operation. On the other hand, Theorem \ref{thm_homotopy_zh_invariance} implies that $\Zh_{\Omega}(D_1)$ and $\Zh_{\Omega}(D_2)$ are related by an $\Omega$-commutator move. It must then be shown that these two moves can be performed interchangeably.  
 
First, let $C$ in $\Sigma_{\infty}$ be a closed curve parallel to $\partial U \subset \Sigma_{\infty}$, where $B^2 \approx U \subset \Sigma$ is a small neighborhood of $\infty$. Let $\widecheck{C}$ denote the corresponding virtual knot diagram representing a parallel copy of $\partial S$. Let $\widecheck{b}$ be a small $\Omega$-colored over-crossing arc added to $\widecheck{C}$ in the homotopy $\Zh$-construction applied to $C$. For the moment, assume that $\widecheck{C}$ contains only one such arc. Using rotational moves, $\widecheck{b}$ can be pushed along $\widecheck{C}$ to any other point by moving the virtual crossings it encounters on $\widecheck{C}$ out of the way. Hence, for all of the $\Omega$-colored arcs added to $\widecheck{C}$ by $\Zh_{\Omega}^S$, it can be assumed that they can be placed as in the left-hand side of Figure \ref{fig_commutator_move} or some cyclic permutation thereof. By Lemma \ref{lemma_commutator_permute}, any cyclic permutation of the vertical over-crossing arcs in an $\Omega$-commutator move is equivalent to the $\Omega$-commutator move itself. Hence, all the $\Omega$-colored arcs can be pulled off $\widecheck{C}$ as in the right hand side of Figure \ref{fig_commutator_move}. Then $\Zh_{\Omega}^S(C)$ is equivalent as an $\Omega$-prismatic link to the disjoint union of $\widecheck{C}$ and an unlink consisting of one unknotted component for each element of $\Omega$.   

Now, for $D_1$ and $D_2$, the argument above shows that $\Zh_{\Omega}^S(D_1  \# C)$ and $\Zh_{\Omega}^S(D_1)\# \widecheck{C}$ are rotationally equivalent $\Omega$-prismatic links, where $\#$ denotes the appropriate connected sum. Since $D_2=D_1 \# C $, $\Zh_{\Omega}^S(D_1)$ and $\Zh_{\Omega}^S(D_2)$ are rotationally equivalent, up to base-point operations.  
\end{proof}

\subsection{$U_q(\mathfrak{gl}(m|n))$ invariants of knots in $\Sigma \times [0,1]$ $\&$ $\Sigma_{\infty} \times [0,1]$} \label{sec_invar_in_thick_surf} In Section \ref{sec_quantum_CSW}, the CSW polynomial of a link in $\Sigma \times [0,1]$ was identified with a composition of the $U_q(\mathfrak{gl}(1|1))$ invariant and the homotopy $\Zh$-construction on $\Sigma$, assuming a fixed choice of symplectic basis. This section extends this idea to $U_q(\mathfrak{gl}(m|n))$ to obtain invariants of knots in various thickened surfaces.

\begin{definition}[Prismatic $U_q(\mathfrak{gl}(m|n))$ Reshetikhin-Turaev polynomial] Let $\Sigma$ be a closed connected orientable surface and $\Omega=\{x_1,y_1,\ldots,x_g,y_g\}$ be a choice of symplectic basis. Let $S$ be a fixed screen for $\Sigma,\Omega$ and let $D$ be a link diagram on $\Sigma$. Define:
\[
\widetilde{f}_{(\Sigma,\Omega,D)}^{\,\, m|n}(q,x_1,y_1,\ldots,x_g,y_g)=q^{(m-n)\text{wr}(D)} \widetilde{Q}^{\,\, m|n}_{\Omega}(\Zh_{\Omega}^S(D)),
\]
where $\text{wr}(D)$ is the classical writhe of $D$ on $\Sigma$.
\end{definition}

Next we consider the effect of Reidemeister equivalence and orientation preserving diffeomorphisms on the prismatic $U_q(\mathfrak{gl}(m|n))$ Reshetikhin-Turaev polynomial. Recall from Section \ref{sec_covering} that if $h:\Sigma \to \Sigma$ is an orientation preserving diffeomorphism, there is an induced map $h_{\sharp}:\mathbb{Z}[H_1(\Sigma;\mathbb{Z})] \to \mathbb{Z}[H_1(\Sigma;\mathbb{Z})]$ whose effect on a polynomial $f(t,x_1,y_1,\ldots,x_g,y_g)$ is to apply the induced map on homology to the coefficients of $t$.  This same notation will be used in the following lemma.

\begin{lemma} \label{lemma_fmn_invariants} Let $D_1,D_2 \subset \Sigma$ be link diagrams, $\Omega=\{x_1,y_1,\ldots,x_g,y_g\}$ a symplectic basis at $\infty$. 
\begin{enumerate} 
\item If $D_1,D_2$ are Reidemeister equivalent on $\Sigma_{\infty}$, then: 
\[
\widetilde{f}^{\,\,m|n}_{(\Sigma,\Omega,D_1)}(q,x_1,y_1,\ldots,x_g,y_g)=\widetilde{f}^{\,\,m|n}_{(\Sigma,\Omega,D_2)}(q,x_1,y_1,\ldots,x_g,y_g).
\]
\item If $h(D_1)=D_2$ for some $h:\Sigma \to \Sigma$ with $h(\infty)=\infty$, then relative to the same screen $S$:
\[
h_{\sharp}\left(\widetilde{f}^{\,\,m|n}_{(\Sigma,\Omega,D_1)}(q,x_1,y_1,\ldots,x_g,y_g)\right)=\widetilde{f}^{\,\,m|n}_{(\Sigma,h(\Omega),D_2)}(q,x_1',y_1',\ldots,x_g',y_g'),
\]
where $h(\Omega)=\{x_1',y_1',\ldots,x_g',y_g'\}$.
\end{enumerate} 
\end{lemma}
\begin{proof} By Lemma \ref{lemma_prism_zh_rel_screen}, $\Zh_{\Omega}^S(D_1)$ and $\Zh_{\Omega}^S(D_2)$ are rotationally equivalent, up to base-point moves. On $\Sigma_{\infty}$, base-point operations are not allowed, so it follows that $\Zh_{\Omega}^S(D_1)$ and $\Zh_{\Omega}^S(D_2)$ are rotationally equivalent $\Omega$-prismatic links. By Theorem \ref{thm_mv_functor_is_wd}, $\widetilde{Q}^{\,\,m|n}_{\Omega}$ is an invariant of framed rotational isotopy. As can be easily checked with our conventions, the effect of $\widetilde{Q}^{\,\,m|n}_{\Omega}$ on a positive classical curl is $q^{-(m-n)} I_2$ and the effect on a negative classical curl is $q^{(m-n)} I_2$. It follows that $\widetilde{f}^{\, m|n}_{(\Sigma,\Omega,D)}$ is invariant under all Reidemeister moves on $\Sigma$. This proves the first claim. The second claim follows as in Lemma \ref{lemma_homotopy_zh_and_diffeo}. Indeed, $h$ maps $\Omega$ to a symplectic basis $h(\Omega)$. Using the same screen $S$, the sides of the polygon $P$ can be relabeled by the elements of $h(\Omega)$.  But since $D_2=h(D_1)$ and $h$ preserves the intersection pairing, it follows that $\Zh_{\Omega}^S(D_1)=\Zh_{h(\Omega)}^S(D_2)$.
\end{proof}

Recall from Section \ref{sec_csw_poly} that polynomials $f,f' \in \mathbb{Z}[\langle q \rangle  \times H_1(\Sigma;\mathbb{Z})]$ are called \emph{equivalent} if $f=u \cdot \varphi_{\sharp}(f')$ for some orientation preserving diffeomorphism $\varphi$ of $\Sigma$ and unit $u$. If we consider $\widetilde{f}^{\,\, m|n}_{(\Sigma,\Omega,D)}(q,x_1,y_1,\ldots,x_g,y_g)$ to be an element of $\mathbb{Z}[\langle q \rangle  \times H_1(\Sigma;\mathbb{Z})]$, this can be applied to $U_q(\mathfrak{gl}(m|n))$  Reshetikhin-Turaev polynomials to obtain invariants of links in thickened surfaces.

\begin{theorem} Let $\Sigma$ be a closed connected orientable surface and $\Omega=\{x_1,y_1,\ldots,x_g,y_g\}$ be a symplectic basis with base point $\infty$. For a link $L \subset \Sigma \times [0,1]$, let $D$ be a diagram of $L$ on $\Sigma$. Then:
\begin{enumerate}
\item The equivalence class of $\widetilde{f}^{\,\,1|1}_{(\Sigma,\Omega,D)}(q,x_1,y_1,\ldots,x_g,y_g)$ is an invariant of links in $\Sigma \times [0,1]$.
\item The equivalence class of $\widetilde{f}^{\,\,m|n}_{(\Sigma,\Omega,D)}(q,x_1,y_1,\ldots,x_g,y_g)$ is an invariant of links in $\Sigma_{\infty} \times [0,1]$.
\end{enumerate}
\end{theorem}
\begin{proof} If $L_1,L_2$ are equivalent in $\Sigma \times [0,1]$, then they have diagrams $D_1,D_2$, respectively, on $\Sigma$, that are related by a finite sequence of Reidemeister moves on $\Sigma$ and diffeomorphisms of $\Sigma$. The second claim is thus an immediate consequence of Theorem \ref{lemma_fmn_invariants}. For the first claim, suppose that $D_1$ and $D_2$ are related by an isotopy over $\infty$. In this event, $\Zh_{\Omega}^S(D_1)$ and $\Zh_{\Omega}^S(D_2)$ are related by a base-point operation. A base-point operation does not change the virtual equivalence class, but can change the rotational equivalence class. However, when $m=n=1$, the effect of $\widetilde{Q}^{\,m|n}_{\Omega}(T)$ on a virtual curl $T$ is simply to multiply the identity matrix by $q^{\pm 1}$ (see \cite{CP}, Section 3.4). Thus, $\widetilde{f}^{\,\,1|1}_{(\Sigma,\Omega,D_1)}(q,x_1,y_1,\ldots,x_g,y_g)$ and $\widetilde{f}^{\,\,1|1}_{(\Sigma,\Omega,D_2)}(q,x_1,y_1,\ldots,x_g,y_g)$ are equivalent polynomials.
\end{proof}

\subsection{The CSW bound for $U_q(\mathfrak{gl}(m|n))$} \label{sec_gen_glmn_bounds} The first task in generalizing the CSW bound is to define the symplectic rank of a prismatic $U_q(\mathfrak{gl}(m|n))$ polynomial. Consider the coefficients of $q$ in the  polynomial $\widetilde{f}^{\, m|n}_{(\Sigma,\sigma,D)}(q,x_1,y_1,\ldots,x_g,y_g)$. Each is a power product of the elements of $\Omega$. But $\Omega$ is identified with a symplectic basis for $H_1(\Sigma;\mathbb{Z})$. Thus, if each coefficient is written additively rather than as a power product, we obtain an element of $H_1(\Sigma;\mathbb{Z})$. Let $W_{(\Sigma,\Omega,D)}$ be the submodule of $H_1(\Sigma;\mathbb{R})$ generated by these rewritten coefficients of $\widetilde{f}^{\, m|n}_{(\Sigma,\sigma,D)}(q,x_1,y_1,\ldots,x_g,y_g)$.

\begin{definition}[Symplectic rank of $\widetilde{f}^{\,\,m|n}_{(\Sigma,\sigma,D)}$] The \emph{symplectic rank} of $\widetilde{f}^{\, m|n}_{(\Sigma,\sigma,D)}(q,x_1,y_1,\ldots,x_g,y_g)$ is defined to be the symplectic rank of the submodule $W_{(\Sigma,\sigma,D)}$ of $H_1(\Sigma;\mathbb{R})$ generated by the coefficients of $q$ in $\widetilde{f}^{\, m|n}_{(\Sigma,\sigma,D)}(q,x_1,y_1,\ldots,x_g,y_g)$, written additively. 
\end{definition}

Now we can prove the main theorem of this paper. First, two lemmas are needed.

\begin{lemma} \label{lemma_symp_rank_reidemeister} Suppose $D_1,D_2$ are Reidemeister equivalent on $\Sigma$ and $\Omega=\{x_1,y_1,\ldots,x_g,y_g\}$ is a symplectic basis. Then relative to the same screen $S$:
\[
\text{rk}_s\left(\widetilde{f}^{\,\,m|n}_{(\Sigma,\Omega,D_1)}\right)=\text{rk}_s\left(\widetilde{f}^{\,\,m|n}_{(\Sigma,\Omega,D_2)}\right).
\]
\end{lemma}
\begin{proof}  If $D_1,D_2$ are Reidemeister equivalent on $\Sigma_{\infty}$, then relative to the same screen $S$, they have equal $U_q(\mathfrak{gl}(m|n))$ Reshetikhin-Turaev polynomial by Lemma \ref{lemma_fmn_invariants}. Hence, in this case, they must have the same symplectic rank. If $D_1,D_2$ are related by an isotopy over a base-point, Lemma \ref{lemma_prism_zh_rel_screen} implies that $\Zh_{\Omega}^S(D_1)$ and $\Zh_{\Omega}^S(D_2)$ are equivalent prismatic rotational links, up to a base-point operation. Since the functor $\widetilde{Q}_{\Omega}^{m|n}$ is an invariant of rotational isotopy of prismatic links, we need only consider the effect of a base-point operation on the symplectic rank. A base-point operation is a connected sum of $\Zh_{\Omega}^S(D_1)$ with a parallel (or antiparallel) copy $\widecheck{C}$ of the screen boundary $ \partial S$. Then we may write $\Zh_{\Omega}^S(D_2)=\Zh_{\Omega}^S(D_1)\# \widecheck{C}$. As the connect sum does not effect the $\Omega$-part of the prismatic link, no new factors of any $z_i \in \Omega$ can be added to the coefficients of $q$ in $\widetilde{f}^{\,\, m|n}_{(\Sigma,\Omega,D_2)}(q,x_1,y_1,\ldots,x_g,y_g)$ when the operation is performed. Thus, the symplectic rank of the polynomials is preserved by the operation, even though the polynomials themselves may differ. 
\end{proof}

\begin{lemma} \label{lemma_symp_rank_cob} Suppose $D_1,D_2$ are Reidemeister equivalent on $\Sigma$ and $\Omega=\{x_1,y_1,\ldots,x_g,y_g\}$ is a symplectic basis. If $h(D_1)=D_2$ for some orientation preserving diffeomorphism $h:\Sigma \to \Sigma$, then relative to the same screen $S$:
\[
\text{rk}_s\left(\widetilde{f}^{\,\,m|n}_{(\Sigma,\Omega,D_1)}\right)=\text{rk}_s\left(\widetilde{f}^{\,\,m|n}_{(\Sigma,h(\Omega),D_2)}\right).
\] 
\end{lemma}
\begin{proof} By Lemma \ref{lemma_fmn_invariants}, the effect of $h$ is to apply the induced map on homology to the coefficients of $q$ in $\widetilde{f}^{\,\,m|n}_{(\Sigma,\Omega,D)}(q,x_1,y_1,\ldots,x_g,y_g)$. But the induced map in homology is an element of $\text{Sp}(2g,\mathbb{Z})$, so that it preserves the intersection form on $\Sigma$. This implies that:
\[
h_{\#}\left(\frac{W_{(\Sigma,\Omega,D_1)}}{W_{(\Sigma,\Omega,D_1)} \cap W_{(\Sigma,\Omega,D_1)}^{\perp}}\right) =\frac{W_{(\Sigma,h(\Omega),h(D_1))}}{W_{(\Sigma,h(\Omega),h(D_1))} \cap W_{(\Sigma,h(\Omega),h(D_1))}^{\perp}}.
\]
Thus, the two polynomials must have the same symplectic rank. Note that this argument is essentially identical to the one given by Carter-Silver-Williams (see \cite{CSW}, Proposition 5.5).
\end{proof}

\begin{theorem}[Main Theorem] \label{thm_main} If  $D$ is a diagram of a non-split link $ \mathscr{L} \subset \Sigma \times [0,1]$, and $L$ is its virtual link type, then for all $m,n \ge 1$:
\[
\text{rk}_s\left(\widetilde{f}^{\,\,m|n}_{(\Sigma,\Omega,D)}\right) \le 2 \cdot \widecheck{g}_2(L).
\]
\end{theorem}
\begin{proof} By Carter-Silver-Williams \cite{CSW}, Theorem 6.1 (rewritten as Theorem \ref{thm_csw} above), $\widecheck{g}_2(L)$ is twice the symplectic rank of $\widetilde{\pi}_{\mathscr{L}}$. By definition, $\text{rk}_s(\widetilde{\pi}_{\mathscr{L}})$ is the minimum of the symplectic ranks of all operator group presentations of $\widetilde{\pi}_{\mathscr{L}}$. Let $\widetilde{P}$ be an arbitrary operator group presentation  corresponding to a choice of symplectic basis $\Omega$. For this choice, there is a corresponding $U_q(\mathfrak{gl}(m|n))$ polynomial $\widetilde{f}^{\,m|n}(q,x_1,y_1,\ldots,x_g,y_g)$. By Theorem \ref{thm_prismatic_link_group}, it follows that the operators appearing in the presentation $\widetilde{P}$ of $\widetilde{\pi}_{\mathscr{L}}$ are the same as the variables which appear in $\Zh_{\Omega}^S(D)$. $W_{(\Sigma,\Omega,D)}$ is a submodule of $W_{\widetilde{P}}$ and hence $\text{rk}_s(\widetilde{f}^{\,\,m|n}_{(\Sigma,\Omega,D)})$ is at most the symplectic rank of $\widetilde{P}$. On the other hand, Lemmas \ref{lemma_symp_rank_reidemeister} and \ref{lemma_symp_rank_cob} imply that $\text{rk}_s(\widetilde{f}^{\,\,m|n}_{(\Sigma,\Omega,D)})$ is an invariant of links in $\Sigma \times [0,1]$. Hence, it can be no more than  $\text{rk}_s(\widetilde{P})$, where $\widetilde{P}$ is a presentation realizing the minimum $\text{rk}_s(\widetilde{\pi}_{\mathscr{L}})$.   
\end{proof}

\section{The generalized Alexander polynomial and the homology\!\!\! $\Zh$-construction} \label{sec_GAP_and_homology_zh}

The aim of this section is to relate the prismatic $U_q(\mathfrak{gl}(1|1))$ Reshetikhin-Turaev polynomial, the CSW polynomial, and the generalized Alexander polynomial. First we need the following theorem, which states the homology\! $\Zh$-construction can be obtained from the homotopy\! $\Zh$-construction for an appropriately chosen surface $\Sigma$ by making the symplectic palette $\Omega$ monochrome.

\begin{figure}[htb]
\begin{tabular}{|ccc|} \hline & & \\ & 
\includegraphics[scale=.6]{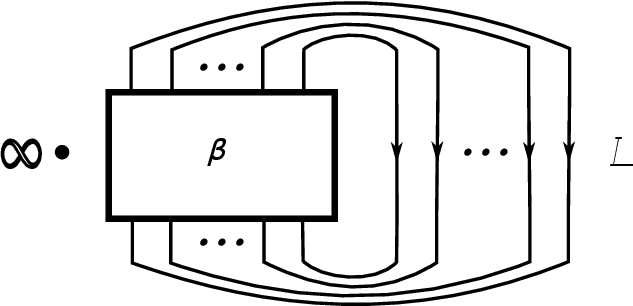} & \\ & & \\ \hline
\end{tabular}
\caption{The base point $\infty$ is always on the left.} \label{fig_pt_at_infty}
\end{figure}

\begin{theorem} \label{thm_zhop_and_prismatic} Let $L$ be a virtual link diagram. Then there is a surface $\Sigma$, a diagram $D$ on $\Sigma$ representing $L$, and a symplectic basis $\Omega=\{x_1,y_1,\ldots,x_g,y_g\}$ such that $\Zh^{\text{op}}(L)$ is semi-welded equivalent to $\Zh_{\Omega}(D)$ by setting $x_i,y_i \to \omega$ for all $i$ and joining all $\omega$-colored components into one.  
\end{theorem}

\begin{remark} We emphasize in the above theorem statement that we are working in the virtual category, and not in the rotational category. This is allowed for our application to the generalized Alexander polynomial since $m=n=1$.
\end{remark}

\begin{proof} By Theorem \ref{thm_prismatic_braid_closure}, we may write $L=\widehat{\beta}$ for some $\beta \in VB_N$. Choose a point $\infty$ to the left of $L$, as shown in Figure \ref{fig_pt_at_infty}. Then $L$ may be viewed as a diagram on $\mathbb{S}^2=\mathbb{R}^2 \cup \{\infty\}$. To create $\Sigma$ and $D$, add $1$-handles to $\mathbb{S}^2$ at each of the virtual crossings $\chi_i$ in the braid word for $\beta$. See Figure \ref{fig_add_one_handles}. For each $\chi_i$ appearing in the word $\beta$, choose the curves $x_j,y_j$ shown in Figure \ref{fig_choose_symp_basis}, left, as its contribution to the symplectic basis. Numbering each pair from the bottom of the virtual braid  word to its top, the result is a symplectic basis $\Omega=\{x_1,y_1,\ldots,x_g,y_g\}$, where $g$ is the number of virtual crossings in $\beta$. Now, create a polygonal region by cutting along all the curves which appear in $\Omega$. Near each virtual crossing in $\beta$, the result resembles Figure \ref{fig_choose_symp_basis}, right. With these preparations, the homotopy $\Zh$-construction can now be performed. If the curlicues shown in Figure \ref{fig_choose_symp_basis} are pulled leftwards, we have the picture shown in the top left in Figure \ref{fig_done_with_mv_zh}. This configuration can be rearranged by following the arrows $\raisebox{.5pt}{\textcircled{\raisebox{-.9pt} {1}}}$ and $\raisebox{.5pt}{\textcircled{\raisebox{-.9pt} {2}}}$. Finally, identify the colors $x_i,y_i$ to be the same color $\omega$ and reconnect all of the $\omega$-colored over-crossing arcs into a single component. Because of the semi-welded move, this can be done arbitrarily. In Figure \ref{fig_done_with_mv_zh}, bottom left, observe that the over-crossing arcs have been reconnected so that some Reidemeister 2 moves are evident. Removing these arcs gives the picture shown on the left in Figure \ref{fig_pull_off_single_color_zh}. Since we are in the virtual category, the curlicues can also be removed, as shown on the right in Figure \ref{fig_pull_off_single_color_zh}. Thus, the effect of $\Zh_{\Omega}$ is to flank each virtual crossing of $\beta$ with a pair of over-crossing $\omega$-colored arcs. In \cite{chrisman_todd_23}, Definition 3.1.5, this was called a \emph{virtual Alexander system} for $L$. Then by \cite{chrisman_todd_23}, Theorem A, this is equivalent to the $\Zh^{\text{op}}$-construction.
\end{proof}

\begin{figure}[htb]
\begin{tabular}{|ccccc|} \hline & & & & \\ &
\begin{tabular}{c}
\includegraphics[scale=.6]{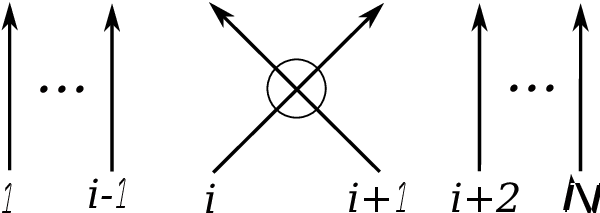}
\end{tabular} & \begin{tabular}{c} $\longrightarrow$ \end{tabular} & \begin{tabular}{c}
\includegraphics[scale=.65]{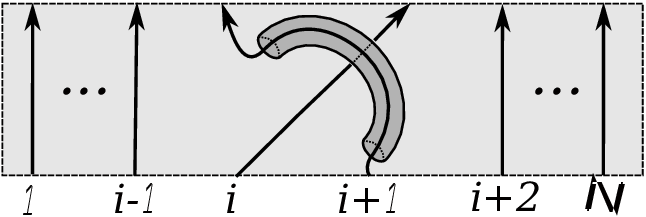} \end{tabular} & \\ & & & & \\ \hline
\end{tabular}
\caption{Adding $1$-handles at the virtual crossings in a virtual braid.}\label{fig_add_one_handles}
\end{figure}

\begin{figure}[htb]
\begin{tabular}{|ccccc|} \hline & & & &  \\
& \begin{tabular}{c} \includegraphics[scale=.58]{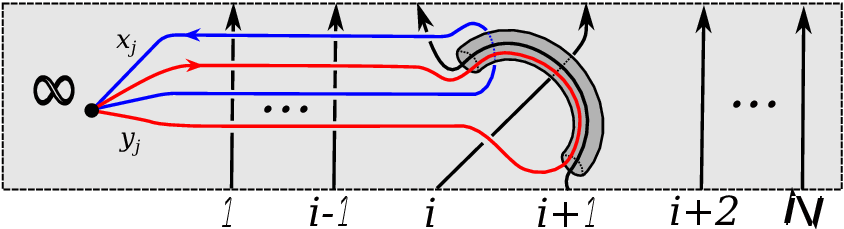} \end{tabular} & \hspace{-.5cm} \begin{tabular}{c} $\longrightarrow$ \end{tabular} & \hspace{-1cm} \begin{tabular}{c} \includegraphics[scale=.33]{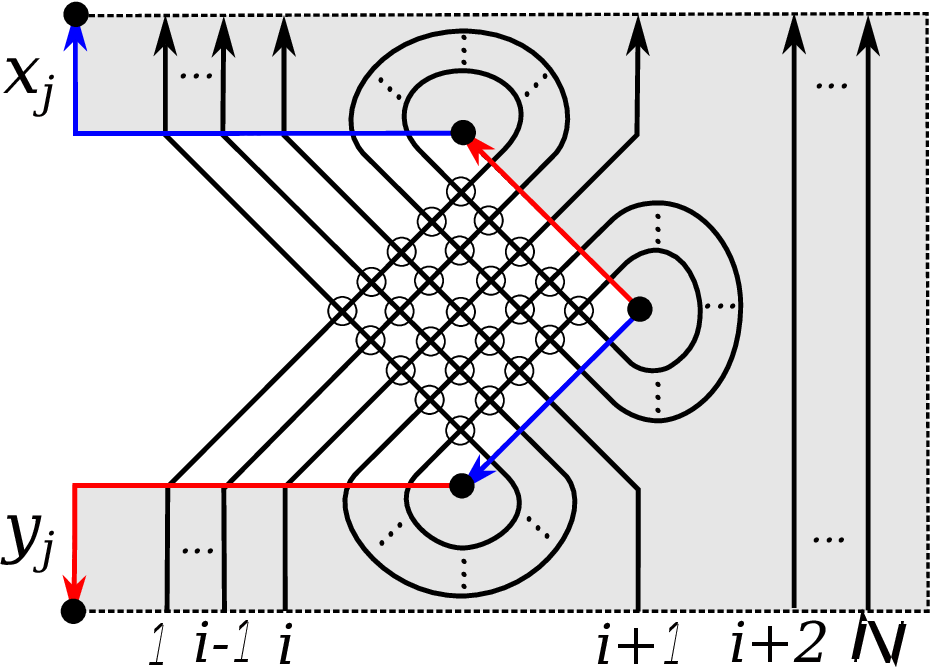} \end{tabular} &  \\ & & & & \\ \hline
\end{tabular}
\caption{(Left) A choice of symplectic basis for each $1$-handle. (Right) Cutting along the symplectic basis to obtain a virtual link diagram.} \label{fig_choose_symp_basis}
\end{figure}

\begin{figure}[htb]
\begin{tabular}{|c|} \hline  
\xymatrix{
 \begin{tabular}{c} \includegraphics[scale=.5]{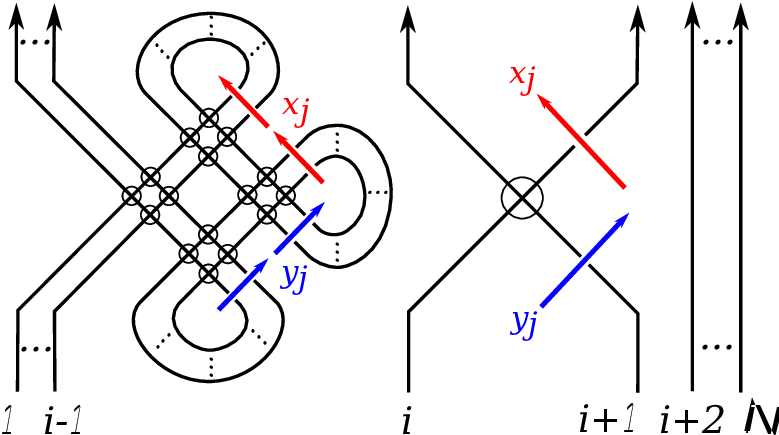} 
\end{tabular} \ar[r]^-{\raisebox{.5pt}{\textcircled{\raisebox{-.9pt} {1}}}} & \ar[d]^--{\raisebox{.5pt}{\textcircled{\raisebox{-.9pt} {2}}}} \begin{tabular}{c} \includegraphics[scale=.5]{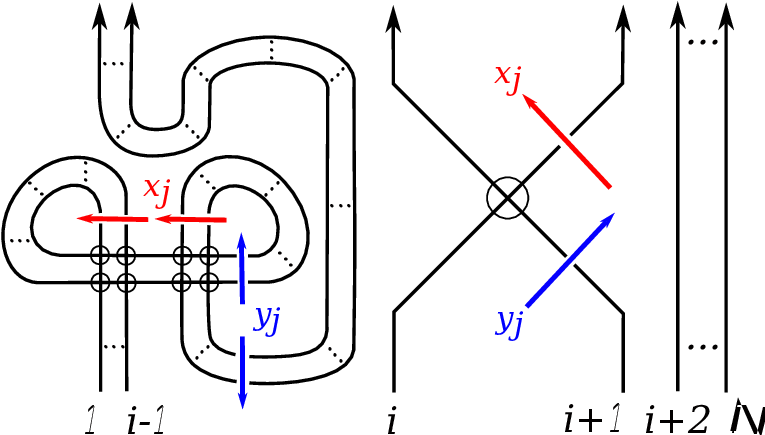}
\end{tabular}\\  \begin{tabular}{c} \includegraphics[scale=.5]{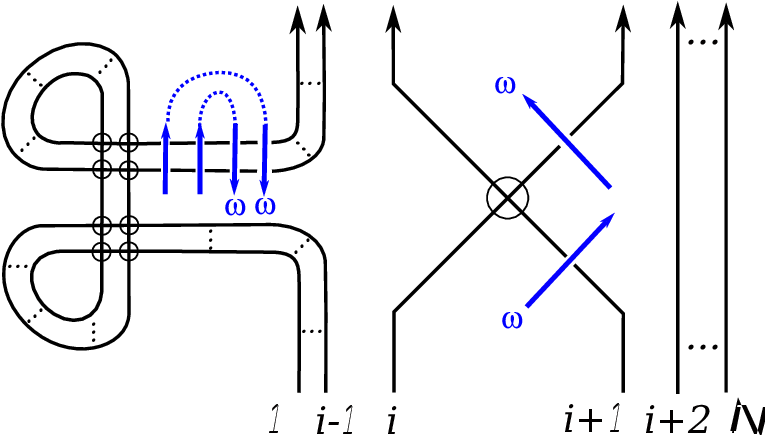} \end{tabular} & \ar[l]^--{\raisebox{.5pt}{\textcircled{\raisebox{-.9pt} {3}}}} \begin{tabular}{c} \includegraphics[scale=.5]{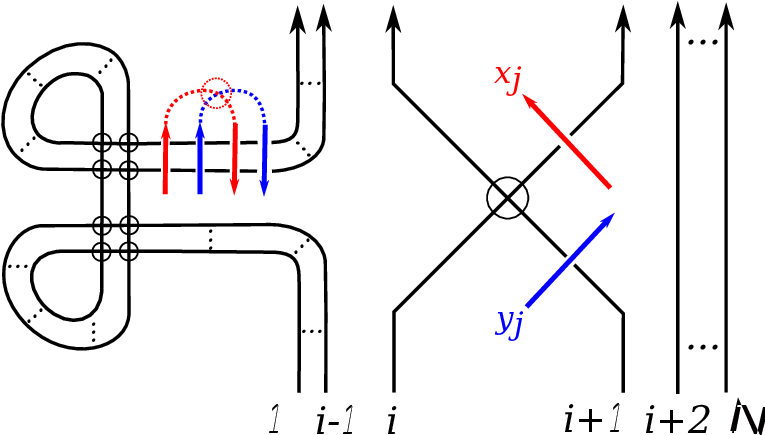} 
\end{tabular}} \\ \hline
\end{tabular}
\caption{Homotopy $\Zh$-construction for $D$ as in  Figure \ref{fig_choose_symp_basis}.} \label{fig_done_with_mv_zh}
\end{figure}

\begin{figure}[htb]
\begin{tabular}{|c|} \hline \\
\xymatrix{\begin{tabular}{c}\includegraphics[scale=.5]{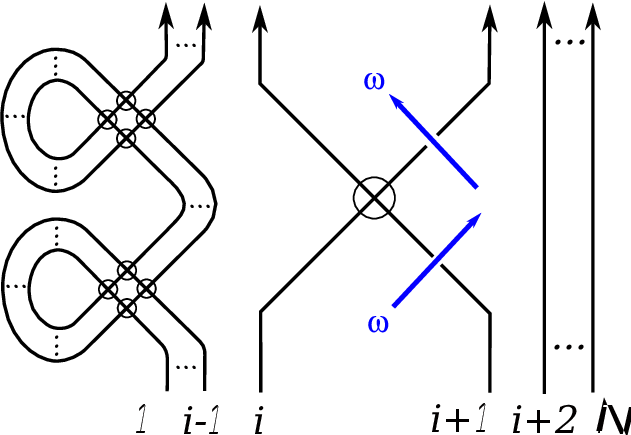} \end{tabular} \ar[r] & \begin{tabular}{c} \includegraphics[scale=.5]{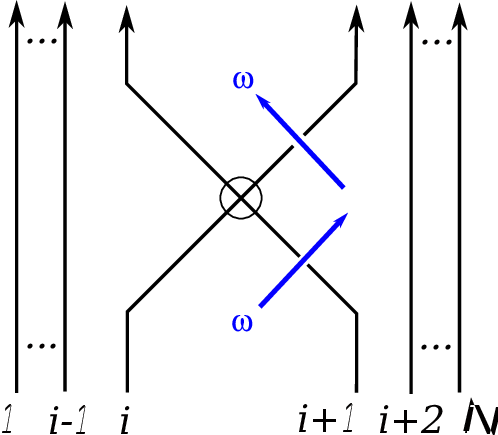} \end{tabular}}
\\ \\ \hline
\end{tabular}
\caption{Identifying the colors in $\bm{\omega}$ to a single one, $w$, gives the $\Zh$-construction.} \label{fig_pull_off_single_color_zh}
\end{figure}

In \cite{boden_chrisman_21}, it was proved that the generalized Alexander polynomial $G_L(s,t)$ of a virtual link $L$ factors through the (homology) $\Zh$-construction. Furthermore, the generalized Alexander polynomial of $L$ was shown to be equal to the (usual) multivariable Alexander polynomial of the $(n+1)$-component link $\Zh(L)=L \cup \omega$. The extra variable in the GAP corresponds to the extra component. This was used to give a quantum model for the GAP in \cite{CP}, where the GAP was also generalized to virtual tangles. Indeed, the GAP is the extended $U_q(\mathfrak{gl}(1|1))$ Reshetikhin-Turaev functor of $\Zh(L)$. By the preceding theorem, the $\Zh^{\text{op}}$-construction can be obtained from the homotopy $\Zh$-construction by setting all of the colors in $\Omega$ to a single color $\omega$. This suggests the GAP is a specialization of the prismatic $U_q(\mathfrak{gl}(1|1))$ Reshetikhin-Turaev polynomial, and hence, also the CSW polynomial.

\begin{theorem} \label{thm_gap_from_prism} Let $L$ be a virtual link diagram and let $\Sigma$, $\Omega$, and $D$ be as in Theorem \ref{thm_zhop_and_prismatic}. After a change of variables, and up to multiples of powers of $q$, the generalized Alexander polynomial $L$ can be obtained from $\widetilde{Q}^{1|1}_{\Omega}(\Zh_{\Omega}(D))$ by the substitutions $q \to q^{-1}$, $x_i,y_i \to w^{-1}$ for all $i$. In particular, the GAP is a specialization of the CSW polynomial $\Delta_D^0(t,x_1,y_1,\ldots,x_g,y_g)$.
\end{theorem}
\begin{proof} Denote the extended $U_q(\mathfrak{gl}(1|1))$ Reshetikhin-Turaev functor of a semi-welded link $J$ from \cite{CP} by $\widetilde{F}^{1|1}_J$. Then by \cite{CP}, Theorem 6.3.1, we have that:
\[
G_{L}(q^{-2}w^{-1},w) \doteq \widetilde{F}^{1|1}_{\zh(L)}(q,w).
\] 
Two changes are needed to relate $\widetilde{F}^{1|1}_{\zh(L)}$ and $\widetilde{Q}_{\Omega}^{1|1}(\Zh_{\Omega}(D))$. First, recall that $\widetilde{F}^{1|1}_{-}$ and $\widetilde{Q}^{1|1}_{\Omega}(-)$ use opposite conventions for the choice of the $R$-matrix at the positive and negative crossings. Agreement can the be forced by making the substitution $q \to q^{-1}$. Note that this also requires reordering the chosen basis for tensor products in the definition of $\widetilde{Q}^{1|1}_{\Omega}$, but this does not effect the value of the invariant. Secondly, Theorem \ref{thm_zhop_and_prismatic} implies that  setting $x_i,y_i \to \omega$ in $\Zh_{\Omega}(D)$ gives the $\Zh^{\text{op}}$-construction. The definition of $\widetilde{Q}^{1|1}_{\Omega}(-)$ in Section \ref{sec_uqglmn_rt} shows that changing the orientation of an $\Omega$-colored over-crossing arc changes its value by the substitution $w^{\pm 1} \to w^{\mp 1}$. Putting these two observations gives the substitution $x_i,y_i \to w^{-1}$ for all $i$. Then by Theorem \ref{thm_quantum_csw_model}, we have that:
\[
G_{L}(q^{-2}w^{-1},w) \doteq \left(\widetilde{Q}_{\Omega}^{1|1}(\Zh_{\Omega}(D))\right)(q^{-1},w^{-1},w^{-1},\ldots,w^{-1},w^{-1}) \doteq \Delta_D^0\left(q^2,w,w,\ldots,w,w\right). \qedhere
\]
\end{proof}


\section{The homotopy \!$\Zh$-construction $\&$ the surface bracket} \label{sec_homotopy_and_surface}

In Section \ref{sec_surf}, we define a notion of symplectic rank for the Dye-Kauffman surface bracket. Then it is shown that the condition for minimality of a representative from \cite{dye_kauffman_surf} is equivalent to the maximality of this symplectic rank. Following this, in Section \ref{sec_surf_factors}, we prove that the surface bracket factors through the homotopy $\Zh$-construction.


\begin{figure}[htb]
\begin{tabular}{|ccc|}  \hline & & \\
\begin{tabular}{c} \includegraphics[scale=.6]{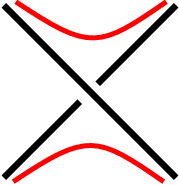} \\ \underline{$A$ smoothing:} \end{tabular} & & \begin{tabular}{c} \includegraphics[scale=.6]{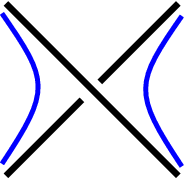} \\ \underline{$B$ smoothing:}\end{tabular}\\ & & \\ \hline
\end{tabular}
\caption{Smoothings used to create a state of a diagram.} \label{fig_smoothings}
\end{figure}

\subsection{$\mathbb{Z}_2$ symplectic rank of the surface bracket} \label{sec_surf} We briefly review the definition of the surface bracket \cite{dye_kauffman_surf}. Let $D$ be a link diagram on a surface $\Sigma$. A \emph{state} $S$ of $D$ is a choice of $A$ or $B$ smoothing at each crossing of $D$. See Figure \ref{fig_smoothings}. Denote by $\#_A(S)$ the number of $A$-smoothings used in the creation of the state and $\#_B(S)$ the number $B$-smoothings. Each state $S$ is a collection of simple closed curves on $\Sigma$. Let $\#_0(S)$ be the number of these simple closed curves which bound a disc on $\Sigma$ and let $[S]$ denote the formal sum of curves in $S$ that do not bound discs. Let $\mathscr{S}(D)$ be the set of all states of $D$. The \emph{surface bracket} $\langle D \rangle_{\Sigma}$ of $D$ is given by: 
\begin{equation} \label{defn_surf_bracket}
\langle D \rangle_{\Sigma}=\sum_{S \in \mathcal{S}(D)} A^{\#_A(S)-\#_B(S)} d^{\#_0(S)}[S], \quad d=-A^2-A^{-2}.
\end{equation}
The formal sums $[S]$ are considered equivalent up to isotopy on $\Sigma$. The surface bracket may then be viewed as an element of the Kauffman bracket skein module on $\Sigma$. Hence, $\langle D \rangle_{\Sigma}$ is an invariant of framed Reidemeister equivalence on $\Sigma$. The following useful result of Dye and Kauffman shows how the surface bracket can be used to verify that a given representative of a virtual link on a surface realizes the minimal genus. 

\begin{theorem}[Dye-Kauffman \cite{dye_kauffman_surf}, Theorem 2.4] \label{thm_dye_kauffman}  Decompose $\Sigma$ as a connect sum of tori, $\Sigma=T_1 \# T_2 \# \cdots \# T_g$. Let $D$ be a link diagram on $\Sigma$ and let $\mathscr{S}(D)=\{S_1,\ldots,S_r\}$ be the set of states of $D$ on $\Sigma$. Let $\kappa_k:\Sigma \to T_k$ be the ``collapsing map'' and $(\kappa_k)_*:H_1(\Sigma;\mathbb{Z}) \to H_1(T_k;\mathbb{Z})$ the induced map. If for each $k$, there are states $S_i$, $S_j$ which have non-zero coefficients in $\langle D \rangle_{\Sigma}$ and which contain (arbitrarily oriented) state curves $Y_i$, $Y_j$, respectively, such that $(\kappa_k)_*([Y_i])\cdot (\kappa_k)_*([Y_j]) \ne 0$, then $D$ is a minimal genus representative for the virtual link type of $D$.
\end{theorem}
\begin{remark} The map $\kappa_k$ for $\Sigma= T_1\#\cdots \# T_{k-1} \# T_k \# T_{k+1} \# \ldots \#T_g$ identifies $T_1\#\cdots \# T_{k-1}$ to a single point $\star_{k-1}$ in $T_k$ and $T_{k+1} \# \cdots \# T_g$ to a single point $\star_{k+1}$ with $\star_{k+1} \ne \star_{k-1}$.
\end{remark}

Henceforth, the hypothesis in Theorem \ref{thm_dye_kauffman} guaranteeing minimality will be called the \emph{Dye-Kauffman condition}. The main result of this section is that the Dye-Kauffman condition can be restated in terms of the symplectic rank of $\langle D \rangle_{\Sigma}$, appropriately defined. First observe that the surface bracket applies to unoriented links and that the condition in the theorem is obtained by applying arbitrary orientations to the state curves. This means that it only makes sense to define the symplectic rank in $H_1(\Sigma;\mathbb{Z}_2)$, rather than in $H_1(\Sigma;\mathbb{R})$ as in the case of the prismatic $U_q(\mathfrak{gl}(m|n))$ polynomials. In practice, however, the surface bracket is often computed modulo $\mathbb{Z}_2$ (see e.g. \cite{dye_kauffman_surf}, Corollary 4.2), so this change of coefficients does pose not a serious threat to its utility. Now, let $D$ be a link diagram on a surface $\Sigma$ and let $\Omega=\{x_1,y_1,\ldots,x_g,y_g\}$ be a choice of symplectic basis. For a state $S$ of $D$, we assign a variable to $[S]$ as follows. The variable is viewed as an element of the group ring $\mathbb{Z}[H_1(\Sigma;\mathbb{Z}_2)]$. For each state curve $Y$ of $S$, we have the variable corresponding to its homology class in $\mathbb{Z}_2$, written multiplicatively. Then sum over all $Y$ in $[S]$:
\[
\sum_{Y \in [S]} \left( \prod_{i=1}^g x_i^{[Y]\cdot [y_i]} y_i^{[Y] \cdot [x_i]}\right),
\]
where $-\cdot-$ represents the $\mathbb{Z}_2$ intersection form on $\Sigma$. This will be called the \emph{state variable}. Now, let $W_{SB}$ be the submodule of $H_1(\Sigma;\mathbb{Z}_2)$ generated by the state variables of states $S$ such that $[S]$ has a nonzero coefficient in $\langle D\rangle_{\Sigma}$. More precisely, if $[S]$ has a nonzero coefficient, write each of the terms in its state variable additively. Then take the submodule generated by all such vectors of state curves for all $S$ such that $[S]$ has a nonzero coefficient. The \emph{$\mathbb{Z}_2$-symplectic rank}  $\text{rk}_{s}^{\mathbb{Z}_2}(\langle D \rangle_{\Sigma})$ of $\langle D \rangle_{\Sigma}$ is the symplectic rank of $W_{SB}$. That is, $\text{rk}_{s}^{\mathbb{Z}_2}(\langle D \rangle_{\Sigma})=\text{rank}(W_{SB}/(W_{SB} \cap W_{SB}^{\perp}))$.

\begin{lemma} \label{lemma_sb_sr} The $\mathbb{Z}_2$-symplectic rank of the surface bracket  $\langle D \rangle _{\Sigma}$ is maximal  (i.e. $=2 \cdot g(\Sigma)$) if and only if the Dye-Kauffman condition over $\mathbb{Z}_2$ (see Theorem \ref{thm_dye_kauffman}) is satisfied.
\end{lemma}
\begin{proof} First observe that the coefficients for a single state are self-orthogonal, so that contributions to the $\mathbb{Z}_2$-symplectic rank must come from variables in different states. This can be seen by writing out all of the curves in the state in terms of the symplectic basis. Suppose that a state $S$ consists of curves $Y_1,\ldots,Y_s$. Then write:
\[
[Y_i]=\sum_{j=1}^g \alpha_{ij} [x_j]+\beta_{ij} [y_j], \quad \alpha_{ij} \in \mathbb{Z}_2.
\]
Since the state curves have no geometric intersections, it must be that for a fixed $i$ and $j$, $\alpha_{ij}=\alpha_{kj}$ and $\beta_{ij}=\beta_{kj}$ for all $k$, $1 \le k \le s$. Hence, if $M$ is the submodule of $H_1(\Sigma;\mathbb{Z}_2)$ generated by the coefficients for a single state, it follows that $M=M^{\perp}$.

Now, suppose the $\mathbb{Z}_2$-symplectic rank is maximal. Then for each $j$, there must be a state $S$ containing a curve $Y_{i}$ such that $\alpha_{ij}=1$ and a state $S'$ containing a curve $Y_{k}'$ such that $\beta_{kj}'=1$. Furthermore, the state $S'$ may be chosen so that $\beta_{ij} \alpha_{kj}'=0$, for otherwise it would be impossible to generate the submodule $\langle x_j,y_j \rangle$ needed for maximality. By the argument in the preceding paragraph, $S$ and $S'$ must be distinct states. Now we simply compute (with $\mathbb{Z}_2$ coefficients):
\[
(\kappa_j)_*([Y_i]) \cdot (\kappa_j)_*(Y_k')=\begin{pmatrix}
\alpha_{ij} & \beta_{ij}
\end{pmatrix} \begin{pmatrix}
0 & 1 \\ 1 & 0
\end{pmatrix} \begin{pmatrix}
\alpha_{kj}' \\ \beta_{kj}'
\end{pmatrix}= \alpha_{ij} \beta_{kj}'+\beta_{ij} \alpha_{kj}'=
1
\]
This completes the forward direction of the proof. Conversely, suppose that the condition of Theorem \ref{thm_dye_kauffman} is satisfied over $\mathbb{Z}_2$. Then for each $j$, there are states $S$, $S'$ containing curves $Y_i$, $Y_k'$ such that $(\kappa_j)_*([Y_i]) \cdot (\kappa_j)_*(Y_k')=1$. Now write out all the possibilities for the coefficients in homology given this condition:
\begin{align*}
\alpha_{ij}=\beta_{ij}=1, & \quad \alpha_{kj}'=0, \beta_{kj}'=1,\\
\alpha_{ij}=\beta_{ij}=1, & \quad \alpha_{kj}'=1, \beta_{kj}'=0,\\
\alpha_{kj}'=\beta_{kj}'=1, & \quad \alpha_{ij}=0, \beta_{ij}=1,\\
\alpha_{kj}'=\beta_{kj}'=1, & \quad \alpha_{ij}=1, \beta_{ij}=0,\\
\alpha_{ij}=\beta_{kj}'=1, & \quad \alpha_{kj}'=\beta_{ij}=0,\\
\alpha_{kj}'=\beta_{ij}=1, & \quad \alpha_{ij}=\beta_{kj}'=0.
\end{align*}
In each case, it follows that $\langle \alpha_{ij} [x_j]+\beta_{ij} [y_j], \alpha_{kj}' [x_j]+\beta_{kj}' [y_j] \rangle=\langle [x_j],[y_j] \rangle$. The vectors corresponding to all of the state curves for all of the states can thus be put into reduced row echelon form over $\mathbb{Z}_2$ so that the rank is $2g$. Hence, the $\mathbb{Z}_2$-symplectic rank is maximal.
\end{proof}

\begin{theorem} If the $\mathbb{Z}_2$-symplectic rank of the surface bracket $\langle D \rangle_{\Sigma}$ is maximal, then  $D$ is a minimal genus representative of the virtual link type of $D$.
\end{theorem}
\begin{proof} By Lemma \ref{lemma_sb_sr}, the condition of  Theorem \ref{thm_dye_kauffman} with $\mathbb{Z}_2$ coefficients is satisfied. But if the condition is satisfied with coefficients over $\mathbb{Z}_2$, then they must also be satisfied over $\mathbb{Z}$. Indeed, the state curves that satisfy the condition with $\mathbb{Z}_2$ coefficients must have odd intersection numbers with $\mathbb{Z}$ coefficients. Hence, the result follows by Theorem \ref{thm_dye_kauffman}.
\end{proof}

\subsection{Factoring the surface bracket through homotopy\!\! $\Zh$} \label{sec_surf_factors} In \cite{chrisman_todd_23}, Section 4.3, it was proved that the arrow polynomial factors through the homology $\Zh$-construction. Here we generalize the argument to show that the surface bracket factors through the homotopy $\Zh$-construction. We begin by defining a bracket polynomial $\langle\langle \widecheck{L} \rangle\rangle_{\Omega}$ for $\widecheck{L}$ a complete \emph{unoriented} $\Omega$-prismatic link diagram. Such links will be considered equivalent up to unoriented extended Reidemeister moves, unoriented $\Omega$-semi-welded moves, and oriented $\Omega$-commutator moves. Let $\widecheck{D}$ denote the $\alpha$-part of $\widecheck{L}$. By an \emph{$\Omega$-prismatic state} $\widecheck{S}$ of $\widecheck{L}$, we mean a choice of $A$ or $B$-smoothing at each of the classical crossings of $\widecheck{D}$. Then $\widecheck{S}$ is an unoriented complete $\Omega$-prismatic link diagram. For each such state, there is an $\Omega$-prismatic state variable. To define it, we will use the notation of Proposition \ref{prop_coordinatize_zh}: $\widecheck{z}_i$ denotes the component of $\widecheck{L}$ colored by $z_i \in \Omega$. Let $[\widecheck{S}]_{\alpha}$ denote the set of $\alpha$-colored state curves in an $\Omega$-prismatic state $\widecheck{S}$. Then the corresponding $\Omega$-prismatic state variable is given by: 
\[ 
\sum_{\widecheck{Y} \in [\widecheck{S}]_{\alpha}} \left( \prod_{i=1}^g x_i^{\vlk(\widecheck{x_i},\widecheck{Y})} y_i^{\vlk(\widecheck{y_i},\widecheck{Y})}\right),
\]
where the exponents on the right hand side are calculated in $\mathbb{Z}_2$. This is required to make $\vlk(\widecheck{z}_i,\widecheck{Y})$ well defined for unoriented diagrams. Now, let $\#_0(\widecheck{S})$ denote the number of components in the $\alpha$-part of  $\widecheck{S}$ that are unlinked from the $\Omega$-part. This equals the number of unknotted components in the $\alpha$-part that are split from the $\Omega$-part. Let $\mathscr{S}(\widecheck{L})$ be the set of $\Omega$-prismatic states. Then define:
\[
\langle \langle \widecheck{L} \rangle \rangle_{\Omega}=\sum_{\widecheck{S} \in \mathscr{S}(\widecheck{L})} A^{\#_A(\widecheck{S})-\#_B(\widecheck{S})} d^{\#_0(\widecheck{S})} \left(\sum_{\widecheck{Y} \in [\widecheck{S}]_{\alpha}} \left( \prod_{i=1}^g x_i^{\vlk(\widecheck{x_i},\widecheck{Y})} y_i^{\vlk(\widecheck{y_i},\widecheck{Y})}\right)\right).
\]
\begin{proposition} $\langle\langle \widecheck{L} \rangle \rangle_{\Omega}$ is an invariant of framed\footnote{Recall that this means the classical Reidemeister 1 move is not allowed on the $\alpha$-part} unoriented $\Omega$-prismatic links.
\end{proposition}
\begin{proof} That $\langle\langle \widecheck{L} \rangle \rangle_{\Omega}$ is invariant under the extended Reidemeister moves, except Reidemeister 1, follows as in the case of the bracket polynomial for unoriented virtual links. See also \cite{chrisman_todd_23}, Proposition 4.2.2. An $\Omega$-semi-welded move does not involve any classical crossings in the $\alpha$-part and hence cannot effect the polynomial weight in the variable $A$ for any state. The prismatic state variable is also unchanged by such a move as the virtual linking number between the the $\Omega$-over-crossing arc and the $\alpha$-colored arcs are the same on both sides of the move. Lastly, note that an $\Omega$-commutator move has no effect on $\langle \langle \widecheck{L} \rangle \rangle_{\Omega}$. Indeed, on the left side of the move in Figure \ref{fig_commutator_move}, each component of the $\Omega$-part has 2 intersections with the $\alpha$-part. Hence, the $\Omega$-colored over-crossing arcs contribute trivially to the prismatic state variables on both sides of the move.
\end{proof}

Now we can prove the main result of this section, which essentially says that the surface bracket with $\mathbb{Z}_2$ homological coefficients can be written as a composition of $\langle\langle \cdot \rangle \rangle_{\Omega}$ and $\Zh_{\Omega}(D)$.

\begin{theorem} Let $D$ be a link diagram on a surface $\Sigma$ and let $\Omega=\{x_1,y_1,\ldots,x_g,y_g\}$ be a choice of symplectic basis. For every state $[S]$ of $D$, replace $[S]$ with its state variable in $\langle D \rangle_{\Sigma}$. Then: 
$$\langle D \rangle_{\Sigma}=\langle \langle \Zh_{\Omega}(D) \rangle\rangle_{\Omega}.$$
\end{theorem}

\begin{proof} There is a one-to-one correspondence between the states $S$ of $D$ on $\Sigma$ and the $\Omega$-prismatic states $\widecheck{S}$ of $\Zh_{\Omega}(D)$. Furthermore, for each state $S$, there is a one-to-one correspondence between $[S]$ and $[\widecheck{S}]_{\alpha}$. 
The result then follows by applying Proposition \ref{prop_coordinatize_zh} with $\mathbb{Z}_2$ coefficients, since:
\[
\sum_{Y \in [S]} \left( \prod_{i=1}^g x_i^{[Y]\cdot [y_i]} y_i^{[Y] \cdot [x_i]}\right)=\sum_{\widecheck{Y} \in [\widecheck{S}]_{\alpha}} \left( \prod_{i=1}^g x_i^{\vlk(\widecheck{x_i},\widecheck{Y})} y_i^{\vlk(\widecheck{y_i},\widecheck{Y})}\right). \qedhere
\]
\end{proof}

\section{Examples $\&$ Applications} \label{sec_applications}
In this section, we compare the virtual $2$-genus bounds from the prismatic $U_q(\mathfrak{gl}(m|n))$ polynomials to those of other well-known invariants. The two main geometric applications (non-additivity of the virtual 3-genus and necessity of stabilization to realize the GAP) are also discussed. All calculations below were done in \emph{Mathematica} \cite{mathematica}. A supplementary notebook is available here:
\centerline{
\url{https://github.com/micah-chrisman/virtual-knots/blob/main/prismatic_zh_final_switch.nb}}
In figures below, ``$\circlearrowleft$'' denotes a tangle that should be rotated by $\pi/2$ to calculate the invariants. 
\subsection{Running example: the right virtual trefoil} \label{sec_ex_simple} For completion, we give the value of some prismatic $U_q(\mathfrak{gl}(m|n)))$ polynomials for the right-handed virtual trefoil. Let $D$ be the diagram on the torus $\Sigma$ as shown in Figure \ref{fig_mvzh_example} and $\Omega=\{x_1,y_1\}$ the symplectic basis. A tangle decomposition of $\Zh_{\Omega}(D)$ is shown in Figure \ref{fig_virt_tref_mvzh_tangle}. Some prismatic polynomials are:
\begin{align*}
\widetilde{f}^{\,\,1|1}_{(\Sigma,\Omega,D)}(q,x_1,y_1) &= -\frac{q^2 x_1}{y_1}-\frac{x_1}{q^2}+\frac{q^2}{y_1}+\frac{1}{q^2}+x_1-\frac{1}{y_1},\\
\widetilde{f}^{\,\,2|1}_{(\Sigma,\Omega,D)}(q,x_1,y_1) &= \frac{q^6}{y_1}-q^5-\frac{q^4 x_1}{y_1}+q^3+q^2 x_1-\frac{x_1}{q^2}-\frac{q^2}{y_1}-q+\frac{1}{q}+2, \\
\widetilde{f}^{\,\,3|1}_{(\Sigma,\Omega,D)}(q,x_1,y_1) &= \frac{q^{10}}{y_1}-q^9-q^7-\frac{q^6 x_1}{y_1}+2 q^5+q^4 x_1-\frac{q^4}{y_1}-2 q^3-\frac{x_1}{q^2}+3 q^2+q+\frac{1}{q},\\
\widetilde{f}^{\,\,2|2}_{(\Sigma,\Omega,D)}(q,x_1,y_1) &= -\frac{x_1}{q^4}+\frac{q^4}{y_1}-\frac{q^3 x_1}{y_1}+\frac{x_1}{q^3 y_1}-q^3+\frac{1}{q^3}-\frac{2 q^2 x_1}{y_1}+q^2 x_1-\frac{x_1}{q^2}+\frac{q^2}{y_1}-\frac{1}{q^2 y_1}+\frac{2}{q^2}\\ &+\frac{q x_1}{y_1}-\frac{x_1}{q y_1}+q-\frac{1}{q}+x_1-\frac{1}{y_1}.
\end{align*}
Note that in each case, the submodule generated by the coefficients of powers of $q$ is all of $H_1(\Sigma;\mathbb{R})$. Each has a symplectic rank of $2$, and as expected, the virtual $2$-genus of the virtual trefoil is 1. 

\begin{figure}[htb]
\begin{tabular}{|ccc|} \hline & & \\ & \includegraphics[scale=.6]{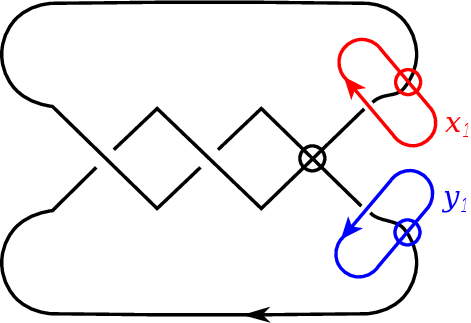} & \\ & & \\ \hline
\end{tabular}
\caption{Tangle decomposition ($\circlearrowleft$) of the right virtual trefoil from Figure \ref{fig_mvzh_example}.} \label{fig_virt_tref_mvzh_tangle}
\end{figure}

Next we verify the relationship between the prismatic $U_q(\mathfrak{gl}(1|1))$ polynomial, the CSW polynomial, and the GAP. From Section \ref{sec_csw_example}, we have that the CSW polynomial is given by:
\[
\Delta_K^0(q^{-2},x_1^{-1},y_1^{-1}) \doteq -\frac{1}{q^4 y_1}+\frac{1}{q^4}-\frac{x_1}{q^2}+\frac{1}{q^2 y_1}-\frac{x_1}{y_1}+x_1.
\]
The prismatic link $\Zh_{\Omega}(D)$ in Figure \ref{fig_virt_tref_mvzh_tangle} can be represented as the closure of a prismatic braid $\beta \in VB_{2,\Omega}$, where $\beta=\lambda_{2,y_1}^{-1}\lambda_{1,x_1} (\chi_1 \otimes 1_2) (\sigma_1^{-1} \otimes 1_2)^2$. Then the substitution $t \to q^{-2}$ , the change of basis, and the determinant formula (Theorem \ref{thm_CSW_from_braid}) give the correct value after scaling by $q^{-4}$:
\[
q^{-4} \cdot \det\left(\rho_{N,\Omega}(\lambda_{2,y_1}^{-1})\rho_{N,\Omega}(\lambda_{1,x_1})\rho_{N,\Omega}(\chi \otimes 1_2)(\rho_{N,\Omega}(\sigma_1^{-1} \otimes 1_2))^2-I_2\right)=-\frac{1}{q^4 y_1}+\frac{1}{q^4}-\frac{x_1}{q^2}+\frac{1}{q^2 y_1}-\frac{x_1}{y_1}+x_1.
\]
Now, apply the determinant-trace formula \emph{and the change of basis} for $\bigwedge\!^* U \to V^{\otimes 2}$. After multiplying this by $q^{-4}$, we get exact agreement with the prismatic $U_q(\mathfrak{gl}(1|1))$ polynomial given above:
\begin{align*}
& q^{-4} \cdot \text{tr}\left(\begin{pmatrix}
 1 & 0 & 0 & 0 \\
 0 & \frac{1}{y_1} & 0 & 0 \\
 0 & 0 & 1 & 0 \\
 0 & 0 & 0 & \frac{1}{y_1} \\
\end{pmatrix} \begin{pmatrix}
 1 & 0 & 0 & 0 \\
 0 & 1 & 0 & 0 \\
 0 & 0 & x_1 & 0 \\
 0 & 0 & 0 & x_1 \\
\end{pmatrix} \begin{pmatrix}
    1 & 0 & 0 & 0 \\
    0 & 0 & q & 0 \\
    0 & \tfrac{1}{q} & 0 & 0 \\
    0 & 0 & 0 & -1
\end{pmatrix} \begin{pmatrix}
    1 & 0 & 0 & 0 \\
    0 & 0 & q & 0 \\
    0 & q & 1-q^2 & 0 \\
    0 & 0 & 0 & -q^2
\end{pmatrix}^2 \cdot \mu^{\otimes 2} \right) \\
=&-\frac{q^2 x_1}{y_1}-\frac{x_1}{q^2}+\frac{q^2}{y_1}+\frac{1}{q^2}+x_1-\frac{1}{y_1}
\end{align*}
Lastly, consider the GAP. The GAP of the right virtual trefoil can be obtained from that of the left virtual trefoil by the substitution $s \to s^{-1}$ and $t \to t^{-1}$. From Green's table \cite{green}, we then have:
\[
G_K(s,t)=-s^2 t^2+s^2 t+s t^2-s-t+1 \implies G_K(q^{-2} w^{-1},w)=\frac{1}{q^4 w}-\frac{1}{q^4}+\frac{w}{q^2}-\frac{1}{q^2 w}-w+1
\]
The conventions for the positive and negative crossings in our paper agree with those from \cite{green}, which are the same as those in Kauffman-Radford \cite{kauffman_radford_00}. Hence, to obtain the GAP from the prismatic $U_q(\mathfrak{gl}(1|1))$ polynomial, the substitution is $x_1,y_1 \to w^{-1}$. Compare, also, with the CSW polynomial:
\begin{align*}
(-1) \cdot \Delta^0_K(q^{-2},w^{-1},w^{-1}) &= \frac{1}{q^4 w}-\frac{1}{q^4}+\frac{w}{q^2}-\frac{1}{q^2 w}-w+1,\\
(-q^{-2}) \cdot \widetilde{f}_{(\Sigma,\Omega,D)}^{\,\,1|1}(q,w^{-1},w^{-1}) &=\frac{1}{q^4 w}-\frac{1}{q^4}+\frac{w}{q^2}-\frac{1}{q^2 w}-w+1.
\end{align*}

Note that Theorem \ref{thm_zhop_and_prismatic} only guarantees the existence of a surface $\Sigma$ and symplectic basis $\Omega$ for which the substitutions $x_i,y_i \to w^{-1}$, $q \to q^{-1}$ in $\widetilde{Q}^{1|1}_{\Omega}(K)$ is the GAP. This particular pair $\Sigma,\Omega$ apply for the right virtual trefoil because the $\omega$ component represents exactly the homology class $-[D] \in H_1(\Sigma)$, as can be seen in Figure \ref{fig_rel_seif}. Otherwise, the substitution need not equal the GAP.

\subsection{Generalized $r$-Kishino knots}\label{sec_kishino} The generalized fly knot is the $1$-$1$ virtual tangle shown in Figure \ref{fig_gen_n_kishino}, left. The two classical crossings can be positively or negatively signed. For $r \ge 1$, a \emph{generalized $r$-Kishino knot} $K_r$ is the connected sum of $r+1$ generalized fly knots. See Figure \ref{fig_gen_n_kishino}, right, where each of the $2r+2$ classical crossings may be signed positive or negative. In \cite{adams_tg}, Adams et al. proved that the virtual $2$-genus of $K_r$ is $r+1$. In brief, they show that $K_r$ can always be realized as a tg-hyperbolic knot on a surface of genus $r+1$. Since a hyberbolic knot can have no essential annulus in its complement, no destabilizing annulus exists. Hence, the virtual $2$-genus is realized. For further examples of tg-hyperbolic composite virtual knots see Adams-Simons \cite{adams_simons}. The virtual $2$-genus for the $r=1$ case, with various choices of the classical crossing signs, has been well studied. See, for example, Dye-Kauffman \cite{dye_kauffman_surf}, Carter-Silver-Williams \cite{CSW}, and Miller \cite{miller}.   

\begin{figure}[htb]
\begin{tabular}{|cc|} \hline &  \\\begin{tabular}{c} \includegraphics[scale=.4]{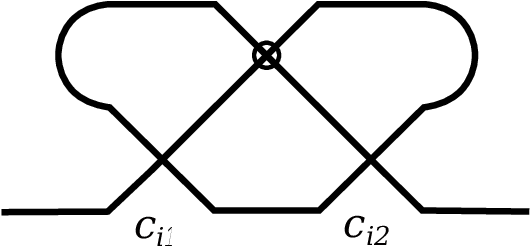} \\ \underline{Generalized Fly:} \end{tabular} & \begin{tabular}{c} \includegraphics[scale=.4]{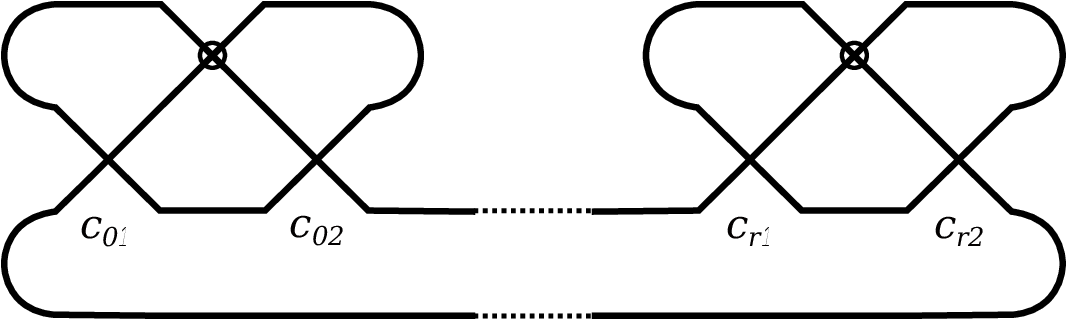} \\ \underline{Generalized $r$-Kishino:} \end{tabular} \\ &  \\ \hline
\end{tabular}
\caption{The crossings in a generalized fly or $r$-Kishino knot may have sign $c_{ij}= \pm 1$.} \label{fig_gen_n_kishino}
\end{figure}

Here we will prove the virtual $2$-genus of $K_r$ is $r+1$ using the prismatic $U_q(\mathfrak{gl}(1|1))$ polynomial. Label the classical crossings as from left to right as $(0,1)$, $(0,2)$, $\ldots$, $(r,1)$, $(r,2)$. Let $c_{ij}=\pm 1$ denote the choice of the sign at the crossing $(i,j)$. Figure \ref{fig_gen_n_kishino_screen} shows  $K_r$ can be realized as a knot diagram on a surface $\Sigma$ of genus $r+1$. Performing the homotopy $\Zh$-construction with these choices gives the virtual link diagram shown in the figure. This can be decomposed as $\curvearrowright \circ ((T_{r} \circ \cdots \circ T_0) \otimes \downarrow ) \circ \raisebox{\depth}{\rotatebox{180}{$\curvearrowright$}}$, where $T_i$ is a generalized fly with some $\Omega$-colored components. The composition $T_{r} \circ \cdots \circ T_0$ is a $1$-$1$ virtual tangle diagram, so that $\widetilde{Q}_{\Omega}^{1|1}(T_{r} \circ \cdots \circ T_0)$ is a $2 \times 2$ diagonal matrix $\text{diag}(\alpha_1,\alpha_2)$. Then:
\begin{equation} \label{eqn_gen_r_kishino}
\widetilde{Q}_{\Omega}^{1|1}(\curvearrowright \circ ((T_{r} \circ \cdots \circ T_0) \otimes \downarrow ) \circ \raisebox{\depth}{\rotatebox{180}{$\curvearrowright$}})= \begin{pmatrix} q & 0 & 0 & -q \end{pmatrix} \begin{pmatrix}
\alpha_1 & 0 & 0 & 0 \\ 0 & \alpha_1 & 0 & 0 \\ 0 & 0 & \alpha_2 & 0 \\ 0 & 0 & 0 & \alpha_2 
\end{pmatrix} \begin{pmatrix}
1 \\ 0 \\ 0 \\ 1
\end{pmatrix}=q(\alpha_1-\alpha_2).
\end{equation}
To compute the symplectic rank, it remains to determine the contribution of each $T_i$ to the terms $\alpha_1,\alpha_2$. The contribution of $T_i$ for each of the four possible crossings signs is given below.  
\[
\begin{array}{cc}
\underline{\begin{array}{c} c_{i1}=+1 \\ c_{i2}=-1\end{array} :}\,\, \left(
 \begin{array}{cc}
 -\frac{q^3}{x_i}+\frac{q}{x_i}+q & 0 \\
 0 & -\frac{q^3}{x_i}+\frac{q y_i}{x_i}+\frac{q}{x_i} \\
\end{array}
\right) & \underline{\begin{array}{c} c_{i1}=+1 \\ c_{i2}=+1\end{array} :}\,\,\left(
\begin{array}{cc}
 \frac{1}{q} & 0 \\
 0 & \frac{q^3 y_i}{x_i}-\frac{q^3}{x_i}+\frac{q}{x_i}-q y_i+\frac{y_i}{q} \\
\end{array}
\right) \\ & \\
\underline{\begin{array}{c} c_{i1}=-1 \\ c_{i2}=+1\end{array} :}\,\,
\left(
\begin{array}{cc}
 -\frac{q^3}{x_i}+q^3+\frac{q}{x_i}-q y_i+\frac{y_i}{q} & 0 \\
 0 & \frac{y_i}{q x_i} \\
\end{array}
\right) & \underline{\begin{array}{c} c_{i1}=-1 \\ c_{i2}=+1\end{array} :}\,\,\left(
\begin{array}{cc}
 -q y_i+\frac{y_i}{q}+q & 0 \\
 0 & \frac{q y_i}{x_i}-q y_i+\frac{y_i}{q} \\
\end{array}
\right)
\end{array}
\]
Then $\alpha_1$, $\alpha_2$ are products of polynomials taken from these four matrices.  Suppose, for example, suppose that for $i \ne j$, $T_i$ and  $T_j$ are both of the first type of matrix, so that $c_{i1}=c_{j1}=+1$, $c_{i2}=c_{j2}=-1$. Then Equation \ref{eqn_gen_r_kishino} must contain terms of the form $q^k/x_i$, $q^k/x_j$, $q^k y_j/(x_i x_j)$, and $q^k y_i/(x_j x_i)$. Writing these coefficients additively, the submodule they generate is given by $\langle -x_i,-x_j,y_j-x_j-x_i,y_i-x_i-x_j \rangle=\langle x_i,y_i,x_j,y_j \rangle$. Likewise for the other 15 possible choices of $T_i$ and $T_j$, the submodule generated can be easily seen to be $\langle x_i,y_i,x_j,y_j \rangle$. The total contribution of the nonzero coefficients of $q^k$ in the prismatic $U_q(\mathfrak{gl}(1|1))$ polynomial is therefore all of $H_1(\Sigma;\mathbb{R})$. Hence the symplectic rank is $2(r+1)$ and the virtual $2$-genus is $r+1$. 
\begin{figure}[htb]
\begin{tabular}{|ccc|} \hline & & \\ & \includegraphics[scale=.65]{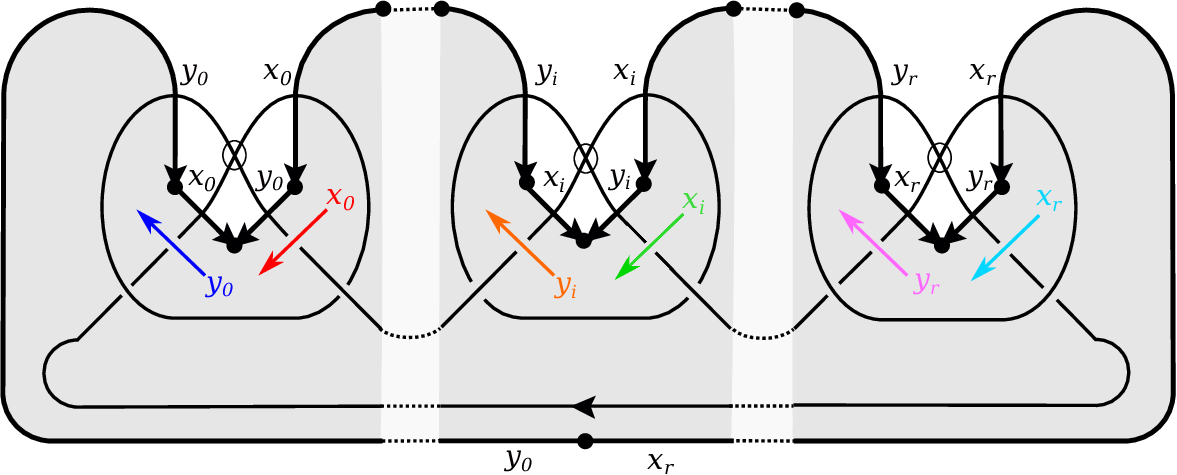} & \\ & & \\ \hline
\end{tabular}
\caption{Homotopy $\Zh$-construction for a generalized $r$-Kishino knot.} \label{fig_gen_n_kishino_screen}
\end{figure}

\subsection{Satellite virtual knots} \label{sec_ex_CSW_compare} In \cite{CSW}, Example 7.2, the CSW polynomial of a satellite knot $K$ in the thickened torus was calculated to be $\Delta_K^0(t,x_1,y_1)=(t-1)(1-x_1 y_1)^2$. The knot, which is a Whitehead double of the right-handed virtual trefoil on the torus, is shown in Figure \ref{fig_satellite_tangle}. The symplectic rank then is $0$, so that the virtual 2-genus is not detected directly. The CSW polynomial can nonetheless used to determine the virtual 2-genus, since the virtual 2-genus of a satellite knot is equal to that of its companion (see Silver-Williams \cite{silver_williams_sat}). In this case, the companion has virtual $2$-genus 1 (see Section \ref{sec_ex_simple}). The prismatic $U_q(\mathfrak{gl}(2|1))$ invariant detects the virtual $2$-genus of this satellite knot \emph{directly}. Note also that $K$ is not tg-hyperbolic, so that hyperbolicity cannot be used to directly determine the virtual $2$-genus. The $U_q(\mathfrak{gl}(1|1))$ and $U_q(\mathfrak{gl}(2|1))$ polynomials are: 
\begin{align*}
\widetilde{f}^{\,\,1|1}_{(\Sigma,\Omega,D)}(q,x_1,y_1) &= q x_1 y_1+\frac{q}{x_1 y_1}-\frac{1}{q x_1 y_1}-\frac{x_1 y_1}{q}-2 q+\frac{2}{q} \\
\widetilde{f}^{\,\,2|1}_{(\Sigma,\Omega,D)}(q,x_1,y_1)&=\frac{q^{10} y_1}{x_1}-2 q^{10}+\frac{q^9}{x_1}+q^9 y_1-\frac{4 q^8 y_1}{x_1}+6 q^8-q^7 x_1 y_1-\frac{q^7}{x_1 y_1}+q^7 x_1-\frac{3 q^7}{x_1}-3 q^7 y_1\\
&+\frac{q^7}{y_1}+\frac{7 q^6 y_1}{x_1}+\frac{q^6 x_1}{y_1}+\frac{x_1}{q^6 y_1}+\frac{q^6}{x_1}+q^6 y_1-8 q^6+2 q^5 x_1 y_1-\frac{2 q^5 y_1}{x_1}+\frac{2 q^5}{x_1 y_1}+\frac{2 x_1}{q^5 y_1}\\&-3 q^5 x_1+\frac{3 q^5}{x_1}-\frac{2 x_1}{q^5}+3 q^5 y_1-\frac{3 q^5}{y_1}-\frac{2}{q^5 y_1}-\frac{8 q^4 y_1}{x_1}-\frac{4 q^4 x_1}{y_1}-\frac{4 x_1}{q^4 y_1}+q^4 x_1-\frac{2 q^4}{x_1}\\&-\frac{x_1}{q^4}-2 q^4 y_1+\frac{q^4}{y_1}-\frac{1}{q^4 y_1}+10 q^4+\frac{8}{q^4}-2 q^3 x_1 y_1+\frac{6 q^3 y_1}{x_1}-\frac{2 q^3}{x_1 y_1}-\frac{6 x_1}{q^3 y_1}+3 q^3 x_1\\&-\frac{q^3}{x_1}+\frac{4 x_1}{q^3}-\frac{2}{q^3 x_1}-q^3 y_1+\frac{3 q^3}{y_1}-\frac{2 y_1}{q^3}+\frac{4}{q^3 y_1}+2 q^2 x_1 y_1+\frac{7 q^2 y_1}{x_1}+\frac{7 q^2 x_1}{y_1}+\frac{2 q^2}{x_1 y_1}\\&+\frac{y_1}{q^2 x_1}+\frac{7 x_1}{q^2 y_1}-2 q^2 x_1+\frac{2 x_1}{q^2}-\frac{1}{q^2 x_1}-\frac{2 q^2}{y_1}-\frac{y_1}{q^2}+\frac{2}{q^2 y_1}-12 q^2-\frac{14}{q^2}+2 q x_1 y_1\\&-\frac{6 q y_1}{x_1}-\frac{2 q x_1}{y_1}+\frac{2 q}{x_1 y_1}+\frac{2 y_1}{q x_1}+\frac{6 x_1}{q y_1}-\frac{1}{q x_1 y_1}-\frac{x_1 y_1}{q}-q x_1-\frac{2 q}{x_1}+\frac{4}{q x_1}-\frac{2 x_1}{q}\\&-2 q y_1-\frac{q}{y_1}-\frac{2}{q y_1}+\frac{4 y_1}{q}-2 x_1 y_1-\frac{4 y_1}{x_1}-\frac{8 x_1}{y_1}-\frac{2}{x_1 y_1}+\frac{2}{x_1}+2 y_1+13
\end{align*}
Note that symplectic rank for the $U_q(\mathfrak{gl}(1|1))$ case is 0, which is the expected result by Theorem \ref{thm_quantum_csw_model}. The symplectic rank of the $U_q(\mathfrak{gl}(2|1))$ is $2$, giving the value of $1$ for the virtual $2$-genus.

\begin{figure}[htb]
\begin{tabular}{|ccc|} \hline & & \\ \begin{tabular}{c} \includegraphics[scale=.47]{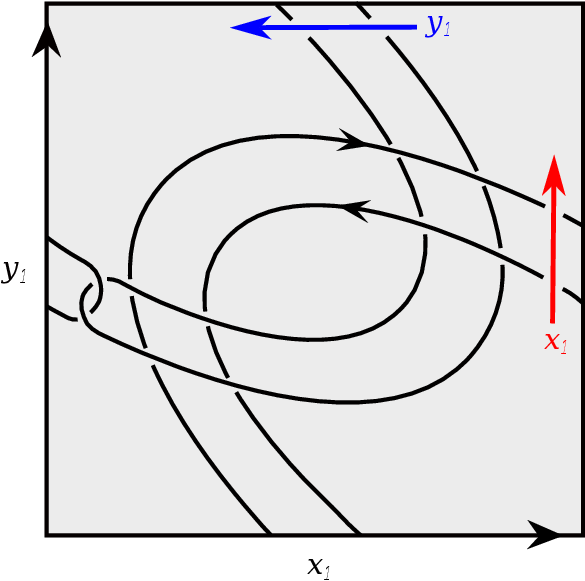} \end{tabular} & \begin{tabular}{c} \includegraphics[scale=.6]{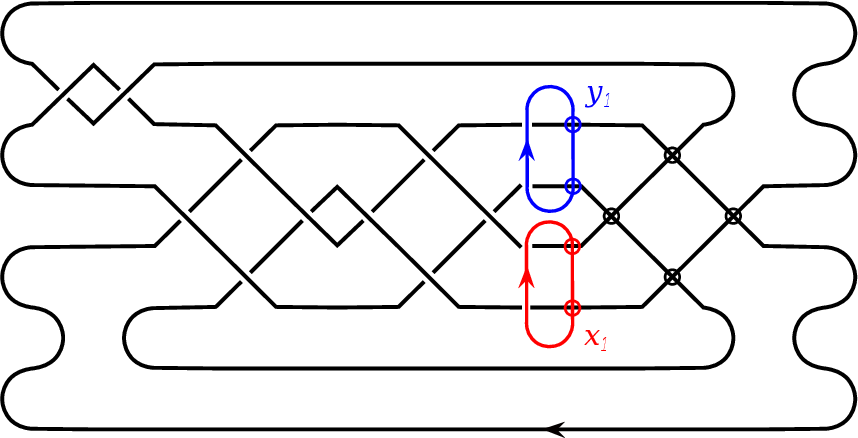} \end{tabular} & \\ & & \\ \hline
\end{tabular}
\caption{A satellite knot and its homotopy $\Zh$-construction ($\circlearrowleft$).} \label{fig_satellite_tangle}
\end{figure} 

Up to the indeterminancies for the prismatic $U_q(\mathfrak{gl}(1|1))$ polynomial and the CSW polynomial, $\widetilde{f}^{\,1|1}_{(\Sigma,\Omega,D)}$ and $\Delta_D^0$ can be seen to agree after the substitution $t \to q^{-2}$, $x_1,y_1 \to x_1^{-1},y_1^{-1}$ Indeed:
\[
(-q x_1 y_1) \Delta_K^0(q^{-2},x_1^{-1},y_1^{-1})=\frac{(q-1) (q+1) (x_1 y_1-1)^2}{q x_1 y_1}=\widetilde{f}^{\,1|1}_{(\Sigma,\Omega,D)}(q,x_1,y_1)
\]

\subsection{Virtual 2-genus and the GAP} \label{sec_v2g_GAP} Here we show that the GAP is not always realizable as a specialization of the CSW polynomial on a surface of minimal genus. The virtual knot $K=3.5$ from Green's table \cite{green} has virtual $2$-genus 1. The GAP is given by:
\begin{align}
G_K(s,t)&=-\frac{1}{s^3 t^3}+\frac{1}{s^3 t}+\frac{1}{s^2 t^2}-\frac{1}{s^2}+\frac{1}{s t^3}-\frac{1}{s t}-\frac{1}{t^2}+1 \\
\label{eqn_no_realize_factor} \implies G_K(q^{-2}w^{-1},w) &= \frac{(q-1) (q+1) q^3 (w-1) (w+1) \left(q^2 w-1\right) \left(q^2 w+1\right)}{w^2}
\end{align}
It is known that $\widecheck{g}_2(K)=1$. A diagram $D$ on the torus is shown in Figure \ref{fig_no_realize}, left. The homotopy $\Zh$-construction is shown in Figure \ref{fig_no_realize}, right. Note that $K$ can be represented as the closure of the braid beta $\beta$ shown in Figure \ref{fig_no_realize_braid}. Since $\beta$ has 4 virtual crossings, Theorem \ref{thm_gap_from_prism} implies $G_K(s,t)$ can be realized as a specialization of the CSW polynomial on a surface of genus $4$. Now we show this cannot be done on the torus. In other words, one must stabilize the genus to realize the GAP.

\begin{figure}[htb]
\begin{tabular}{|ccc|} \hline  \begin{tabular}{c}\\ \includegraphics[scale=.4]{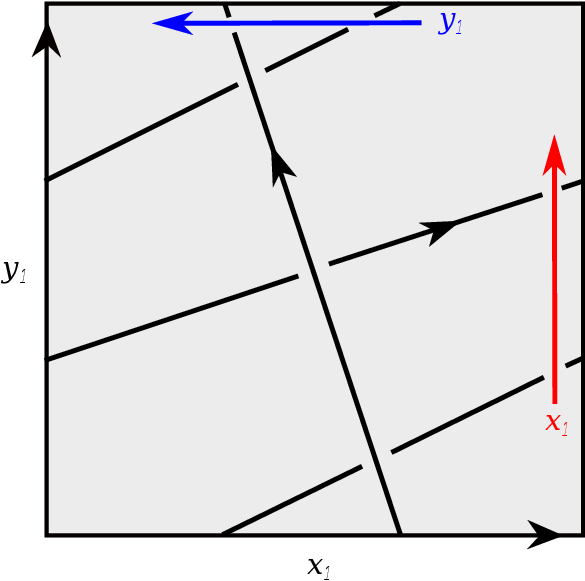} \end{tabular} & \begin{tabular}{c} \includegraphics[scale=.6]{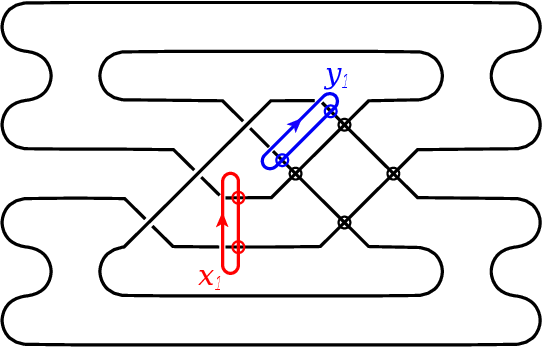} \end{tabular} & \\ \hline
\end{tabular}
\caption{Homotopy $\Zh$-construction (right, $\circlearrowleft$) for the virtual knot 3.5 (left).} \label{fig_no_realize}
\end{figure}

It must be proved that $\widetilde{f}^{\,\,1|1}_{(\Sigma,\Omega,D)}(q,x_1,y_1)$ is not equivalent to a polynomial which specializes to $G_K(q^{-2} w^{-1},w)$. The prismatic $U_q(\mathfrak{gl}(1|1))$ polynomial is given by:
\[
\widetilde{f}^{\,\,1|1}_{(\Sigma,\Omega,D)}(q,x_1,y_1)=-\frac{x_1^2 y_1^2}{q^3}+\frac{x_1^2 y_1}{q^3}-q^3 y_1+q^3-\frac{x_1^2 y_1}{q}-q x_1 y_1+\frac{x_1 y_1}{q}+q y_1
\] 
Now suppose that $\begin{pmatrix} a & b \\ c & d \end{pmatrix} \in \text{Sp}(2,\mathbb{Z})$ represents a symplectic change of basis, so that $ad-bc=1$. Then $x_1 \to a x_1+c y_1$ and $y_1 \to b x_1 +d y_1$. During the specialization, the $x_1$ coordinate becomes $(a+c) w^{-1}$ and the $y_1$ coordinate becomes $(b+d) w^{-1}$. Write $z=a+c$ and $p=b+d$. Then:
\begin{align} \label{eqn_no_realize_prismatic}
\widetilde{f}^{\,\,1|1}_{(\Sigma,\Omega,D)}(1,z w^{-1},p w^{-1})&=-\frac{\left(p z-w^2\right) \left(p z+w^2\right)}{w^4}.
\end{align}
By Equation \ref{eqn_no_realize_factor}, it follows that evaluating $q$ to $1$ gives $0$. In Equation (\ref{eqn_no_realize_prismatic}), we must have that $p$ and $z$ are integers. But no such value of $p$ and $z$ can make  (\ref{eqn_no_realize_prismatic}) vanish. Hence, the GAP of 3.5 cannot be realized as a specialization of the CSW polynomial on a  minimal genus surface.
\begin{figure}[htb]
\begin{tabular}{|ccc|} \hline & & \\ & \includegraphics[scale=.5]{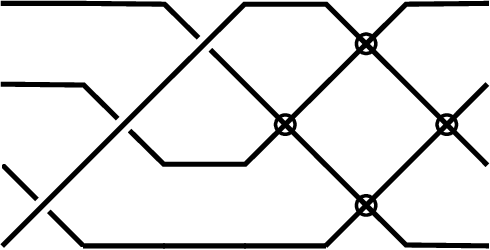} & \\ & & \\ \hline
\end{tabular}
\caption{($\circlearrowleft$) A virtual braid whose closure is 3.5.} \label{fig_no_realize_braid}
\end{figure}

\subsection{Checkerboard-colorable knots} \label{sec_check} In the case of a checkerboard-colorable virtual knot, the Gordon-Litherland pairing gives a simple determinant test for checking the minimality of a diagram. Here we will show that the prismatic $U_q(\mathfrak{gl}(m|n))$ polynomials can detect the virtual 2-genus for checkerboard-colorable knots even when the determinant test fails. First we recall some basic definitions. For more details, the reader is referred to \cite{BCK}. 

A virtual knot is said to be \emph{checkerboard-colorable} if it has a $\mathbb{Z}_2$-homologically trivial representative $K$ in a thickened surface $\Sigma \times [0,1]$, i.e. $[K]=0 \in H_1(\Sigma \times [0,1]; \mathbb{Z}_2)$. A minimal genus representative of a checkerboard-colorable knot has a checkerboard-coloring in the usual sense. The regions of the diagram on the surface are discs and each disc may be colored either black or white so that each edge borders both a black region and a white region. Every checkerboard colorable representative bounds a checkerboard spanning surface $F$. For a fixed checkerboard coloring of a knot diagram on a surface $\Sigma$, there are two checkerboard surfaces $F_1,F_2$. Unlike in the classical case, these two spanning surfaces are not $S^*$-equivalent when the genus of $\Sigma$ is greater than $0$ (see \cite{BCK}). Each spanning surface has a corresponding Gordon-Litherland pairing $\mathscr{G}_{F_1},\mathscr{G}_{F_2}$. In \cite{BCK}, Theorem 3.2, it was proved that if $\det(\mathscr{G}_{F_1}) \cdot \det(\mathscr{G}_{F_2}) \ne 0$, then $K$ realizes the virtual $2$-genus.   

\begin{figure}[htb]
\begin{tabular}{|ccc|} \hline & & \\ & \includegraphics[scale=.5]{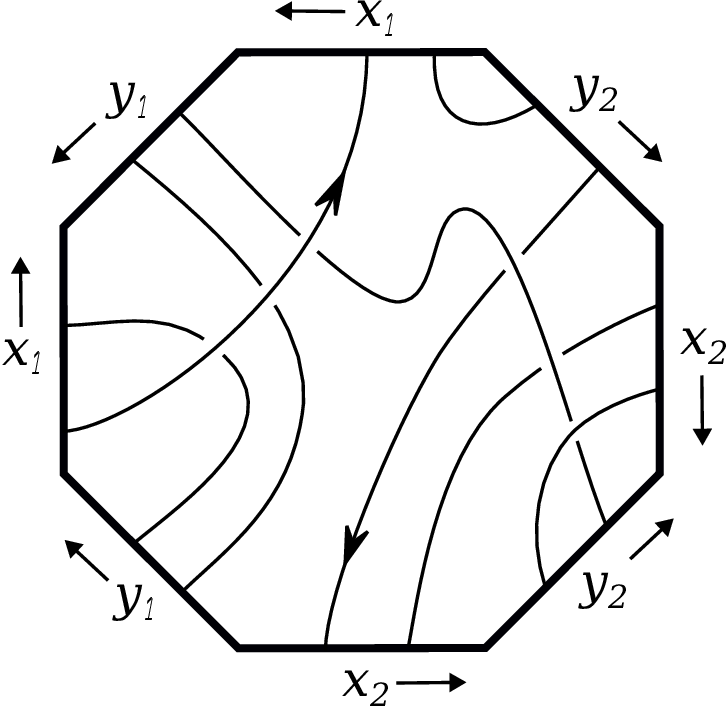} & \\ & & \\ \hline
\end{tabular}
\caption{The virtual knot 6.70394 on a genus $2$ surface.} \label{fig_checkerboard_example}
\end{figure}
    
The value of determinant test is that it can be computed very quickly. The disadvantage is that it does not apply to many virtual knots. Of the 92800 virtual knots up to six classical crossings, only 1674 are checkerboard colorable ($ \approx 1.8\%$). This number includes, for example, all of the alternating virtual knots. When the test applies, however, it is very effective. Of the 1674 checkerboard-colorable knots, the determinant test fails on just 244 of them ($\approx 15 \%$). One such example is 6.70394 from Green's table \cite{green}, which has one zero determinant and one non-zero determinant (see \cite{BCK} and its associated tables). A representative $D$ on a genus 2 surface is shown in Figure \ref{fig_checkerboard_example}. Figure \ref{fig_6pt70394} shows a virtual tangle decomposition of its homotopy $\Zh$-construction using the screen in Figure \ref{fig_screen}. The prismatic $U_q(\mathfrak{gl}(1|1))$ and $U_q(\mathfrak{gl}(2|1))$ polynomials are:
\begin{align*}
\widetilde{f}^{\,\,1|1}_{(\Sigma,\Omega,D)}(q,x_1,y_1,x_2,y_2) &= \frac{x_2}{q^9 x_1^2 y_1^2 y_2^2}-\frac{x_2}{q^9 x_1^2 y_1 y_2^2}-\frac{x_2^2}{q^7 x_1^2 y_1^2 y_2^2}+\frac{x_2^2}{q^7 x_1^2 y_1 y_2^2}-\frac{x_2}{q^7 x_1^2 y_1^2 y_2^2}+\frac{x_2}{q^7 x_1^2 y_1^2 y_2}\\& +\frac{2 x_2}{q^7 x_1^2 y_1 y_2^2}-\frac{x_2}{q^7 x_1^2 y_1 y_2}-\frac{x_2}{q^7 x_1 y_1 y_2^2}-\frac{x_2^2}{q^5 x_1^2 y_1 y_2^2}-\frac{x_2}{q^5 x_1^2 y_1^2 y_2}-\frac{x_2}{q^5 x_1^2 y_1 y_2^2}\\&+\frac{2 x_2}{q^5 x_1^2 y_1 y_2}+\frac{x_2^2}{q^5 x_1 y_1 y_2^2}+\frac{2 x_2}{q^5 x_1 y_1 y_2^2}-\frac{x_2}{q^5 x_1 y_1 y_2}-\frac{x_2}{q^5 y_1 y_2^2}-\frac{x_2}{q^3 x_1^2 y_1 y_2}\\&-\frac{x_2^2}{q^3 x_1 y_1 y_2^2}-\frac{x_2}{q^3 x_1 y_1 y_2^2}+\frac{2 x_2}{q^3 x_1 y_1 y_2}+\frac{x_2^2}{q^3 y_1 y_2^2}+\frac{2 x_2}{q^3 y_1 y_2^2}-\frac{x_2}{q^3 y_1 y_2}\\&-\frac{x_2}{q^3 y_2^2}-\frac{x_2}{q x_1 y_1 y_2}-\frac{x_2^2}{q y_1 y_2^2}-\frac{x_2}{q y_1 y_2^2}-\frac{q x_2}{y_1 y_2}+\frac{2 x_2}{q y_1 y_2}+\frac{x_2}{q y_2^2}+q\\
\widetilde{f}^{\,\,2|1}_{(\Sigma,\Omega,D)}(q,x_1,y_1,x_2,y_2) &=
 \frac{x_2}{q^{16} x_1^2 y_1^2 y_2^2}-\frac{x_2}{q^{16} x_1^2 y_1 y_2^2}+\frac{x_2}{q^{16} y_2^2}-\frac{x_2^2}{q^{14} x_1^2 y_1^2 y_2^2}+\frac{x_2^2}{q^{14} x_1^2 y_1 y_2^2}+\frac{x_2}{q^{14} x_1^2 y_1^2 y_2}\\&-\frac{x_2}{q^{14} x_1^2 y_1 y_2}-\frac{2 x_2}{q^{14} x_1 y_1 y_2^2}-\frac{x_2}{q^{14} y_1 y_2^2}-\frac{x_2}{q^{14} y_2^2}-\frac{x_2}{q^{12} x_1^2 y_1^2 y_2^2}-\frac{x_2}{q^{12} x_1^2 y_1^2 y_2}\\&+\frac{2 x_2}{q^{12} x_1^2 y_1 y_2^2}+\frac{x_2}{q^{12} x_1^2 y_1 y_2}+\frac{2 x_2^2}{q^{12} x_1 y_1 y_2^2}+\frac{2 x_2}{q^{12} x_1 y_1 y_2^2}-\frac{2 x_2}{q^{12} x_1 y_1 y_2}+\frac{x_2^2}{q^{12} y_1 y_2^2}\\&-\frac{x_2}{q^{12} y_1 y_2}-\frac{x_2^2}{q^{10} x_1^2 y_1 y_2^2}+\frac{x_2}{q^{10} x_1^2 y_1^2 y_2}-\frac{2 x_2^2}{q^{10} x_1 y_1 y_2^2}+\frac{2 x_2}{q^{10} x_1 y_1 y_2^2}+\frac{4 x_2}{q^{10} x_1 y_1 y_2}\\&+\frac{2 x_2}{q^{10} y_1 y_2^2}+\frac{x_2}{q^{10} y_1 y_2}-\frac{3 x_2}{q^{10} y_2^2}-\frac{x_2}{q^8 x_1^2 y_1^2 y_2}-\frac{x_2}{q^8 x_1^2 y_1 y_2^2}-\frac{2 x_2}{q^8 x_1 y_1 y_2^2}-\frac{4 x_2}{q^8 x_1 y_1 y_2}\\&-\frac{x_2^2}{q^8 y_1 y_2^2}+\frac{3 x_2}{q^8 y_2^2}+\frac{x_2}{q^6 x_1^2 y_1 y_2}+\frac{4 x_2}{q^6 x_1 y_1 y_2}-\frac{x_2}{q^6 y_1 y_2^2}-\frac{1}{q^6}-\frac{x_2}{q^4 x_1^2 y_1 y_2}\\&-\frac{2 x_2}{q^4 x_1 y_1 y_2}+\frac{x_2}{q^4 y_1 y_2}-\frac{x_2}{q^2 y_1 y_2}+\frac{2}{q^2}+1
\end{align*}

In both cases, the symplectic rank is 4, so that the virtual 2-genus of 2 is detected. In particular, the CSW polynomial detects $
\widecheck{g}_2(K)$, even though the Gordon-Litherland determinants do not. 

\begin{figure}[htb]
\begin{tabular}{|ccc|} \hline & & \\ & \includegraphics[scale=.74]{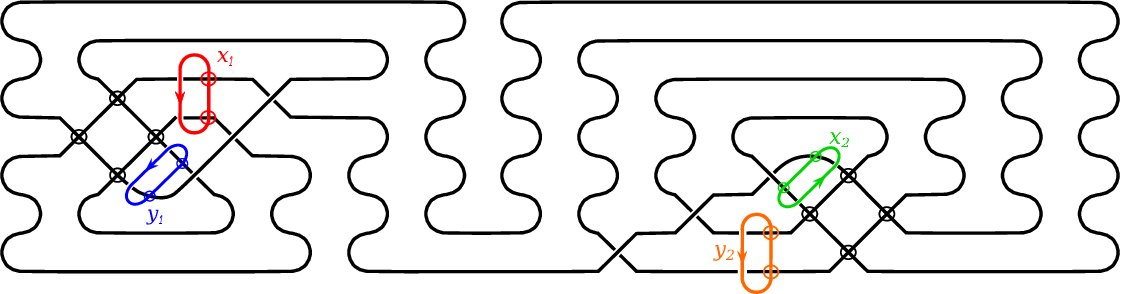} & \\ & & \\ \hline
\end{tabular}
\caption{Homotopy\!\! $\Zh$-construction ($\circlearrowleft$) of 6.70394 from Figure \ref{fig_checkerboard_example}.} \label{fig_6pt70394}
\end{figure}    
    
\subsection{Almost classical knots} Recall that a virtual knot is said to be \emph{almost classical} (AC) if it has a homologically trivial representative $K$ in a thickened surface $\Sigma \times [0,1]$. That is, $[K]=0 \in H_1(\Sigma \times [0,1];\mathbb{Z})$. Related to the surface bracket is the arrow polynomial of Dye-Kauffman \cite{dye_kauffman_arrow}, which is equivalent to the Miyazawa polynomial \cite{miyazawa}. It is also generalization of the Jones polynomial. It contains additional variables which can be used to give a lower bound on the virtual $2$-genus (see Dye-Kauffman \cite{dye_kauffman_arrow}, Theorem 4.5). However, the arrow polynomial of an almost classical knot is equal to its Jones polynomial. For a proof of this fact using the $\Zh$-construction, see \cite{chrisman_todd_23}. As a result, the arrow polynomial can never detect the minimal genus of a (non-classical) almost classical knot.  The satellite knot in Section \ref{sec_ex_CSW_compare} is almost classical as it bounds an evident genus 1 Seifert surface. Hence, we already have an example where the prismatic $U_q(\mathfrak{gl}(2|1))$ polynomial detects its virtual $2$-genus, whereas the arrow polynomial does not. Other examples can be found in Green's virtual knot table \cite{green}. For instance, consider the almost classical knot 4.99. A prismatic $\Zh$-construction is shown in Figure \ref{fig_almost_classical_tangle}. The prismatic $U_q(\mathfrak{gl}(1|1))$ and $U_q(\mathfrak{gl}(2|1))$ polynomials are: 
\begin{align*}
\widetilde{f}^{\,\,1|1}_{(\Sigma,\Omega,D)}(q,x_1,y_1) &=
 q^3 x_1-q^3 y_1-2 q x_1+\frac{x_1}{q}+2 q y_1-\frac{y_1}{q}, \\
\widetilde{f}^{\,\,2|1}_{(\Sigma,\Omega,D)}(q,x_1,y_1) &=q^4 x_1-3 q^4 y_1+\frac{2 x_1}{q^2}+6 q^2 y_1-3 x_1-3 y_1+1
\end{align*}

In both cases, the symplectic rank is easily seen to be 2, so that the virtual 2-genus is 1. It is interesting to note that $4.99$ is also a slice virtual knot, so that its generalized Alexander polynomial, various index polynomials, and graded genus are all trivial. For slice almost classical knots, it is therefore generally difficult to determine the virtual $2$-genus. A more interesting illustration of this will be given in the next subsection. 

\begin{figure}[htb]
\begin{tabular}{|ccccc|} \hline & & & & \\ & \begin{tabular}{c} \includegraphics[scale=.46]{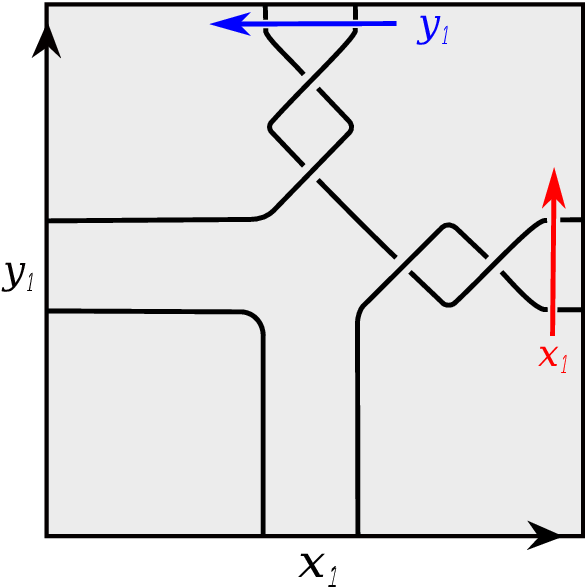} \end{tabular} &  & \begin{tabular}{c} \includegraphics[scale=.56]{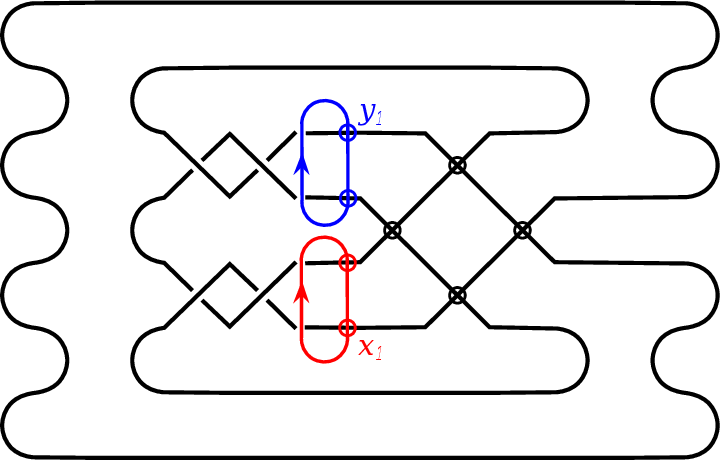} \end{tabular} &  \\ & & & & \\ \hline
\end{tabular}
\caption{Homotopy $\Zh$-construction ($\circlearrowleft$) for the almost classical knot 4.99.} \label{fig_almost_classical_tangle}
\end{figure}

\subsection{The virtual $3$-genus} \label{sec_v3g} The \emph{$3$-genus} $g_3(K)$ of a knot $K$ in $\mathbb{S}^3$ is the smallest genus of any Seifert surface bounded by the knot. The $3$-genus of the unknot is $0$, and the unknot is the only knot having 3-genus $0$. If $K$ is an almost classical knot, it has a homologically trivial representative in some thickened surface. Any homologically trivial representative bounds a Seifert surface in this thickened surface. The \emph{virtual $3$-genus} $\widecheck{g}_3(K)$ of an almost classical knot $K$ is the smallest genus of all its Seifert surfaces, with the minimum taken over all its homologically trivial representatives. The virtual 2-genus and virtual 3-genus of an almost classical knot can be simultaneously realized: a minimal genus knot bounds a Seifert surface of minimal genus (see \cite{BGHNW_17}, Corollary 6.5). 

As is well-known, the classical $3$-genus is additive under the connected sum operation: $g_3(K_1\#K_2)=g_3(K_1)+g_3(K_2)$. This fact is important; it is used, for example, in the prime decomposition theorem for classical knots. There is also a connect sum operation for virtual knots, but the virtual knot type of the sum depends on the choice of arcs involved (see Kauffman-Manturov \cite{kauffman_manturov}). For instance, any of the 2-Kishino knots from Section \ref{sec_kishino} is a non-trivial connected sum of trivial knots. It is natural to ask if the virtual $3$-genus is additive. That is, if $K_1$, $K_2$, and $K_1\#K_2$ are all almost classical is $\widecheck{g}_3(K_1\#K_2)=\widecheck{g}_3(K_1)+\widecheck{g}_3(K_2)$? We will show this is false using the prismatic $U_q(\mathfrak{gl}(2|1))$ polynomial. 

\begin{figure}[htb]
\begin{tabular}{|cc|} \hline & \\ 
\begin{tabular}{c}
\includegraphics[scale=.5]{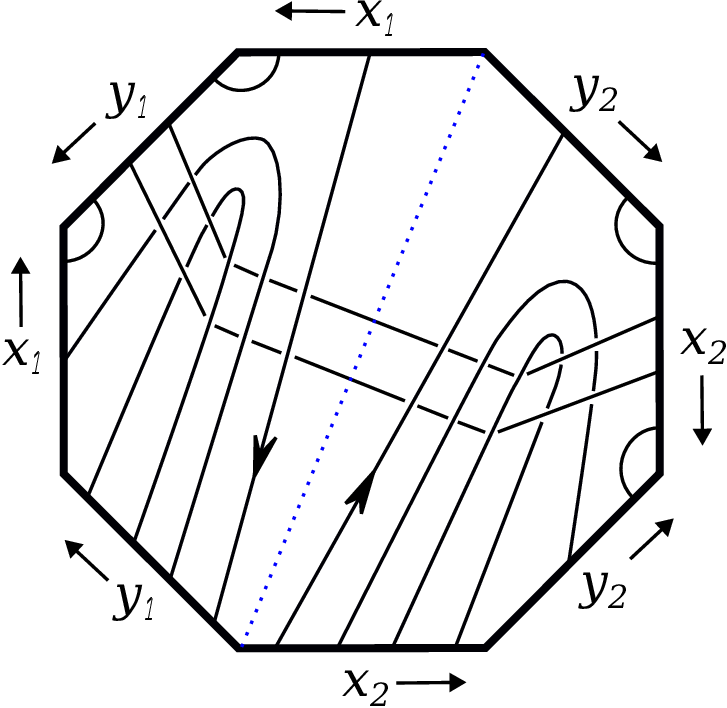} \end{tabular} & \begin{tabular}{c}
\includegraphics[scale=.5]{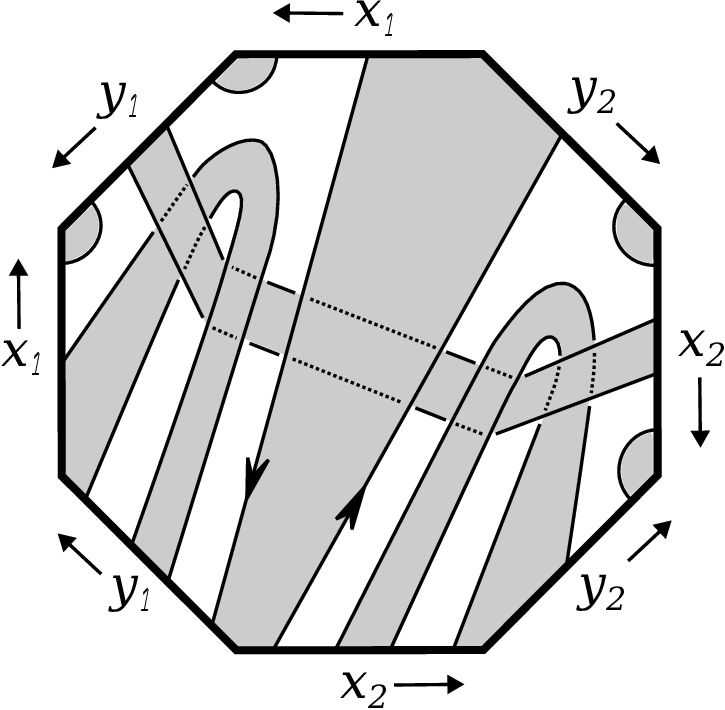}  \end{tabular} \\ & \\ \hline \end{tabular}
\caption{A virtual 3-genus 1 virtual knot that is a connected sum of two unknots.} \label{fig_the_knot} 
\end{figure}

Consider the knot diagram $D$ on a genus 2 surface $\Sigma$ shown on the left in Figure \ref{fig_the_knot}. The dashed blue line indicates two connected sums: a connected sum of two unknots and a connected sum of two tori. This knot is almost classical. A genus one Seifert surface is shaded in grey on the right-hand side of Figure \ref{fig_the_knot}. As an AC knot, it has a well-defined Alexander polynomial. But being a connected sum of trivial knots, its Alexander polynomial is necessarily $1$. Hence, the Alexander polynomial lower bound on the virtual $3$-genus (see \cite{BGHNW_17}, Theorem  7.9) is 0. 

Fortunately, to prove that $\widecheck{g}_3(D)=1$, it suffices to show that $\widecheck{g}_2(D)>0$. Indeed, the only virtual knot having virtual $3$-genus $0$ is the unknot. We will show that the virtual $2$-genus is exactly 2.  Observe that $D$, as connect sum of two trivial knots, is also a slice virtual knot. Most easily computable invariants are automatically trivial on slice almost classical knots that are the sum of unknots. The graded genus, the various index polynomials, the generalized Alexander polynomial, the Jones polynomial, and the arrow polynomial are all trivial. Since $D$ is checkerboard colorable, the determinant test applies (see Section \ref{sec_check}). However, using the program to compute Goeritz matrices from \cite{BCK}, we obtain the following.
\[
\left(
\begin{array}{cccccccc}
 2 & 0 & 0 & 0 & -1 & -1 & 0 & 0 \\
 0 & 0 & 0 & 0 & 0 & 0 & 1 & 0 \\
 0 & 0 & 0 & 0 & 0 & 0 & 1 & 0 \\
 0 & 0 & 0 & 0 & 0 & 0 & 0 & 0 \\
 -1 & 0 & 0 & 0 & 0 & 0 & 0 & 1 \\
 -1 & 0 & 0 & 0 & 0 & 0 & 0 & 1 \\
 0 & 1 & 1 & 0 & 0 & 0 & -2 & 0 \\
 0 & 0 & 0 & 0 & 1 & 1 & 0 & -2 \\
\end{array}
\right),\quad \left(
\begin{array}{cccccccc}
 0 & -2 & 1 & 1 & 0 & 0 & 0 & 0 \\
 -2 & 0 & 1 & 0 & 0 & 0 & 0 & 0 \\
 1 & 1 & 0 & -1 & 0 & 0 & 0 & 0 \\
 1 & 0 & -1 & 0 & 0 & 0 & 0 & 0 \\
 0 & 0 & 0 & 0 & 0 & -2 & 1 & 1 \\
 0 & 0 & 0 & 0 & -2 & 0 & 1 & 0 \\
 0 & 0 & 0 & 0 & 1 & 1 & 0 & -1 \\
 0 & 0 & 0 & 0 & 1 & 0 & -1 & 0 \\
\end{array}
\right)
\]
The first matrix is singular, and hence, the determinant test fails. The high classical crossing number of 20 makes it difficult to compute the surface bracket. Likewise, testing for tg-hyperbolicity is impractical for high crossing number knots on surfaces of genus greater one. The CSW polynomial can also be computed directly from the depicted diagram, but the value turns out to be $\Delta^0_K(t,x_i,y_i)=0$. This value was computed by hand by the authors using the definition of the CSW polynomial , but it also follows from calculation of the prismatic $U_q(\mathfrak{gl}(1|1))$ polynomial below.

\begin{figure}[htb]
\begin{tabular}{|cccc|} \hline & & & \\  \multicolumn{4}{|c|}{\begin{tabular}{c} \includegraphics[scale=.8]{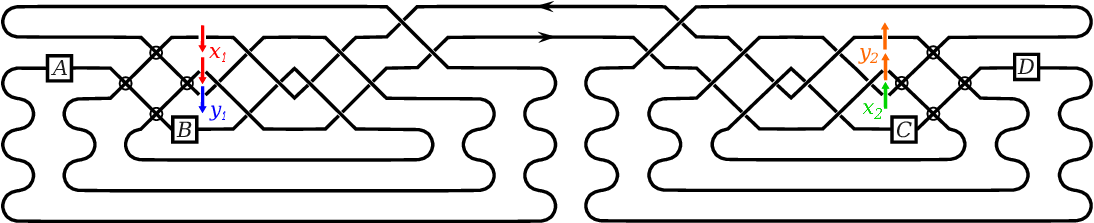} \end{tabular}} \\ & & & \\ \hline & & & \\ \multicolumn{4}{|c|}{\begin{tabular}{cccc} \begin{tabular}{c} \includegraphics[scale=.45]{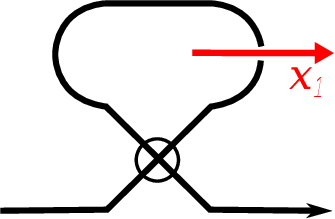}\\ $\underline{A}:$ \end{tabular} & \begin{tabular}{c} \includegraphics[scale=.45]{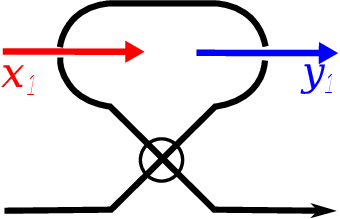} \\ $\underline{B}:$ \end{tabular} & \begin{tabular}{c} \includegraphics[scale=.45]{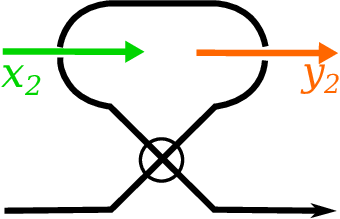} \\ $\underline{C}:$\end{tabular} & \begin{tabular}{c} \includegraphics[scale=.45]{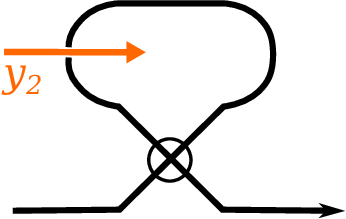}\\ $\underline{D}:$\end{tabular} \end{tabular}} \\ & & & \\ \hline
\end{tabular}
\caption{Homotopy $\Zh$-construction ($\circlearrowleft$) for the knot in Figure \ref{fig_the_knot}. To speed up the calculation, the functor is first computed on tangles A,B,C,D. These are inserted as needed into the final calculation. See accompanying Mathematic notebook.} \label{fig_ks_knot_rot}
\end{figure}    

Thus we are left with the prismatic $U_q(\mathfrak{gl}(m|n))$ polynomials with $(m,n)\ne (1,1)$. A homotopy $\Zh$-construction for $D$ is given in Figure \ref{fig_ks_knot_rot}. To reduce the width of the virtual tangle diagram, the value of the virtual curls $A,B,C,D$ that appear in the figure were precomputed and then inserted when required into the calculation of the Reshetikhin-Turaev functor. The values for $U_q(\mathfrak{gl}(1|1))$ and  $U_q(\mathfrak{gl}(2|1))$ are given below. As the value for $U_q(\mathfrak{gl}(2|1))$ has 698 terms when expanded, only a few terms have been included. All terms can be found in the accompanying \emph{Mathematica} notebook.  
\begin{align*}
\widetilde{f}^{\,\,1|1}_{(\Sigma,\Omega,D)}(q,x_1,y_1,x_2,y_2) &= 0 \\
\widetilde{f}^{\,\,2|1}_{(\Sigma,\Omega,D)}(q,x_1,y_1,x_2,y_2)  &=\frac{2}{q^{16}}-\frac{6}{q^{14}}+\frac{8}{q^{12}}-\frac{14}{q^{10}}+\frac{22}{q^8}-\frac{22}{q^6}+\frac{25}{q^4}+25 q^2-\frac{15}{q^2}-14+ \ldots \\
& \ldots -\frac{x_2}{q^{13}}+\frac{4 x_2}{q^{11}}-\frac{8 x_2}{q^9}+\frac{13 x_2}{q^7}-28 q^5 x_2-\frac{19 x_2}{q^5}+2 q^3 x_2+\frac{24 x_2}{q^3}+\ldots \\
&\ldots -\frac{x_2}{q^{14} y_2}+\frac{6 x_2}{q^{12} y_2}-\frac{x_2}{q^{11} y_2}-\frac{17 x_2}{q^{10} y_2}+\frac{5 x_2}{q^9 y_2}+\frac{34 x_2}{q^8 y_2}-\frac{12 x_2}{q^7 y_2}-\frac{59 x_2}{q^6 y_2}+\ldots \\
&\ldots \frac{x_2}{q^{10} x_1}+\frac{15 q^8 x_2}{x_1}-\frac{5 x_2}{q^8 x_1}-\frac{20 q^6 x_2}{x_1}+\frac{11 x_2}{q^6 x_1}+\frac{14 q^4 x_2}{x_1}-\frac{15 x_2}{q^4 x_1}-\frac{q^2 x_2}{x_1}+\ldots \\
& \ldots+\frac{29 x_2 y_1}{q^{10} y_2 x_1}+\frac{49 x_2 y_1}{q^9 y_2 x_1}-\frac{68 x_2 y_1}{q^8 y_2 x_1}-\frac{91 x_2 y_1}{q^7 y_2 x_1}+\frac{126 x_2 y_1}{q^6 y_2 x_1}-\frac{255 x_2 y_1}{y_2 x_1}+\ldots
\end{align*}

From the included terms for the prismatic $U_q(\mathfrak{gl}(2|1))$ polynomial, it follows that the symplectic rank is $4$ and the virtual 2-genus is $2$. Thus $D$ is non-trivial and the virtual $3$-genus is not additive.

\bibliographystyle{amsplain}

\bibliography{0_super_bib}

\providecommand{\bysame}{\leavevmode\hbox to3em{\hrulefill}\thinspace}
\providecommand{\MR}{\relax\ifhmode\unskip\space\fi MR }
\providecommand{\MRhref}[2]{%
  \href{http://www.ams.org/mathscinet-getitem?mr=#1}{#2}
}
\providecommand{\href}[2]{#2}
\begin{thebibliography}{10}

\bibitem{adams_tg}
C.~Adams, O.~Eisenberg, J.~Greenberg, K.~Kapoor, Z.~Liang, K.~O'Connor,
  N.~Pacheco-Tallaj, and Y.~Wang, \emph{{T}g-hyperbolicity of virtual links},
  J. Knot Theory Ramifications \textbf{28} (2019), no.~12, 1950080, 26.
  \MR{4059928}

\bibitem{adams_simons}
C.~Adams and A.~Simons, \emph{T{G}-hyperbolicity of composition of virtual
  knots}, Comm. Anal. Geom. \textbf{32} (2024), no.~10, 2617--2671.
  \MR{4844960}

\bibitem{bar_natan_talk}
D.~Bar-Natan, \emph{Crossing the crossings},
  \url{http://drorbn.net/AcademicPensieve/2015-11/xtx/xtx.pdf} (2015).

\bibitem{bardakov_bellingeri}
V.~G. Bardakov and P.~Bellingeri, \emph{Groups of virtual and welded links}, J.
  Knot Theory Ramifications \textbf{23} (2014), no.~3, 1450014, 23.
  \MR{3200494}

\bibitem{boden_chrisman_21}
H.~U. Boden and M.~Chrisman, \emph{Virtual concordance and the generalized
  {A}lexander polynomial}, J. Knot Theory Ramifications \textbf{30} (2021),
  no.~5, Paper No. 2150030, 35. \MR{4284604}

\bibitem{BCK}
H.~U. Boden, M.~Chrisman, and H.~Karimi, \emph{The {G}ordon-{L}itherland
  pairing for links in thickened surfaces}, Internat. J. Math. \textbf{33}
  (2022), no.~10-11, Paper No. 2250078, 47. \MR{4514298}

\bibitem{BDGGHN_15}
H.~U. Boden, E.~Dies, A.~I. Gaudreau, A.~Gerlings, E.~Harper, and A.~J. Nicas,
  \emph{Alexander invariants for virtual knots}, J. Knot Theory Ramifications
  \textbf{24} (2015), no.~3, 1550009, 62. \MR{3342135}

\bibitem{BGHNW_17}
H.~U. Boden, E.~Gaudreau, R.and~Harper, A.~J. Nicas, and L.~White,
  \emph{Virtual knot groups and almost classical knots}, Fund. Math.
  \textbf{238} (2017), no.~2, 101--142. \MR{3640614}

\bibitem{boden_karimi_sikora}
H.~U. Boden, H.~Karimi, and A.~S. Sikora, \emph{Adequate links in thickened
  surfaces and the generalized {T}ait conjectures}, Algebr. Geom. Topol.
  \textbf{23} (2023), no.~5, 2271--2308. \MR{4621431}

\bibitem{boden_rushworth}
H.~U. Boden and W.~Rushworth, \emph{Minimal crossing number implies minimal
  supporting genus}, Bull. Lond. Math. Soc. \textbf{53} (2021), no.~4,
  1174--1184. \MR{4311827}

\bibitem{boninger}
J.~Boninger, \emph{A quantum invariant of links in {$T^2\times I$} with volume
  conjecture behavior}, Algebr. Geom. Topol. \textbf{23} (2023), no.~4,
  1891--1934. \MR{4602416}

\bibitem{bz}
G.~Burde and H.~Zieschang, \emph{Knots}, second ed., De Gruyter Studies in
  Mathematics, vol.~5, Walter de Gruyter \& Co., Berlin, 2003. \MR{1959408}

\bibitem{CSW}
J.~S. Carter, D.~S. Silver, and S.~G. Williams, \emph{Invariants of links in
  thickened surfaces}, Algebr. Geom. Topol. \textbf{14} (2014), no.~3,
  1377--1394. \MR{3190597}

\bibitem{CP}
M.~Chrisman and A.~Poudel, \emph{Lie superalgebra invariants and almost
  classical knots},
  \href{https://arxiv.org/abs/2405.07375}{arXiv:2405.07375[math.GT]}, 2024.

\bibitem{chrisman_todd_23}
M.~Chrisman and R.~G. Todd, \emph{Characterization and further applications of
  the {B}ar-{N}atan \zh-construction}, Osaka J. Math. \textbf{62} (2025),
  no.~3, 351--379. \MR{4938000}

\bibitem{crans_henrich_nelson}
A.~S. Crans, A.~Henrich, and S.~Nelson, \emph{Polynomial knot and link
  invariants from the virtual biquandle}, J. Knot Theory Ramifications
  \textbf{22} (2013), no.~4, 134004, 15. \MR{3055555}

\bibitem{de_wit_ishii_links}
D.~De~Wit, A.~Ishii, and J.~Links, \emph{Infinitely many two-variable
  generalisations of the {A}lexander-{C}onway polynomial}, Algebr. Geom. Topol.
  \textbf{5} (2005), 405--418. \MR{2153122}

\bibitem{dye_kauffman_arrow}
H.~A. Dye and L.~H. Kauffman, \emph{Virtual crossing number and the arrow
  polynomial}, J. Knot Theory Ramifications \textbf{18} (2009), no.~10,
  1335--1357. \MR{2583800}

\bibitem{dye_kauffman_surf}
H.~A. Dye and Louis~H. Kauffman, \emph{Minimal surface representations of
  virtual knots and links}, Algebr. Geom. Topol. \textbf{5} (2005), 509--535.
  \MR{2153118}

\bibitem{green}
J.~Green, \emph{A table of virtual knots},
  \url{http://www.math.toronto.edu/drorbn/Students/GreenJ} (2004).

\bibitem{mathematica}
Wolfram~Research{,} Inc., \emph{Mathematica, {V}ersion 14.0}, Champaign, IL,
  2024.

\bibitem{jackson_moffatt}
D.~M. Jackson and I.~Moffatt, \emph{An introduction to quantum and {V}assiliev
  knot invariants}, CMS Books in Mathematics/Ouvrages de Math\'{e}matiques de
  la SMC, Springer, Cham, 2019. \MR{3931694}

\bibitem{jaeger_kauffman_saleur_94}
F.~Jaeger, L.~H. Kauffman, and H.~Saleur, \emph{The {C}onway polynomial in
  {${\bf R}^3$} and in thickened surfaces: a new determinant formulation}, J.
  Combin. Theory Ser. B \textbf{61} (1994), no.~2, 237--259. \MR{1280610}

\bibitem{kamada_v_braid}
S.~Kamada, \emph{Braid presentation of virtual knots and welded knots}, Osaka
  J. Math. \textbf{44} (2007), no.~2, 441--458. \MR{2351010}

\bibitem{kauffman_vkt}
L.~H. Kauffman, \emph{Virtual knot theory}, European J. Combin. \textbf{20}
  (1999), no.~7, 663--690. \MR{1721925}

\bibitem{kauffman_cobordism}
\bysame, \emph{Virtual knot cobordism}, New ideas in low dimensional topology,
  Ser. Knots Everything, vol.~56, World Sci. Publ., Hackensack, NJ, 2015,
  pp.~335--377. \MR{3381329}

\bibitem{kauffman_radford_00}
L.~H. Kauffman and D.~Radford, \emph{Bi-oriented quantum algebras, and a
  generalized {A}lexander polynomial for virtual links}, Diagrammatic morphisms
  and applications ({S}an {F}rancisco, {CA}, 2000), Contemp. Math., vol. 318,
  Amer. Math. Soc., Providence, RI, 2003, pp.~113--140. \MR{1973514}

\bibitem{kauffman_saleur_91}
L.~H. Kauffman and H.~Saleur, \emph{Free fermions and the {A}lexander-{C}onway
  polynomial}, Comm. Math. Phys. \textbf{141} (1991), no.~2, 293--327.
  \MR{1133269}

\bibitem{kauffman_saleur_92}
\bysame, \emph{Fermions and link invariants}, Infinite analysis, {P}art {A},
  {B} ({K}yoto, 1991), Adv. Ser. Math. Phys., vol.~16, World Sci. Publ., River
  Edge, NJ, 1992, pp.~493--532. \MR{1187562}

\bibitem{kauffman_manturov}
L.~H. Kaufman and V.~O. Manturov, \emph{Virtual knots and links}, Tr. Mat.
  Inst. Steklova \textbf{252} (2006), no.~Geom. Topol., Diskret. Geom. i Teor.
  Mnozh., 114--133. \MR{2255973}

\bibitem{kohli_patureau_mirand}
B.-M. Kohli and B.~Patureau-Mirand, \emph{Other quantum relatives of the
  {A}lexander polynomial through the {L}inks-{G}ould invariants}, Proc. Amer.
  Math. Soc. \textbf{145} (2017), no.~12, 5419--5433. \MR{3717968}

\bibitem{kohli_tahar}
B.-M. Kohli and G.~Tahar, \emph{A lower bound for the genus of a knot using the
  {L}inks-{G}ould invariant}, 2024.

\bibitem{kreinbihl}
J.~Kreinbihl, \emph{A {F}ox-{M}ilnor theorem for knots in a thickened surface},
  J. Knot Theory Ramifications \textbf{28} (2019), no.~12, 1950073, 15.
  \MR{4059921}

\bibitem{kuperberg}
G.~Kuperberg, \emph{What is a virtual link?}, Algebr. Geom. Topol. \textbf{3}
  (2003), 587--591. \MR{1997331}

\bibitem{lopez_neumann_van_der_veen}
D.~Lopez-Neumann and R.~van~der Veen, \emph{A plumbing-multiplicative function
  from the {L}inks-{G}ould invariant}, 2025.

\bibitem{manturov_02}
V.~O. Manturov, \emph{On invariants of virtual links}, Acta Appl. Math.
  \textbf{72} (2002), no.~3, 295--309. \MR{1916950}

\bibitem{manturov_surf}
\bysame, \emph{Kauffman-like polynomial and curves in 2-surfaces}, J. Knot
  Theory Ramifications \textbf{12} (2003), no.~8, 1145--1153. \MR{2017986}

\bibitem{manturov_projection}
\bysame, \emph{Parity and projection from virtual knots to classical knots}, J.
  Knot Theory Ramifications \textbf{22} (2013), no.~9, 1350044, 20.
  \MR{3105303}

\bibitem{mellor_16}
B.~Mellor, \emph{Alexander and writhe polynomials for virtual knots}, J. Knot
  Theory Ramifications \textbf{25} (2016), no.~8, 1650050, 30. \MR{3530309}

\bibitem{miller}
K.~A. Miller, \emph{The homological arrow polynomial for virtual links}, J.
  Knot Theory Ramifications \textbf{32} (2023), no.~1, Paper No. 2350005, 42.
  \MR{4557486}

\bibitem{miyazawa}
Y.~Miyazawa, \emph{A multi-variable polynomial invariant for virtual knots and
  links}, J. Knot Theory Ramifications \textbf{17} (2008), no.~11, 1311--1326.
  \MR{2469206}

\bibitem{lopez_neumann_van_der_veen_fibred}
D.~López Neumann and R.~van~der Veen, \emph{Non-semisimple $\mathfrak{sl}_2$
  quantum invariants of fibred links}, 2024.

\bibitem{queffelec_19}
H.~Queffelec, \emph{Polynomial link invariants and quantum algebras}, Winter
  Braids Lect. Notes \textbf{6} (2019), no.~Winter Braids IX (Reims, 2019),
  Exp. No. 4, 20. \MR{4374307}

\bibitem{queffelec_sartori}
H.~Queffelec and A.~Sartori, \emph{Mixed quantum skew {H}owe duality and link
  invariants of type {$A$}}, J. Pure Appl. Algebra \textbf{223} (2019), no.~7,
  2733--2779. \MR{3912946}

\bibitem{robinson}
D.~J.~S. Robinson, \emph{A course in the theory of groups}, second ed.,
  Graduate Texts in Mathematics, vol.~80, Springer-Verlag, New York, 1996.
  \MR{1357169}

\bibitem{rolfsen}
D.~Rolfsen, \emph{Knots and links}, Mathematics Lecture Series, vol.~7, Publish
  or Perish, Inc., Houston, TX, 1990, Corrected reprint of the 1976 original.
  \MR{1277811}

\bibitem{sawollek_01}
J.~Sawollek, \emph{On alexander-conway polynomials for virtual knots and
  links}, \href{https://arxiv.org/pdf/math/9912173}{math/9912173}, 2001.

\bibitem{silver_williams_03}
D.~S. Silver and S.~G. Williams, \emph{Polynomial invariants of virtual links},
  J. Knot Theory Ramifications \textbf{12} (2003), no.~7, 987--1000.
  \MR{2017967}

\bibitem{silver_williams_06_II}
\bysame, \emph{Crowell's derived group and twisted polynomials}, J. Knot Theory
  Ramifications \textbf{15} (2006), no.~8, 1079--1094. \MR{2275098}

\bibitem{silver_williams_sat}
\bysame, \emph{Virtual genus of satellite links}, J. Knot Theory Ramifications
  \textbf{22} (2013), no.~3, 1350008, 4. \MR{3048509}

\bibitem{zhang_2}
R.~B. Zhang, \emph{Structure and representations of the quantum general linear
  supergroup}, Comm. Math. Phys. \textbf{195} (1998), no.~3, 525--547.
  \MR{1640999}

\bibitem{zhang}
\bysame, \emph{Quantum enveloping superalgebras and link invariants}, J. Math.
  Phys. \textbf{43} (2002), no.~4, 2029--2048. \MR{1892766}

\end{thebibliography}

\end{document}